\documentclass[11pt,a4paper,oneside,reqno,notitlepage]{amsart}
\usepackage{amssymb}

\usepackage{amsfonts}
\usepackage{amscd,amsmath}
\usepackage{euscript}
\usepackage[T2A]{fontenc}
\usepackage{comment}
\usepackage[arrow,matrix,curve]{xy}\SilentMatrices
\def\xyma{\xymatrix@M.7em}
\usepackage{float}

\def\lotimes{\buildrel{L}\over\otimes}

\numberwithin{equation}{section}

\newtheorem{cor}{Corollary}[section]
\newtheorem{defi}{Definition}[section]
\newtheorem{prop}{Proposition}[section]
\newtheorem{theorem}{Theorem}[section]
\newtheorem{lemma}{Lemma}[section]
\newtheorem{remark}{Remark}[section]
\newtheorem{example}{Example}[section]

\def\ga{\gamma}

\def\Ga{\Gamma}

\def\Lam{\Lambda}
\def\lam{\lambda}

\def\om{\omega}
\def\Om{\Omega}

\def\Le{\EuScript L}

\def\Z{{\mathbb{Z}}}

\def\la{\longrightarrow}
\def\ot{\otimes}

\def\lot{\stackrel{L}{\ot}}
\def\bee{\begin{equation}}
\def\ee{\end{equation}}
\def\noi{\noindent}
\def\tor{\mathrm{Tor}}
\def\Tor{\mathrm{Tor}}

\begin{document}

\title{Derived functors of non-additive functors and homotopy theory}
\author{Lawrence Breen and Roman Mikhailov}

\address{LB: Universit\'e Paris 13 \newline
\indent Laboratoire CNRS LAGA \newline
\indent 99, avenue Jean-Baptiste Cl\'ement \newline
\indent  93430  Villetaneuse\newline
\indent  France}
 \email{breen@math.univ-paris13.fr}
\address{RM: Steklov Mathematical Institute \newline
\indent Department of Algebra \newline
\indent Gubkina 8 \newline
 \indent Moscow 119991 \newline
\indent Russia}
\email{romanvm@mi.ras.ru}

\maketitle

\section{Introduction}
\vspace{.5cm}

A great number of  methods for  computing  the homotopy or the
homology  of a topological space
begin with a mod $p$ reduction, and this has proved to be very
efficient even though one then has to deal with an extension problem
when reverting  to integer coefficients.   However, such methods are not
well-suited when one considers spaces which are of an
algebraic nature, such as  Eilenberg-Mac Lane spaces.  That a purely functorial
approach  is possible in such  a case  was already apparent in the
classical paper of  S. Eilenberg and S. Mac Lane \cite{EM},  in which
they calculate directly
the integral homology  of Eilenberg-Mac Lane spaces in low degrees.
  Their results are expressed in terms
of what they called the ``new and  quite bizarre functors'' $\Om
(\Pi)$ and $R (\Pi)$. These functors became more intelligible with
the advent of the Dold-Puppe theory of derived functors of
non-additive functors \cite{DoldPuppe}, as they could then be
 interpreted as left-derived functors
 of the second  exterior power functor  and the second divided power
functor respectively. Higher analogues of these new functors subsequently
appeared in related contexts in a number a places, particularly in the Ph.D.
theses of Mac Lane's students R. Hamsher \cite{Hamsher}   and
G. Decker \cite{Decker}. However, this line of research was not vigorously
pursued, though one should mention  in this context  the work of   H. Baues
\cite{Bau} and of the first author \cite{Breen}, as well as   the
unpublished  preprint of  A. K. Bousfield  \cite{Bou}.

\bigskip

 In the present text, we compute in this functorial
 spirit certain  unstable homotopy
groups of  Moore spaces $M(A,n)$ and in particular of the corresponding
spheres $S^n= M(\Z,n)$ . This approach to the computation of the  homotopy groups of
spheres
is of
particular interest, since much   more structure is revealed when
these homotopy
groups  are described as special values  at the group $A= \Z$  of a
certain  functor.  Our
method is in some sense quite classical, since it relies on D.  Kan's
construction of the loop group $GK$ of a connected simplicial set $K$
and on
E. Curtis' spectral sequence determined by the lower central series
filtration of $GK$.  The initial terms in this spectral sequence were
described by Curtis in terms of the derived functors of the Lie
functors $\Le^n$ \cite{Curtis:65}. In addition, Curtis showed that
these Lie functors are endowed with a natural
filtration whose associated graded components are built up from more familiar functors.

\bigskip

It follows from this description that a key ingredient in such an
approach must be a good understanding of the derived functors of the
functor $\Le^n$. We are  able to achieve this  in low
degrees, where this is  made possible by the fact that this Curtis decomposition of
 the Lie functors reduces this problem to the computation of derived functors
 of iterates of certain elementary functors (particularly the degree
 $r$ symmetric  functor $SP^r$, and the related $r$th exterior
 algebra and  $r$th divided power functors  $\Lam^r$ and   $\Gamma_r$).
In order to deal with such  iterates, we require a composite
functor spectral sequence along the lines of the standard Grothendieck
spectral sequence \cite{weibel}, but for a pair of composable non-additive
functors such as those mentioned above. Such a non-additive composite functor
spectral sequence was   defined by D. Blanc and C. Stover
\cite{B-S}, but here we give a  formulation of its
initial terms which is  better suited to our computational objectives. In fact
this spectral sequence degenerates
in our context at $E^2$, and may rather be thought of, for the functors which
we consider, as a symmetrization of the K\"unneth formula and its
higher analogues \`a la Mac Lane \cite{Mac}, \cite{Mac1}. In this way
we are able to go beyond the computation of the iterates of
$\Lambda^2$  already  considered in this context  by the second author  \cite{MPBook}.

\bigskip

In our quest for the explicit values of the Dold-Puppe derived
functors of the Lie algebra functors $\Le^r$ for certain values of
$r$,
 we  deal with a number of questions of independent interest.
First of all, starting from the description by F. Jean, a student of
the first author, of the derived functors of $SP^r$ and
 $\Lambda^r$ \cite{Jean}, we give a complete
description  of these derived functors (as well as of the
 divided power functor $\Ga_r$)  for $r=2$,  by a
 method different  from  that of H. Baues
and T. Pirashvili in \cite{BauesPirashvili}. We then go on to give a
similarly complete and functorial  description of the corresponding
derived functors of $SP^3,\, \Lam^3$ and $\Ga_3$, and we deduce from this  a functorial  description of  the
derived functors $L_i\Le^3(A,n)$ for all $i$ and $n$. We also compute
certain derived functors of the quartic composite functors
$\Lambda^2\Lambda^2$ and $\Lambda^2\Gamma_2$ and deduce from them certain
values of the derived functors $L_i\Le^4(A,n)$.

 \bigskip

 In order to
achieve a sufficiently precise understanding of some of these derived
functors of $\Le^4$, we were led to introduce  an analogue for Lie functors of
the d\'ecalage morphisms. The latter are determined by the existence of Koszul
complexes,   which relate to each other  the derived functors of
$SP^r$, $\Lambda^r$ and $\Ga_r$. Just as $\Lam^r$ may be viewed as a
super-analogue\footnote{We prefer to call this  the super-analogue,
  rather than the   graded analogue as is more customary, since all our functors are  graded.} of $SP^r$, we need to introduce a  super-analogue
$\Le^r_s$ of the  Lie  functor $\Le^r$. It turns out that this must
not be  the
naive graded commutative version of the Lie functor, in which certain
signs are  changed in the  relations defining it. Instead, it is
necessary to introduce in its definition, following D. Leibowitz in her unpublished
thesis \cite{Leibowitz},  an additional divided square operation with respect
to the Lie super-bracket of odd degree elements.   While there no longer
exist   d\'ecalage isomorphisms between the Lie and super-Lie  functors,
there do exist   canonical pension maps between them, which we call the
semi-d\'ecalage morphisms, and which allows us to give a refined
description of their derived functors in certain cases.

\bigskip

We also rely at some point on the knowledge of the homology of the
 complex $C^n(A)$ dual to the de Rham complex first introduced in the
 present context by V. Franjou, J. Lannes and L. Schwartz in
 \cite{FLS}. We refer to \cite{roman} for  an explicit
 calculation of the values of the homology groups $H_0C^n(A)$ for all $n$
 announced by Jean
 in \cite{Jean}, as well as for  a description of  all the
 homology groups $H_iC^n(A)$ when  $n < 8$. The occurence of 8-torsion
 when $n= 8$, as well as that of a Lie functor when $n=6$, suggests
 that no simple description of these groups can be expected for a
 general $n$.

\bigskip

 In  \S \ref{main}, we use
these tools in order to achieve our goal of computing  algebraically
certain  homotopy groups  of
$n$-spheres and   Moore spaces $M(A,n)$. The task at hand is
twofold. The first part consists, as we have said, in computing the
initial terms of the  Curtis spectral sequence, and for this we rely
  on our knowledge of the
derived functors of certain Lie functors and their super-analogues.
The second  part  consists in understanding certain
 differentials in the spectral sequence. We rely here upon  various methods, some based  on the
 functoriality  of our construction and on the fact that the
 differentials are now natural transformations rather then simply
 group homomorphisms, and others more classical in spirit (the
 suspension of a Moore space, the comparison of a Moore space with the
 corresponding Eilenberg-Mac Lane space, ...).  We have at times in
 this final section made use of known results concerning the homotopy
 of Moore spaces for specific groups $A$, whenever this allowed us to
 progress with our own investigations.  It is quite striking to observe
 how far one can go in the description of these homotopy groups,  with only the knowledge of derived functors of
 quadratic and cubical functors as  the basic  input.

\bigskip

In this final section we proceed in logical order, beginning with
the homotopy groups of $S^2$ and $M(A,2)$  and then moving on to $M(A,n)$ for
increasing values of $n$.  For $A= \Z$, some of our results follow in
the stable range from \cite{6A} where a more efficitient spectral
sequence, which however depends on the choice of a fixed prime $p$,
 is considered (see also \cite{Curtis:71}). As an
illustration of our methods, we begin by explaining  how one may  obtain in this
way the known values of the 3-torsion of  $\pi_i(S^2)$, for values of $i$ up to
$14$. To have gone further, so as to retrieve the classical results
of H.Toda \cite{To}  up to $i= 22$,
would have obliged us to delve further into the  analysis of  the spectral sequence. We
also recover, and reinterpret, some  results of Baues and his
collaborators \cite{Bau},  \cite{Dre}, \cite{BB}.
  In particular, we obtain by our methods
 the results of Baues and Buth \cite{BB}  concerning $\pi_i(M(A,2))$ for $i=4$,  and
  give an improved description of the unpublished results of
 Dreckmann  for $i=5$ \cite{Dre} (see also Baues and Goerss \cite{BG}). Specializing to the case
  $A = \Z/p$, we obtain further information about  the groups
  $\pi_i(M(\Z/p,2))$  whenever  the prime $p$ is odd. By suspending
  these calculations, this gives us in particular  a fully functorial
  description of the  graded components associated to  a  natural filtration
  of  the group $\pi_6(M(A,3))$ which appears to be new. As a consequence of these computations, we can
  recover the value of $\pi_5(M(\Z/3,2))$,  a significant case of  the
  extension by   Neisendorfer \cite{Neisendorfer} to the prime   $p=3$
 of    Cohen, Moore and Neisendorfer's study  \cite{CohenNeisendorfer}
 of  the homotopy of Moore spaces (this value had in fact  previously been
 obtained by D. Leibowitz in  \cite{Leibowitz}).

\bigskip

As a final example, we   examine  the low degree homotopy groups of
$M(\Z/3,5)$ since this  allows us to exhibit in a simple context  some of
the techniques  on which we relied throughout this section. In fact, the
reader may wish to begin with  this case, before going on to the
more delicate unstable computations
 which precede it.

\bigskip

As will be apparent from this description of our paper, a
number of our results in homotopy theory have already appeared in
one form or another  in the litterature, where they are proved by
very diverse methods. Our aim here is to show that these can  all
be obtained by a uniform method, based solely on functorial
techniques from homological and homotopical algebra with integer
coefficients. We expect that such an approach to these questions
will not only allow one  to compute  specific  additional
homotopy groups, but more importantly will shed some new  light
upon their global structure.

\bigskip

{\it Acknowledgements}: We are indebted to P. Goerss for providing us
with a copy of the dissertation \cite{Dre}, and to
A. K. Bousfield for making available to us the thesis \cite{Leibowitz}
as well as  for  his  comments regarding a first version of the present  text.

\section{Derived functors}

\vspace{.5cm}

\subsection{Graded functors.}   Let $\sf Ab$ be the category of
abelian groups and $A$ an object of $\sf Ab$. For any chain complex $C_i$, we will
henceforth denote by $C[n]$  the chain complex defined by
\[
C[n]_i := C_{i-n} \ \text{for all $i$.} \]  In particular, the chain
complex $A[n]$ is concentrated in degree $n$.
 In addition to the graded  tensor power
functor $\otimes:=\bigoplus_{n\geq 0}\otimes^n$, the symmetric
power functor  $SP:=\bigoplus_{n\geq 0}SP^n$, and the exterior
power functor $\Lambda:=\bigoplus_{n\geq 0}\Lambda^n$ (the quotient
of $\otimes A$ by the ideal generated by elements $x\otimes x$ for
all $ x\in A$),  we will consider over $\Z$ the following somewhat
less well-known  functors:

\medskip

{\it 1. The divided power functor:\ } (see \cite{roby})
$\Gamma_*=\bigoplus_{n\geq 0}\Gamma_n: \sf Ab\to \sf Ab$. The
graded abelian group  $\Gamma_\ast(A)$ is generated by symbols
$\gamma_i(x)$ of degree $i\geq 0$ satisfying the following
relations for all $x,y \in A$: \begin{align*} & 1)\  \gamma_0(x) =
1 \\ & 2)\ \gamma_1(x)=x\\ & 3)\
\gamma_s(x)\gamma_t(x)=\binom{s+t}{s}\gamma_{s+t}(x) \\
& 4)\ \gamma_n(x+y)=\sum_{s+t=n}\gamma_s(x)\gamma_t(y),\ n\geq 1\\
& 5)\ \gamma_n(-x)=(-1)^n\gamma_n(x),\ n\geq 1.
\end{align*}
In particular, the canonical map $ A \simeq \Gamma_1(A)$ is an
isomorphism. The degree 2 component $\Ga_2(A)$ of
 $\Gamma_*(A)$  is the Whitehead functor $\Ga
(A)$. It is  universal  for homogenous
 quadratic maps from $A$ into abelian groups.  The following
 additional relations in  $\Gamma_*(A)$   are consequences of the previous ones:
\begin{align*}
& \gamma_r(nx)=n^r\gamma_r(x),\ n\in \mathbb Z;\\
& r\gamma_r(x)=x\gamma_{r-1}(x);\\
& x^r=r!\gamma_r(x);\\
& \gamma_r(x)y^r=x^r\gamma_r(y).
\end{align*}
In addition, a  direct computations implies that
$$
\Gamma_r(\mathbb Z/n)\simeq \mathbb Z/{n(r,n^\infty)},
$$
where  the extended g.c.d $(r,n^\infty)$ is defined by
 $(r,n^\infty) := \lim_{\,m\to \infty\,}(r,n^m)$.

\vspace{.5cm} \noindent{\it 2. The Lie functor\ }  $\EuScript L:
\sf Ab\to Ab$  (see \cite{Curtis:63}). The tensor algebra
$\otimes A$ is endowed
with  a $\mathbb Z$-Lie algebra structure, for which the  bracket
operation is defined by
$$[a,\,b]=a\otimes b-b\otimes a,\quad  a,b\in \otimes (A).$$
One   defines  $n$-fold brackets inductively by setting
\begin{equation}
\label{def:lie-br}
 [a_1, \ldots, a_n] := [[a_1\ldots, a_{n-1}],a_n]
\end{equation}
We will  denote $\ot A$, viewed as a  $\mathbb
Z$-Lie algebra,  by $\otimes(A)^{Lie}.$  Let  $\EuScript
L(A)=\bigoplus_{n\geq 1}\EuScript L^n(A)$ be the sub-Lie ring of
$\otimes(A)^{Lie}$ generated by $A$. Its  degree 2 and  3 components
are  generated
by the expressions
\begin{equation}
\label{lie-3}a\otimes b-b\otimes a \qquad \text{and} \qquad  a \otimes b \otimes c - b \otimes a \otimes c + c
\otimes a \otimes b - c \otimes b \otimes a
\end{equation}
where $a,b,c \in A$. $\EuScript L(A) $ is called the {\it free Lie ring generated by
the abelian group $A$}. It is  universal for homomorphisms from
$A$ to $\Z$-Lie algebras.
The grading of $\otimes A$ determines  a grading   on $\EuScript
L(A)$, so that we obtain a family  of endofunctors
 on the category of abelian groups:
$$
\EuScript  L^i: {\sf Ab}\to {\sf Ab},\ i\geq 1.
$$
In particular,
\bee
\label{liz}
\Le^i(\Z) = 0
\ee
 for all $i >1$. For any  free group $F$ and $i\geq 1$, one has the natural
Magnus-Witt isomorphism (\cite{Mag}, \cite{Witt})
\begin{equation}\label{magnuswitt}
\gamma_i(F)/\gamma_{i+1}(F)\simeq \EuScript  L^i(F_{ab}).
\end{equation}
 where $\gamma_i(F)$ is the
$i$th term in the lower central series of $F$.

\bigskip

\vspace{.5cm}\noindent{\it 3. The Schur functors}. We will also
consider the Schur functors $$J^n, Y^n, E^n: {\sf Ab}\to {\sf
Ab},\ n\geq 2$$ defined by
\begin{align}
\label{schur}
& J^n(A)=ker\{A\otimes SP^{n-1}(A)\to SP^n(A)\},\ n\geq 2,\\ \notag
& Y^n(A)=ker\{ A\otimes \Lambda^{n-1}(A)\to \Lambda^n(A)\},\ n\geq
2,\\ \notag
& E^n(A)=ker\{A\otimes \Gamma_{n-1}(A)\to \Gamma_n(A)\},\ n\geq
2\,.
\end{align}
In particular,
\begin{equation}
\label{J2} J^2(A) = E^2(A)  \simeq \Lambda^2A \qquad \text{and} \qquad
Y^2(A)  \simeq \Gamma_2(A)
\end{equation}
whenever $A$ is free. The functors $Y^n(A)$ are the  $\Z$-forms of
the  Schur functors $\mathbb{S}_{\lambda}(V)$ associated  to the
partition $\lambda= (2, 1\ldots,1)$ of the set $(n)$ (see
\cite{ful-har} exercise 6.11, \cite{ful} chapter  8 (19)). The functors
$J^n(A)$ and $E^n(A)$ are two distinct $\Z$- forms of the Schur
functors $\mathbb{S}_\mu$ associated to the  partition $\mu =
(n-1,1)$ of $(n)$, which is the conjugate
 partition of $\lambda$.

\bigskip

The functors $J^n$ and their derived functors were considered by
E. Curtis \cite{Curtis:65} and J. Schlesinger
\cite{Schlesinger:66}. Just as the quotient of a Lie ring
$\EuScript L$ by the ideal generated by all brackets is an abelian
Lie ring, one can consider the {\it metabelian} Lie rings. These
are  the quotients of  $\EuScript L$ by the ideal generated by
brackets of the form $[[ \ ], [\ ]]$. The following proposition
asserts that the functors $J^n$, restricted to free abelian groups
$A$, are metabelian Lie functors
\begin{prop}
\cite{Schlesinger:66}\label{slez}\index{Schlesinger, J.} Let $A$
be a free abelian group and $n\geq 2$. The 4-term  sequence
\begin{equation}
\label{slez1} 0\to\EuScript L^n(A)\cap \EuScript L^2\EuScript
L^2(A)\to  \EuScript L^n(A)\buildrel{p_n}\over{\longrightarrow}
A\otimes SP^{n-1}(A)\buildrel{r_n}\over{\longrightarrow}
SP^n(A)\to 0,
\end{equation}
is exact, where $r_n$ is the  multiplication, and the map  $p_n$
is defined by
$$
p_n: [m_1,\dots, m_n]\mapsto m_1\otimes \, m_2\dots
m_n\,-\,m_2\otimes \,m_1m_3,\dots m_n,\ \text{where }\ m_j\in A \,
\:\forall j.
$$ The projection
\begin{equation}
\label{defpn} \xymatrix{\EuScript L^n(A) \ar@{->>}[r] & J^n(A)}
\end{equation} of  $\EuScript L^n(A)$ onto its image in $A \ot
SP^{n-1}A$ will  also be denoted  $p_n$, so that if  we set
\[ \tilde{J}^n(A) := \EuScript L^n(A)\cap \EuScript L^2\EuScript
L^2(A), \]
 the sequence \eqref{slez1} splits into a pair of short exact sequences
\begin{equation}
\label{defjntilde} \xymatrix{ 0 \ar[r] &  \tilde{J}^n(A) \ar[r]
&\EuScript L^n(A) \ar[r]^{p_n}
  &  J^n(A) \ar[r] & 0\,.}
\end{equation}
\begin{equation}
\label{defjn}
 \xymatrix{ 0 \ar[r] & J^n(A) \ar[r] & A \ot SP^{n-1} \ar[r] & SP^n(A)
 \ar[r] & 0}
\end{equation} In particular,
\begin{equation}
\label{simplest} \EuScript  L^2(A)\simeq \Lambda^2(A)
 \qquad  \text{and}
 \qquad  \EuScript   L^3(A)\simeq
J^3(A) \end{equation}
  since $\tilde{J}^n(A)$ is trivial  for $n=2,3$.
\end{prop}

Let us now recall some the natural transfomations between these
functors:

\begin{alignat}{3}
 f_n:&\  SP^n(A)&\to \hspace{1cm} &\otimes^nA  \\
 & a_1 \dots a_n &\mapsto \hspace{1cm}  &\sum_{\sigma \in
\Sigma_n}a_{i_1}\otimes \dots \otimes a_{i_n} \notag
\\&&&\notag
\\ g_n: &\ \Lambda^n(A)&\to \hspace{1cm}  &\otimes^nA
\\
& a_1\wedge\dots \wedge a_n &\mapsto  \hspace{1cm}& \sum_{\sigma\in
\Sigma_n}\text{sign} (\sigma) \, a_{i_1}\,\otimes \dots \otimes
a_{i_n} \notag\\ &&& \notag \\ h_n: &\ \Gamma_n(A)& \to \hspace{1cm}  &\otimes^nA\label{rrref} \\
&  \gamma_{r_1}(a_1)\dots \gamma_{r_k}(a_k)&\mapsto \hspace{1cm}  &
\sum_{(i_1,\dots, i_n)}a_{i_1}\otimes \dots \otimes a_{i_n} \notag
\\&&& \notag
\end{alignat}

In these definitions of $f_n$ and $g_n$,  we have set $i_j:=\sigma
(j)$, whereas   the $(i_1,\dots,i_n)$ in the definition of $h_n$
range over the set of  $n$-tuples of integers for  which $j$
occurs $r_j$ times ($1\leq j\leq k$).
 When $A$ is free abelian, the induced  morphism
$$
h_n: \Gamma_n(A)\to (\otimes^nA)^{\Sigma_n},
$$ from $\Gamma_n(A)$ to the group of
   tensors invariant under the action of the symmetric group  is an isomorphism
for all $n\geq 1.$  By the  universal property of the algebra
$\Gamma_\ast(A)$,   $h_n$ may also be characterized as  the map
determined by the divided power algebra  structure on  $\ot  A: = \oplus_n
(\ot^nA)$, where  the product in this algebra  is  defined by  the
shuffle product,  and  the divided powers are characterized  by the rule $\ga_n(a) :=
a\ot \ldots \ot a  \in \ot^nA$  for all $a \in A$.

\subsection{Derived functors.}
Let $A$ be an abelian group, and $F$  an endofunctor on the
category of abelian groups. Recall that for every  $n\geq 0$ the
derived functor of $F$
 in the sense
of Dold-Puppe  \cite{DoldPuppe} are defined by
$$
L_iF(A,n)=\pi_i(FKP_\ast[n]),\ i\geq 0
$$
where $P_\ast \to A$ is a projective resolution of $A$, and
 $K$ is  the Dold-Kan transform,  inverse to the Moore normalization  functor
\[
N:  \mathrm{Simpl}({\sf Ab}) \to C({\sf Ab})
\]
from simplicial abelian groups to chain complexes \cite{weibel} Def. 8.3.6. We denote by
$LF(A,n)$ the object $FK(P_\ast[n])$ in the homotopy category of
simplicial abelian groups determined by  $FK(P_\ast[n])$, so that
\[ L_iF(A,n) = \pi_i(LF(A,n))\,.\]
  We
set $LF(A) :=LF(A,0)$ and  $L_iF(A):= L_iF(A,0)$ for any $\ i\geq
0$.  When the functor $F$ is additive,  the $L_iF(A) $ are isomorphic by
iterated
suspension to  $L_{i+n}F(A,n)$ for all $n$, and coincide with the
usual  derived functors of $F$.   As  examples of these  constructions, observe that the
simplicial models $LF(L\la M)$  of $LFA$ and  $FK((L\la M) [1])$
of $LF(A,1)$ associated to the two-term flat resolution
\begin{equation}\label{res11} 0\to L \buildrel{f}\over\to M \to A\to 0\end{equation}
 of an abelian group $A$ are respectively of the following form in low
 degrees:
\bee \label{lowdeg0}
\begin{matrix}
F(s_0(L)\oplus s_1(L) \oplus s_1s_0(M))
\end{matrix}
\begin{matrix}\buildrel{\partial_0,\partial_1,\partial_2}\over\longrightarrow\\[-3.5mm]\longrightarrow\\[-3.5mm]\longrightarrow\\[-3.5mm]\longleftarrow\\[-3.5mm]
\longleftarrow
\end{matrix}\
F(L\oplus s_0(M))
\begin{matrix}\buildrel{\partial_0,\partial_1}\over\longrightarrow\\[-3mm]\longrightarrow\\[-3mm]\longleftarrow\end{matrix}\
 F(M)\ee
 where the component $F(M)$ is  in degree zero, and
\bee \label{lowdeg}
\begin{matrix}
F(s_0(L)\oplus s_1(L)\oplus s_2(L)\oplus\\
s_1s_0(M)\oplus s_2s_0(M)\oplus s_2s_1(M))
\end{matrix}
\begin{matrix}\buildrel{\partial_0,\dots,\partial_3}\over\longrightarrow\\[-3.5mm]\longrightarrow\\[-3.5mm]\longrightarrow\\[-3.5mm]\longrightarrow\\[-3.5mm]\longleftarrow\\[-3.5mm]
\longleftarrow\\[-3.5mm]\longleftarrow
\end{matrix}\
F(L\oplus s_1(M)\oplus s_0(M))
\begin{matrix}\buildrel{\partial_0,\partial_1,\partial_2}\over\longrightarrow\\[-3mm]\longrightarrow\\[-3mm]\longrightarrow\\[-3mm]\longleftarrow\\[-3mm]\longleftarrow\end{matrix}\
 F(M)\ee
where
 the component
 $F(M)$ is  in degree 1.
It follows from the definition of homology that
$L_i\mathbb{Z}(A,n) \simeq H_i(K(A,n);\,\mathbb{Z})$ for all $n$,
where $K(A,n)$ is an Eilenberg-MacLane space associated to the
abelian group $A$ .
 \bigskip

\noindent{\it 1. Derived functors of $\otimes^n$} \cite{Mac1}. For
$n\geq 1,$ and abelian groups $A_1,\dots, A_n$, we
define\footnote{In \cite{Mac}, Mac Lane  uses the notation
$Trip(A_1,A_2,A_3)$ for the group $\Tor_1(A_1,A_2,A_3)$ and
$\Tor(A_1,\dots,A_n)$ for $\Tor_{n-1}(A_1,\dots,A_n)$.}
$$
\Tor_i(A_1,\dots, A_n):=H_i\left(A_1\buildrel{L}\over \otimes
\dots \buildrel{L}\over\otimes A_n\right),\ i\geq 0.
$$
where $A \buildrel{L}\over\otimes B$ is the derived tensor product
of the abelian groups $A$ and $B$ in the derived category of
abelian groups, as in \cite{weibel} \S 10.6.
  In particular,  $$\Tor_0(A_1,\dots, A_n)\simeq A_1\otimes
\dots\otimes A_n\,\hspace{.5cm} \text{and}\,\hspace{.5cm}
\Tor_i(A_1,\dots, A_n\,)=\,0 \quad \ i\geq n. $$ One sets
\[\Tor(A_1,A_2):=  \Tor_1(A_1,A_2) \quad \text{and}  \quad \Tor^{[n]}(A) :=
\Tor_{n-1}(A,\ldots,A) \quad  \text{($n$ copies of $A$)}. \] While
computations of such iterated Tor functors for specific abelian
groups $A$ are elementary, an explicit functorial description of
the multi-functors $\Tor_i$ is more delicate. The functorial short
exact sequence
$$
0\to \Tor(A_1,A_2)\otimes A_3\to \Tor_1(A_1,A_2,A_3)\to
\Tor(A_1\otimes A_2,A_3)\to  0\,,
$$
   splits unnaturally   \cite{Mac}, \cite{Mac1}. The
 Eilenberg-Zilber theorem determines  natural isomorphisms
$$
L_i\otimes^n  A\simeq \Tor_i(A,\dots,A),\ i\geq 0.$$ The group
$\Tor^{[n]}(A)$. It is generated by the $n$-linear expressions
$\tau_h(a_1, \ldots a_n)$  (where all $a_i$ live in the subgroup
$ {}_hA$ of elements $a$ of $A$ for which $ha=0  \ (h >0)$,
subject to the so-called slide relations
\begin{equation}
\label{slide} \tau_{hk}(a_1,\ldots, a_i,\ldots a_n) =
\tau_{h}(ka_1, \ldots, ka_{i-1}, a_i, k_{i+1}, \ldots, ka_n)
\end{equation} for all $i$ whenever  $hka_j = 0$ for all $j \neq
i$ and $ha_i=0$.
 The associativity of the derived
tensor product functor implies that there are   canonical
isomorphisms
$$
\Tor^{[n]}(A)\simeq \Tor(\Tor^{[n-1]}(A),A),\ n\geq 2.
$$ The  description of derived functors
$L_i\otimes^n A$ for a general $i$ follows from that of
$\Tor^{[n]}(A)$. For every abelian group $A$, $n\geq 1,\ 1\leq
i\leq n-1,$  the group $L_i\otimes^n(A)$ is by  \cite{Mac1}   the quotient
:
$$
L_i\otimes^n(A)\simeq \Tor^{[i+1]}(A)\otimes
(\otimes^{n-i-1}(A))/Jac_{\otimes},
$$
where $Jac_\otimes$ is  the subgroup  of  generalized Jacobi-type
relations, generated by the elements
$$
\sum_{k=1}^{i+2}(-1)^k\tau_{h}(x_1,\dots, \widehat{x_k},\dots,
x_{i+2})\otimes x_k \otimes x_{i+3}\otimes \dots \otimes x_{n}
$$
for all $x_1, \dots, x_n \in
 A$.

\vspace{.5cm} \noindent{\it 2. Derived functors of $SP^n$.} The
map $\ot^n \la SP^n$  induces  a natural epimorphism \bee
\label{epi} \Tor^{[n]}(A)\to L_{n-1}SP^n(A) \ee which sends the
generators $\tau_h(a_1,\dots,a_n)$ of $ \Tor^{[n]}(A)$ to
generators $\beta_h(a_1,\dots,a_n)$ of
\[ \mathcal{S}_n(A) :=L_{n-1}SP^n(A)\,.\]
The kernel of this map is generated by the elements
$\tau_h(a_1,\dots,a_n)$ with $a_i=a_j$ for some $i\neq j$. It is
shown by Jean  in  \cite{Jean} that
\begin{equation}\label{syder}
L_iSP^{n}(A)\simeq (L_iSP^{i+1}(A)\otimes
SP^{n-(i+1)}(A))/Jac_{SP},
\end{equation}
where $Jac_{SP}$ is the subgroup generated by elements of the form
$$
\sum_{k=1}^{i+2}(-1)^k\beta_h(x_1,\dots, \hat x_k,\dots,
x_{i+2})\otimes x_ky_1\dots y_{n-i-2}.
$$ with $x_i \in {}_hA$ and $y_j \in A$ for all $i,j$. The filtration
of $\Z(A,n)$ by powers of the augmentation ideal determine a
filtration on the homology groups  $H_r(K(A,n))$, whose associated graded pieces
are the  $L_rSP^s(A,n)$ \cite{Breen}.

\vspace{.5cm}\noindent{\it 3. Derived functors of $\Lambda^n$.}
For any abelian group $A$ and $n\geq 1,$ we set
$$
\Omega_n(A):=L_{n-1}\Lambda^n(A).
$$
 Consider the action of the symmetric  group $\Sigma_n$ on
$\Tor^{[n]}(A)$, defined by
$$
\sigma \tau_h(a_1,\dots,
a_n)=\text{sign}(\sigma)\,\tau_h(a_{\sigma(1)},\dots,
a_{\sigma(n)})$$ where $ha_1=\dots=ha_n=0,\ a_i\in A,\ \sigma  \in
\Sigma_n.$ We  denote this action  by  $\Sigma_n^\epsilon$. The
natural transformations $g_n$  induces  functorial isomorphisms
between $\Omega_n(A)$ and the $\Sigma_n^\epsilon$-invariants in
$\Tor^{[n]}(A)$
 \cite{Breen},  \cite{Jean} th. 2.3.3:
$$
\Omega_n(A)\simeq (\Tor^{[n]}(A))^{\Sigma_n^\epsilon}.
$$
In particular, for all $n>0$,
 \begin{equation}
\label{omeg-cycl} \Omega_n(\Z/r)\simeq \Z/r \,.
\end{equation}
 In addition,  the morphisms $\tau_h$ which describe the Tor
 functors now  symmetrize  to homomorphisms
\begin{equation}
\label{deflamh} \lambda^n_h: \Gamma_n(\ _hA)\to \Omega_n(A)
\end{equation}
for  $h\geq 1$ and the group $\Omega_n(A)$ is generated by  the
elements
 \begin{equation}\label{torprod}
\omega_{i_1}^h(x_1)\ast \ldots \ast \omega_{i_j}^h(x_j):=
\lambda_h(\gamma_{i_1}(x_1)\dots \gamma_{i_j}(x_j))
\end{equation}
 with $i_k \geq 1$ for all $k$, and    $\sum_{k}i_k = n$. These
 satisfy relations which may be thought of, as in \cite{Breen}, as  symmetrized versions of the
 slide relations \eqref{slide}.
The following description of the derived functors $L_i\Lambda^n$
is given in \cite{Jean} Theorem 2.3.5:
\begin{equation}
\label{derlam} L_i\Lambda^n(A)\simeq (\Omega_{i+1}(A)\otimes
\Lambda^{n-i-1}(A))/Jac_{\Lambda} \,.\end{equation} Here
$Jac_{\Lambda}$ is the subgroup generated by the  expressions
\[
\sum_{k=1}^j\omega_{i_1}^h(x_1)\ast \dots\ast
\omega^h_{i_k-1}(x_k) \ast \ldots \ast \omega_{i_j}^h(x_j) \otimes
x_k \wedge y_1 \wedge \ldots y_{n-i-2}
\]
for all $h$, with $\sum_{k=1}^j = i+2$. In particular, this
implies that for any finite cyclic group  $A$,
\begin{equation}
\label{LLambda-cycl} L_i\Lambda^n(A) = 0  \qquad  i \neq n-1.
\end{equation}

\noindent{\it 4. Derived functors of $\Gamma_n$.} Not
all  is known about derived functors of the divided power
functors.  For an abelian group $A$,  the  double d\'ecalage isomorphism
(described in  \eqref{ddec} below)  determines a composite   isomorphism
$$
L_1\Gamma_2(A)\simeq  L_5SP^2(A,2) \simeq   H_5(K(A,2),\mathbb{Z}) $$ so that
$L_1\Gamma_2(A)$ is isomorphic to the functor $R(A)$ of
Eilenberg-Mac Lane \cite{EM}  \cite{EM} \S 22, defined as
$$   (\Tor(A,A)\oplus \Gamma_2({}_2A))/S,
$$
where $S$ is the  subgroup generated by elements
\begin{align*}
& \tau_h(x,x),\ x\in\ _hA,\ h\in \mathbb N,\\
& \gamma_2(x+y)-\gamma_2(x)-\gamma_2(y)-\tau_2(x,y),\ x,y\in\ _2A.
\end{align*}
 More generally, we  set
\[ R_n(A) := L_{n-1}\Gamma_n(A)\,, \] so that $ R_2(A) = R(A)$, even though this
is inconsistent with  the notation  in \cite{Decker}. The sequence \bee
\label{exsp2}
 0 \la SP^2(A) \la \Gamma_2(A) \la A \ot \Z/2\,, \ee
 is exact for any abelian group $A$, and  derives to the
short exact sequence
$$
0\to L_1SP^2(A)\to L_1\Gamma_2(A)\to\ \Tor(A,\mathbb Z/2)\to 0.
$$
Analogous short   exact sequences were obtained in \cite{Jean} \S 3.1 for the
functor $\Gamma_3$ :
\begin{align}
\label{l1ga2}
& 0\to L_1SP^3(A)\to L_1\Gamma_3(A)\to (\Tor(A,\mathbb Z/2)\otimes A\otimes \mathbb Z/2)\oplus \Tor(A,\mathbb Z/3)\to 0\\
& 0\to L_2SP^3(A)\to L_2\Gamma_3(A)\to\ \Tor(A,\mathbb
Z/2)\otimes\ \Tor(A,\mathbb Z/2)\to 0 \notag
\end{align}

\subsection{Koszul complexes} (\cite{Quillen},
\cite{illusie} I (4.3.1.3)). Let $f: P\to Q$ be a homomorphism of abelian
groups. For $n\geq 1$ and any $k=0,\dots, n-1$ consider the  maps
$$
\kappa_{k+1}: \Lambda^{k+1}(P)\otimes SP^{n-k-1}(Q)\to
\Lambda^k(P)\otimes SP^{n-k}(Q)
$$
defined, for $p_i \in P$ and $ q_j \in Q$, by:
\[
\kappa_{k+1}: p_1\wedge \dots\wedge p_{k+1}\otimes q_{k+2}\dots
q_n\mapsto  \sum_{i=1}^{k+1}(-1)^{k+1-i}p_1\wedge \dots \wedge
\hat p_i\wedge \dots \wedge p_{k+1}\otimes f(p_i)\,q_{k+2}\dots
q_n.
\]
The  associated Koszul complex is defined by
\begin{equation}\label{koszul1}
Kos_n(f):\ \  0\to \Lambda^n(P)\buildrel{\kappa_n}\over\to
\wedge^{n-1}(P)\otimes Q\buildrel{\kappa_{n-1}}\over\to \dots\to
P\otimes SP^{n-1}(Q)\buildrel{\kappa_1}\over\to SP^n(Q)\to 0\,.
\end{equation}
 Dually, one   defines  maps
$$
 \kappa^{k+1}: \Gamma_{k+1}(P)\otimes \Lambda^{n-k-1}(Q)\to
\Gamma_k(P)\otimes \Lambda^{n-k}(Q),\ k=0,\dots, n-1
$$
 by setting
\begin{multline}
\kappa^{k+1}: \gamma_{r_1}(p_1)\ldots  \gamma_{r_k}(p_k) \otimes
q_1\wedge\ldots \wedge q_{n-k-1}
\mapsto \\
 \sum_{j=1}^k  \gamma_{r_1}(p_1)\ldots \gamma_{r_{j-1}}(p_j)\ldots \gamma_{r_k}(p_k)
\otimes f(p_j) \wedge q_1\wedge \ldots \wedge q_{n-k-1}
\end{multline}
These maps  determine  a  dual Koszul complex:
\begin{equation}\label{koszul2}
Kos^n(f):\ \ 0\to \Gamma_n(P)\buildrel{\kappa ^n}\over\to
\Gamma_{n-1}(P)\otimes Q\buildrel{ \kappa^{n-1}}\over\to \dots \to
P\otimes \Lambda^{n-1}(Q) \buildrel{\kappa^1}\over\to
\Lambda^n(Q)\to 0
\end{equation}
The complexes $Kos_n(f)$ and $Kos^n(f)$ are the total degree $n$
components of the Koszul complexes $\Lambda(P)\otimes SP(Q)$ and
$\Gamma(P)\otimes \Lambda(Q)$ associated  to a given  homomorphism
$f:P\to Q$. For a two-term flat resolution \eqref{res11} of an
abelian group $A$,
 the
complexes $Kos_n(f)$ and $Kos^n(f)$ represent the derived category
objects $LSP^n(A)$ and $L\Lambda^n(A)$ respectively (see for
example  \cite{Kock}). In particular, when $P$ is free abelian and
$f$ the identity arrow, both complexes are acyclic.

For $n\geq 2,$ the derived category object
$LSP^{n-1}(A)\buildrel{L}\over \otimes A$ may be represented, for
some 2-term flat resolution  $f:L \to M$ of $A$, by the tensor
product of $Kos_n(f)$ and $L \to M$, in other words as a total
complex associated to the bicomplex
$$
\xymatrix@R=7pt{\Lambda^{n-1}(L)\otimes L\ar@{->}[r] \ar@{->}[d] & \dots \ar@{->}[r] & L\otimes SP^{n-2}(M)\otimes L \ar@{->}[r] \ar@{->}[d] & SP^{n-1}(M)\otimes L\ar@{->}[d]\\
\Lambda^{n-1}(L)\otimes M\ar@{->}[r] & \dots \ar@{->}[r] &
L\otimes SP^{n-2}(M)\otimes M\ar@{->}[r] & SP^{n-1}(M)\otimes M}
$$
The diagram  $LJ^n(A)\to LSP^{n-1}(A)\buildrel{L}\over\otimes A\to
LSP^n(A)$ in the derived category may therefore  be represented by
the following diagram of (horizontal) complexes:
\begin{equation}\label{kosdiag}
{
 \xymatrix@R=7pt{ Y^n(L) \ar@{^{(}->}[d] \ar@{^{(}->}[r] & \dots
\ar@{->}[r] & L\otimes SP^{n-1}(M)\otimes M \ar@{->}[r]
\ar@{^{(}->}[d] & J^n(M) \ar@{^{(}->}[d]
\\
\Lambda^{n-1}(L)\otimes L \ar@{->>}[d] \ar@{^{(}->}[r] & \dots
\ar@{->}[r] & L\otimes SP^{n-2}(M)\otimes M\oplus
SP^{n-1}(M)\otimes L \ar@{->>}[d] \ar@{->}[r] &
SP^{n-1}(M)\otimes M \ar@{->>}[d]\\
\Lambda^n(L) \ar@{^{(}->}[r] & \dots \ar@{->}[r] & L\otimes
SP^{n-1}(M) \ar@{->}[r] & SP^n(M)} }
\end{equation}
The upper line
\begin{equation}\label{y-kos} \xyma{ Y^n(L)  \ar@{->}[r] & \dots \ar@{->}[r] &
L\otimes SP^{n-1}(M)\otimes M \ar@{->>}[r]  & J^n(M)}
\end{equation} in this diagram
 is  a Koszul complex for the
 functors $J^n$ and $Y^n$. A similar diagram, whose lower line is the
 dual Koszul complex \eqref{koszul2} for $n=3$ is described in
 appendix \ref{app:derKcx} below.
 For  Koszul complexes
associated to  more general Schur functors, see \cite{lascoux}
lemma 1.9.1, \cite{a-b-w}.

\subsection{Pensions and d\'ecalage}

Consider the homomorphisms
 (\cite{Bou1} 7.4)
\begin{align*}
& \eta_n: \Lambda^n(A)\otimes \Lambda^n(B)\to SP^n(A\otimes B),\\
& \nu_n: \Gamma_n(A)\otimes SP^n(B)\to SP^n(A\otimes B).
\end{align*}
 These are characterized as the unique homomorphisms for which  the corresponding
 diagrams
\begin{equation}\label{pension1}
\xyma{\Lambda^n(A)\otimes \Lambda^n(B) \ar@{->}[r]^(.55){\eta_n}
\ar@{->}[d]_{g_n\otimes g_n} & SP^n(A\otimes B) \ar@{->}[d]_{f_n}\\
(\otimes^n (A))\otimes (\otimes^n (B)) \ar@{->}[r]^(.58){\lambda_n}
& \otimes^n(A\otimes B)}
\qquad  \qquad
\xyma{\Gamma_n(A)\otimes SP^n(B) \ar@{->}[r]^(.55){\nu_n}
\ar@{->}[d]_{h_n\otimes f_n} & SP^n(A\otimes B) \ar@{->}[d]_{f_n}\\
(\otimes^n (A))\otimes (\otimes^n (B)) \ar@{->}[r]^(.58){\lambda_n}
& \otimes^n(A\otimes B)}
\end{equation}
commutes, with $\lambda_n$  defined by
$$
\lambda_n: (a_1\otimes \dots \otimes a_n)\otimes (b_1\otimes
\dots\otimes b_n)\mapsto (a_1\otimes b_1)\otimes \dots\otimes
(a_n\otimes b_n),\ a_i\in A,\ b_i\in B.
$$
The map $\eta_n$ is  explicitly given   by the formula
\begin{equation}
\label{etan} (a_1 \wedge \ldots \wedge a_n) \otimes (b_1 \wedge
\ldots \wedge b_n) \mapsto \sum_{\sigma\in \Sigma_n} \text{sign}
(\sigma) (a_1 \otimes b_{\sigma(1)}) \ldots (a_n \otimes
b_{\sigma(n)})\,. \end{equation}

It is shown in  \cite{Bou1} 7.6 that for any free simplicial
abelian group $X$, the maps $\eta_n$ and $\nu_n$ induce the
isomorphisms of homotopy groups
\begin{align}
& \pi_i(\Lambda^n(X)\otimes \Lambda^nK(\mathbb Z,1))\to
\pi_i(SP^n(X\otimes K(\mathbb Z,1))),\ i\geq 0,\label{pen11}\\
& \pi_i(\Gamma_n(X)\otimes SP^nK(\mathbb Z,2))\to
\pi_i(SP^n(X\otimes K(\mathbb Z,2))),\ i\geq 0.\label{pen12}
\end{align}
so that there are natural isomorphisms
\begin{align*} & L\Lambda^n(A,m)\buildrel{L}\over\otimes
\Lambda^n(\mathbb
Z,1)\simeq LSP^n(A,m+1),\\
& L\Gamma_n(A,m)\buildrel{L}\over\otimes SP^n(\mathbb Z,2)\simeq
LSP^n(A,m+2).
\end{align*}
in the derived category. These derived pairings induce for  $n
\geq 1$, by adjunction with the volume element $n$-cycle in
$\Lambda^n(\mathbb{Z}, 1)_n = \Lambda^n(\mathbb{Z}^n)$ and the
corresponding element in $SP^n(\mathbb Z,2)$ respectively, a pair of
functorial pension morphisms
\begin{align}
&L\Lambda^j(A,n)[j] \la LSP^j(A,n+1)  \label{bq1a}\\
& L\Gamma^j(A,n)[j] \la L\Lambda^j(A,n+1)  \label{bq2a}
\end{align}
in the derived category. The maps induced on homotopy groups are
are  of the form
\begin{align}
& L_{i}\Lambda^j(A,n)\simeq  L_{i+j}\,SP^j(A,n+1) \label{bq1}\\
& L_{i}\Gamma_j(A,n)\simeq  L_{i+j}\,\Lambda^j(A,n+1)\label{bq2}
\end{align}
The inverses of  these maps are  iterated boundary maps arising from  the exactness of the Koszul
complexes and are known as d\'ecalage
isomorphisms  \cite{illusie}  I 4.3.2.  Composing the last two
determines a double d\'ecalage isomorphism:
\bee
\label{ddec}
 L_{i}\Gamma_j(A,n)\simeq L_{i+2j}\,SP^j(A,n+2)\,.
\ee

Similarly,  it follows
from the existence of  Koszul sequences of type \eqref{y-kos}
that there exist  d\'ecalage isomorphisms
\begin{equation}
\label{declie}
 L_{i}Y^j(A,n)\simeq  L_{i+j}J^j(A,n+1)
\end{equation}
between the derived functor of
$J^j$ and $Y^j$ for all $j,n \geq 0$.

\section{The de Rham complex and its dual}

\vspace{.5cm} Let $A$ be an abelian group. For $n\geq 1$, let
$D_*^n(A)$ and $C_*^n(A)$ be the complexes of abelian groups
defined by
\begin{align*}
& D_i^n(A)=SP^i(A)\otimes \Lambda^{n-i}(A),\ 0\leq i\leq n,\\
& C_i^n(A)=\Lambda^i(A)\otimes \Gamma_{n-i}(A),\ 0\leq i\leq n,
\end{align*}
where the differentials $d^i: D_i^n(A)\to D_{i-1}^n(A)$ and $d_i:
C_i^n(A)\to C_{i-1}^n(A)$ are:
\begin{align*}
& d^i((b_1\dots b_i)\otimes b_{i+1}\wedge \dots \wedge b_n)=
\sum_{k=1}^i(b_1\dots \hat b_k \dots b_i)\otimes b_k\wedge
b_{i+1}\wedge\dots \wedge b_n\\
& d_i(b_1\wedge \dots \wedge b_i\otimes X)=\sum_{k=1}^i(-1)^k
b_1\wedge \dots \wedge\hat b_k\wedge \dots \wedge b_i\otimes b_kX
\end{align*}
for any $X\in \Gamma_{n-i}(A)$. The complex $D^n(A)$ is the degree
$n$ component of the classical de Rham complex, first introduced
in the present context of polynomial functors in \cite{FLS} and
denoted $\Omega_n$ in \cite{Franjou}. The dual complexes  $C^n(A)$
were considered in \cite{Jean}. We will call them  the dual de
Rham complexes.

\bigskip

We will  now  give a functorial description of  certain homology
groups of these  complexes $C^n(A)$.
\begin{prop}\label{jean}
Let $A$ be a free abelian. Then \begin{enumerate} \item
\cite{Franjou} For any prime number $p$, $H_0C^p(A)=A\otimes
\mathbb Z/p,$ and $H_iC^p(A)=0,$ for all $i>0$; \item \cite{Jean}
 There is a natural isomorphism
$$ H_0C^n(A)\simeq \bigoplus_{p|n,\ p\ \text{prime}}\Gamma_{n/p}(A\otimes \mathbb Z/p).
$$
\end{enumerate}
\end{prop}
\noindent A  proof of proposition \ref{jean} (2) is given in
\cite{roman}.

\bigskip

The higher homology groups $H_iC^n(A)$ are more complicated. The
following table, which is a consequence of the main  theorem in \cite{roman},
gives a  complete description of
$H_iC^n(A)$ for $n\leq 7$ and
$A$   free abelian:
\begin{table}[H]
\begin{tabular}{ccccccccccccccc}
 $n$ & \vline & $H_0C^n(A)$ & $H_1C^n(A)$ & $H_2C^n(A)$ & $H_3C^n(A)$\\
 \hline
8 & \vline & $\Gamma_4(A\otimes \mathbb Z/2)$ & $ \ast$  & $\ast$ & $\ast$  \\
7 & \vline & $A\otimes \mathbb Z/7$ & 0 & 0 & 0\\
6 & \vline & $\Gamma_2(A\otimes \mathbb Z/3)\oplus \Gamma_3(A\otimes\mathbb Z/2)$ & $\Lambda^2(A\otimes \mathbb Z/3)\oplus \EuScript L^3(A\otimes \mathbb Z/2)$ & $\Lambda^3(A\otimes \mathbb Z/2)$ & 0\\
5 & \vline & $A\otimes \mathbb Z/5$ & 0 & 0 & 0\\
4 & \vline & $\Gamma_2(A\otimes \mathbb Z/2)$
& $\Lambda^2(A\otimes \mathbb Z/2)$ & 0&0\\
3 & \vline & $A\otimes \mathbb Z/3$ & 0 & 0&0\\
2 & \vline & $A\otimes \mathbb Z/2$ & 0&0&0\\
\end{tabular}
\caption{}
\end{table}
\noindent For example, the isomorphism
 \bee\label{c4iso}
f: \Lambda^2(A\otimes \mathbb Z/2)\to H_1C^4(A)\ee
is defined, for representatives $a,b\in A$ of $ \bar a,\bar b\in
A\otimes \mathbb Z/2  $, by
$$
f: \bar a\otimes \bar b\mapsto a\otimes a\gamma_2(b)-b\otimes
b\gamma_2(a).
$$

\subsection{Comparing  the de Rham and  Koszul complexes}
For any free abelian group $A$, consider the following natural
monomorphism of complexes:
\begin{equation}\label{compkos}
\xyma{Kos_n(A):\ar@{^{(}->}[d] & \Lambda^n(A) \ar@{=}[d]
\ar@{^{(}->}[r] & \Lambda^{n-1}(A)\otimes A \ar@{=}[d] \ar@{->}[r]
& \dots \ar@{->}[r]
& A\otimes SP^{n-1}(A) \ar@{->}[d] \ar@{->>}[r] & SP^n(A) \ar@{->}[d]\\
C^n(A): & \Lambda^n(A)\ar@{^{(}->}[r] & \Lambda^{n-1}(A)\otimes A \ar@{->}[r] & \dots \ar@{->}[r] & A\otimes \Gamma_{n-1}(A) \ar@{->}[r] & \Gamma_n(A)\\
}
\end{equation}
Let us denote the cokernel of this map by $D^n(A)$. We have
\begin{multline*}
D^n(A): \ \ \Lambda^{n-2}(A)\otimes W_2(A) \to
\Lambda^{n-3}(A)\otimes W_3(A) \to \dots\to A\otimes W_{n-1}(A)
\to W_{n}(A)
\end{multline*}
where
$$
W_n(A)={\rm coker}\{SP^n(A)\to \Gamma_n(A)\}
$$
Since Koszul complex is acyclic, it follows that
$$
H_iC^n(A)\simeq H_iD^n(A),\ n\geq 0.
$$
Proposition \ref{jean} implies that the sequence:
\begin{align}\label{w3}
0\to A\otimes A\otimes\mathbb Z/2 \to W_3(A)\to A\otimes \mathbb
Z/3\to 0
\end{align}
is exact (and  splits naturally).
 For every $n\geq 2,\ m\geq 0$, we obtain the natural exact
sequence:
\begin{align}
\label{lw3} 0\to LSP^n(A,m)\to L\Gamma_n(A,m)\to LW_n(A,m)\to 0
\end{align}
Passing to  homotopy groups and applying the d\'ecalage
isomorphisms (\ref{bq1}), (\ref{bq2}), this yields the  long exact
sequence
\begin{multline}\label{sequence}
\dots\to L_iSP^n(A,m)\to L_{i+2n}SP^n(A,m+2)\to L_iW_n(A,m)\to\\
L_{i-1}SP^n(A,m)\to L_{i+2n}SP^n(A,m+2)\to L_{i-1}W_n(A,m)\to
\dots
\end{multline}
Let $X$ be a free abelian simplicial group and $k\geq 1, n\geq 2$
be integers. If $\pi_i(X)=0,\ i<k,$ then by  \cite{DoldPuppe},
Satz 12.1
\begin{equation}\label{dp}
\pi_i(SP^n(X))=0,\begin{cases} \text{for}\ i<n,\ \text{when}\
k=1,\\
\text{for}\ i<k+2n-2,\ \text{provided}\ k>1.
\end{cases}
\end{equation}
 We will make use of exact sequence \eqref{sequence} and of the
 assertion \eqref{dp} in order  to  compute  derived functors
of polynomial functors of low degrees.

\section{Derived functors of quadratic functors}
\label{quad} For every abelian group $A$, the  exactness of the
sequence \eqref{exsp2} implies that  $W_2(A)\simeq A\otimes
\mathbb Z/2$. Since this functor is additive, it follows immediately
that
$$
L_iW_2(A,m)\simeq \begin{cases} A\otimes \mathbb Z/2,\ i=m\\
\Tor(A,\mathbb Z/2),\ i=m+1\\
\, 0,\quad  i \neq m, m+1
\end{cases}
$$
for all  $m$. Let us define a new  functor $\lam^2(A)$ by: \bee
\label{newfunct} \lambda^2(A):= \Lambda^2(A) \oplus \tor (A,\Z/2).
\ee  The long exact sequence (\ref{sequence}), the
connectivity result (\ref{dp}), and the d\'ecalage formulas
(\ref{bq1})and  (\ref{bq2}) produce the following  complete
description of the derived functors of the symmetric power
functor $SP^2$.
\begin{prop}{\rm
\label{derSP2}
\begin{align*}
& L_iSP^2(A,n)=\begin{cases} SP^2(A),\ i=0,\ n=0\\ \EuScript
S_2(A),\
  i=1,\ n=0\\
  \Lambda^2(A),\ i=2,\ n=1\\
A\otimes \mathbb Z/2,\ i=n+2,n+4,\dots, n+2[\frac{n-1}{2}]\\
\Tor(A,\mathbb Z/2),\ i=n+3,n+5,\dots, n+2[\frac{n-1}{2}]+1,\ i\neq 2n\\
 \Gamma_2(A),\ i=2n,\ n\neq 0\ \text{even}\\
\lambda^2(A)
,\ i=2n,\ n\neq 1\ \text{odd}\\
R_2(A),\ i=2n+1,\ n\neq 0\ \text{even}\\
\Omega_2(A),\ i=2n+1,\ n\ \text{odd}\\
0\ \text{for all other $i$.}
\end{cases}
\end{align*} }
\end{prop}
 We will only sketch  the
proof of this computation  in the present quadratic situation,
and  will discuss the more elaborate
case of cubical functors  in the following  section. These quadratic results were also
obtained   in  \cite{BauesPirashvili} \S 4  by a different method
(see also \cite{MPBook} \S A.15).

\bigskip

\noindent {\bf Proof:} The first two equations follow from the definitions. By
double d\'ecalage \eqref{ddec}, there is  an iterated isomorphism
\[ \Gamma_2(A) = L_0\Gamma_2(A,0) \simeq
L_4SP^2(A,2)\]
 which determines the sixth   equation above  for $n=2$. The general case of the sixth equation   then follows by induction when we consider the isomorphism $L_{2n}SP^2(A,n)  \simeq L_{2n+2}SP^2(A,n+2)$ from  \eqref{sequence} for $n$ even.      D\'ecalage also implies  that
\[ \Lambda^2A \simeq L_2SP^2(A, 1) \]
and  the sequence \eqref{lw3} then determines a short exact
sequence
\[  0 \to L_2SP^2(A,1) \to  L_2\Gamma_2(A,1) \to \text{Tor}(A,
\mathbb{Z}/2) \to 0\] Consider the following  diagram, in which
the vertical arrows are the suspension maps:
\[
\xymatrix{0 \ar[r]  & L_1SP^2A \ar[r]\ar[d]^{0}  & L_1\Gamma_2(A)
\ar[r] \ar[d]
 &\mathrm{Tor}(A,\mathbb{Z}/2)  \ar@{=}[d] \ar[r] & 0\\
0 \ar[r] & L_2SP^2(A,1) \ar[r] & L_2\Gamma_2(A,1) \ar[r] &
\mathrm{Tor}(A,\mathbb{Z}/2) \ar[r] & 0 }
\]
The left-hand vertical  arrow is trivial by \cite{DoldPuppe}
corollary 6.6, since all elements of $L_1SP^2A$ are in the image
of the arrow \eqref{epi} ($n=2)$. The lower sequence is therefore
split since it is a  pushout by the trivial map, and  a diagram chase
makes it clear
that this splitting is functorial. This proves
the seventh equation in  proposition \eqref{derSP2} for $n=3$  by the  double
d\'ecalage isomorphism
\[ L_2\Gamma_2(A,1) \simeq L_6SP^2(A,3) \]
 The general case of the sixth  equation now  follows by induction since \eqref{lw3} and d\'ecalage imply that
\[ L_iSP^2(A,n) \simeq L_i\Gamma_2(A,n) \simeq L_{i+4}SP^2(A, n+2) \]
for all $n$.  A similar discussion,  in the next degree, shows that
the seventh
and eight equations are also satisfied. The remaining  fourth and
fifth equations  are
proved by considering once more the sequence \eqref{sequence}, and
observing that the
 functors  $L_iSP^2(A,n)$  vanish by
 \cite{DoldPuppe} whenever $i$ is  sufficiently large.
\qed

\bigskip

As a corollary, one finds  that
this computation (and even the  inductive
 reasoning that led to it) can be carried over by the d\'ecalage isomorphisms
 \eqref{bq1}, \eqref{ddec}
 to the derived functors of $\Lambda^2$ and
 $\Gamma_2$.
 We simply state the result:
 \begin{align} \hspace{-1.3cm}  \label{derlam2ga2}
& L_i\Lambda^2(A,n)=\ \ \begin{cases}  \Lambda^2(A),\ i=0,\ n=0,\\
A\otimes \mathbb Z/2,\ i=n+1,n+3\dots, n+2[\frac{n-1}{2}]+1,\
i\neq
2n \\
\Tor(A,\mathbb Z/2),\ i=n+2,n+4,\dots, n+2[\frac{n-1}{2}]\\
\Gamma_2(A),\ i=2n,\ n\ \text{odd}\\
\lam^2(A)
,\ i=2n,\ n\neq 0\ \text{even}\\
R_2(A),\ i=2n+1,\ n\ \text{odd}\\
\Omega_2(A),\ i=2n+1,\ n\ \text{even}\\
0\ \text{for all other $i$.}
\end{cases}
\end{align}
\begin{align}
\label{derga2} & L_i\Gamma_2(A,n)=\ \ \begin{cases}
A\otimes \mathbb Z/2,\ i=n, n+2,\dots, n+2[\frac{n-1}{2}],\ n>0\\
\Tor(A,\mathbb Z/2),\ i=n+1, n+3,\dots, n+2[\frac{n-1}{2}]+1,\ n>0,\ i\neq 2n\\
\Gamma_2(A),\ i=2n,\ n\ \text{even}\\
\lam^2(A)
,\ i=2n,\ n\ \text{odd}\\
R_2(A),\ i=2n+1,\ n\ \text{even}\\
\Omega_2(A),\ i=2n+1,\ n\ \text{odd}\\
0\ \text{for all other $i$.}
\end{cases}
\end{align}

\section{The derived functors of certain  cubical functors}
\label{cub}  It follows from  (\ref{w3})  that
\bee
\label{defw3}
W_3(A)=(A\otimes A \otimes \mathbb{Z}/2)\oplus (A \otimes
\mathbb{Z}/3),
\ee
 and this derives to  an  isomorphism
\bee
\label{deflw3}
LW_3(A) \simeq (A \lot A \lot \mathbb{Z}/2) \oplus( A \lot \mathbb{Z}/3)
 \ee
in the derived category, from which the values of $L_iW_3(A)$ follow
immediately (consistently with the two equations  \eqref{l1ga2}).
 This implies  that
\begin{equation}\label{wf1}
L_iW_3(A,1)=\ \,\,\,\begin{cases} A\otimes \mathbb Z/3,\ i=1\\
A\otimes A\otimes \mathbb Z/2\oplus \Tor(A,\mathbb Z/3),\ i=2\\
\Tor_1(A,A,\mathbb Z/2),\ i=3\\
\Tor_2(A,A,\mathbb Z/2),\ i=4\\
0,\ \text{for all other $i$}
\end{cases}
\end{equation}
and, for $n>1$:
\begin{equation}\label{wf2}
\hspace{-.7cm} L_iW_3(A,n)=\begin{cases} A\otimes \mathbb Z/3,\ i=n\\
\Tor(A,\mathbb Z/3),\ i=n+1\\
A\otimes A\otimes \mathbb Z/2,\ i=2n,\\
\Tor_1(A,A,\mathbb Z/2),\ i=2n+1,\\
\Tor_2(A,A,\mathbb Z/2),\ i=2n+2,\\
0,\ \text{for all other  $i$.}
\end{cases}
\end{equation}

\bigskip

 We will now use this  computation in order to determine the
derived functors of $SP^3$ in all degrees.  Let $X$ be a free
abelian simplicial group. The natural map of simplicial groups
$$
E: SP^3(X)\otimes K(\mathbb Z,2)\to SP^3(X\otimes K(\mathbb Z,2))
$$
induces the pairing map \cite{Bou}:
$$
\epsilon_3: \pi_iSP^3(X)\otimes H_6K(\mathbb Z,2)\to
\pi_{i+6}SP^3(X\otimes K(\mathbb Z,2))
$$
Observe that the following diagram is commutative:
$$
\xyma{\pi_iSP^3(X) \ar@{->}[d] \ar@{->}[r]^{\epsilon_3\ \ \ \ \ \
\ \ \ \ } &
\pi_{i+6}SP^3(X\otimes K(\mathbb Z,2))\\
\pi_i\Gamma_3(X) \ar@{->}[r]^{\simeq\ \ \ \ \ \ \ \ \ \ \ \ } &
\pi_{i+6}(\Gamma_3(X)\otimes SP^3K(\mathbb Z,2)) \ar@{->}[u]}
$$
where the right hand vertical arrow is the homomorphism
\eqref{pen12}.
  It follows from
  \eqref{sequence}, \eqref{wf1} and \eqref{wf2} that the maps
\begin{equation}
\epsilon_3: L_iSP^3(A,n)\to L_{i+6}SP^3(A,n+2)\notag
\end{equation}
are isomorphisms for $,\ i\neq
n-1,n,n+1,n+2,2n-1,2n,2n+1,2n+2$. In addition the sequence
\begin{multline}\label{exactbou} 0\to L_{2n+2}SP^3(A,n)\to
L_{2n+8}SP^3(A,n+2)\to \Tor_2(A,A,\mathbb Z/2)\to\\
L_{2n+1}SP^3(A,n)\to
L_{2n+7}SP^3(A,n+2)\to \Tor_1(A,A,\mathbb Z/2)\to \\
L_{2n}SP^3(A,n)\to
L_{2n+6}SP^3(A,n+2)\to A\otimes A\otimes \mathbb Z/2\to \\
L_{2n-1}SP^3(A,n)\to L_{2n+5}SP^3(A,n+2)
\end{multline}
is exact by \eqref{sequence}. Furthermore, for $n>1$,
\begin{align*}
& L_{n+7}SP^3(A,n+2)\simeq \Tor(A,\mathbb Z/3)\\
& L_{n+6}SP^3(A,n+2)\simeq A\otimes \mathbb Z/3
\end{align*}
by (\ref{dp}). Finally, according to \cite{Bou} corollary 4.3 ,
the
 pension maps
$$
\varepsilon_3: L_iSP^3(A,n)\to L_{i+6}SP^3(A,n+2)
$$
are  split injections,   for all $i\geq 0$ and all  $n>1$.  The
long  exact sequence (\ref{exactbou}) therefore decomposes for $n>1$  into
short exact sequences:
\begin{align*}
0 \to &L_{2n+2}SP^3(A,n) \to  L_{2n+8}SP^3K(A,n+2) \to
\Tor_2(A,A,\mathbb
Z/2)  \to 0\\
0 \to & L_{2n+1}SP^3(A,n) \to  L_{2n+7}SP^3K(A,n+2)\to
\Tor_1(A,A,\mathbb
Z/2)\to 0\\
0 \to  &L_{2n}SP^3(A,n) \to  L_{2n+6}SP^3(A,n+2)\to A\otimes
A\otimes \mathbb Z/2  \to 0
\end{align*}
and an isomorphism
$$L_{2n+5}SP^3(A,n+2)\simeq  L_{2n-1}SP^3(A,n). $$
\begin{example}
\label{sp3a3} {\rm Since the values taken by the derived functors of
  $W_3$ in  \eqref{wf1} and \eqref{wf2}
are distinct, we must   consider the implications of\eqref{wf1}
separately.
Observe that
exact sequence (\ref{sequence}) and the equations \eqref{wf1} imply that
\begin{align*}
& L_{11}SP^3(A,3)=\Omega_3(A),\\
& L_8SP^3(A,3)=A\otimes A\otimes \mathbb Z/2\oplus \Tor(A,\mathbb
Z/3)
,\\ & L_7SP^3(A,3)=A\otimes \mathbb Z/3,\\
\end{align*}
and that the groups $L_iSP^3(A,3)$ for $i= 9,10$ live in the long  exact
sequence \begin{multline}\label{zzzzz1} 0\to L_1\Lambda^3(A)\to
L_{10}SP^3(A,3)\to \Tor_2(A,A,\mathbb Z/2)\buildrel{\partial}\over\to\\
\Lambda^3(A)\buildrel{\varepsilon_3}\over\to L_9SP^3(A,3)\to
\Tor_1(A,A,\mathbb Z/2)\to 0\,.
\end{multline}
The diagram
\[ \xymatrix{
\Lam^3(A)  \ar@{^{(}->}[rr]^{g_3} \ar[d]^{\wr} &&  \ot^{3}(A) \ar[d]^\wr  \\
  L_3SP^3(A,1) \ar[r] & L_3\Ga_3(A,1) \ar[r]^{h_3}  & L_3\ot^3(A,1)
}\]
commutes, where  the left-hand vertical arrow is the map
\eqref{bq1}, and the
right-hand one the corresponding obvious d\'ecalage map for tensor powers. It follows
that the composite map
\[ \epsilon_3: \Lambda^3A \la     L_3\Ga_3(A,1) \simeq L_9SP^3(A,3) \]
is injective so that
 the boundary map $\partial$  in (\ref{zzzzz1})
is trivial. The complete description of the $L_iSP^3(A,3)$  is
therefore given by
\begin{align*}
& L_iSP^3(A,3)=\begin{cases} \Omega_3(A),\ i=11\\
A\otimes A\otimes \mathbb Z/2\oplus \Tor(A,\mathbb Z/3),\ i=8\\
A\otimes \mathbb Z/3,\ i=7\\
0,\ i\neq 7,8,9,10,11
\end{cases}
\end{align*}
and the exactness of the sequences
\begin{align*}
& 0\to L_1\Lambda^3(A)\to L_{10}SP^3(A,3)\to \Tor_2(A,A,\mathbb Z/2)\to 0\\
& 0\to \Lambda^3(A)\to L_9SP^3(A,3)\to \Tor_1(A,A,\mathbb Z/2)\to
0\,.
\end{align*}
}
\end{example}

The discussion  in example \ref{sp3a3} does not only apply to the
derived functors $L_iSP^3(A,n)$ with $n=3$. The corresponding
assertion  for a
general $n$
is the following one:

\begin{theorem}
\label{sp3an} {\rm \noindent Case I: $n\geq 3$ is odd.
$$
L_iSP^3(A,n)=\begin{cases}
A\otimes \mathbb Z/3,\ n+4\leq i<2n+2,\ i-n\equiv 0\mod 4\\
\Tor(A,\mathbb Z/3),\ n+4\leq i<2n+2,\ i-n\equiv 1\mod 4\\
A\otimes A\otimes \mathbb Z/2,\ i=2n+2,\ n\equiv 1\mod 4,
\\ \Tor(A,\mathbb Z/3)\oplus
A\otimes A\otimes \mathbb Z/2,\ i=2n+2,\ n\equiv 3\mod 4\\
\Tor_1(A,A,\mathbb Z/2),\ 2n+3\leq i\leq 3n-2,\ i-n\equiv 2\mod
4,\\ A\otimes \mathbb Z/3\oplus \Tor_1(A,A,\mathbb Z/2),\ 2n+3\leq
i\leq 3n-2,\ i-n\equiv 0 \mod 4,\\
\Omega_3(A),\ i=3n+2,\\
 0,\
\text{for all other $i$.}
\end{cases}
$$
In addition, the folllowing sequences are exact:
\begin{align*}
& 0\to \Tor(A,\mathbb Z/3)\oplus A\otimes A\otimes \mathbb Z/2\to
L_iSP^3(A,3)\to \Tor_2(A,A,\mathbb Z/2)\to 0,\\ & \ \ \ \ \ \ \ \
\ \ \ \ \ \ \ \ \ \ \ \ \ \ \ 2n+4\leq
i\leq 3n-1,\ i-n\equiv 1\mod 4\\
& 0\to A\otimes A\otimes \mathbb Z/2\to L_iSP^3(A,n)\to
\Tor_2(A,A,\mathbb Z/2)\to 0,\\ & \ \ \ \ \ \ \ \ \ \ \ \ \ \ \ \
\ \ \ \ \ \ \ 2n+3\leq i\leq 3n-1,\ i-n\equiv
3 \mod 4,\\ & 0\to L_1\Lambda^3(A)\to L_{3n+1}SP^3(A,n)\to \Tor_2(A,A,\mathbb
Z/2)\to 0\\
& 0\to \Lambda^3(A)\to L_{3n}SP^3(A,n)\to \Tor_1(A,A,\mathbb Z/2)\to 0.
\end{align*}
Case II: $n>3$ is even.
$$
L_iSP^3(A,n)=\begin{cases} A\otimes \mathbb Z/3,\ n+4\leq i<2n+2,\
i-n\equiv 0\mod 4
\\
\Tor(A,\mathbb Z/3),\ n+4\leq i<2n+2,\ i-n\equiv 1\mod 4
\\
A\otimes A\otimes \mathbb Z/2,\ i=2n+2,\ n\equiv 0\mod 4,\\
 A\otimes \mathbb Z/3\oplus A\otimes A\otimes \mathbb Z/2,\ i=2n+2,\
 n\equiv 2\mod 4,\\
\Tor(A,\mathbb Z/3)\oplus \Tor_1(A,A,\mathbb Z/2),\ 2n+3\leq i\leq
3n-1,\ i-n\equiv 1\mod 4,\\ \Tor_1(A,A,\mathbb Z/2),\ 2n+3\leq
i\leq 3n-1,\ i-n\equiv 3\mod 4,\\
L_1\Gamma_3(A),\ i=3n+1, \\
 R_3(A),\ i=3n+2,\\
 0,\
\text{for all other $i$.}
\end{cases}
$$
In addition, the following sequences are exact:
\begin{align*}
& 0\to A\otimes \mathbb Z/3\oplus A\otimes A\otimes \mathbb Z/2\to
L_iSP^3(A,n)\to \Tor_2(A,A,\mathbb Z/2)\to 0,\\ & \ \ \ \ \ \ \ \
\ \ \ \ \ \ \ \ \ \ \ \ \ \ \ 2n+4\leq i\leq 3n-2,\ i-n\equiv
0\mod
4 \\
& 0\to A\otimes A\otimes \mathbb Z/2\to L_iSP^3(A,n)\to
\Tor_2(A,A,\mathbb Z/2)\to 0,\\ & \ \ \ \ \ \ \ \ \ \ \ \ \ \ \ \
\ \ \ \ \ \ \ 2n+4\leq i\leq 3n-2,\
i-n\equiv 2\mod 4,\\
& 0\to \Gamma_3(A)\to L_{3n}SP^3(A,n)\to \Tor_2(A,A,\mathbb
Z/2)\to 0.
\end{align*}
}\end{theorem}

 \bigskip

The corresponding  description of the derived functors of
$\Lambda^3$ and $\Gamma_3$ now follow by  the d\'ecalage
isomorphisms (\ref{bq1}), (\ref{bq2}).

\section{Some derived functors of $SP^4$.}
\vspace{.5cm} We will now make use of the computation of the
homology of the dual de Rham complex $C^4(A)$ in proposition   \ref{jean} in order to
investigate some of  the derived functors of $SP^4$. For $A$  a
free abelian group, we now consider the following diagram with
exact rows and columns, which extends diagram \eqref{compkos}
when  $n=4$:
$$
\xymatrix@R=7pt{\Lambda^4(A) \ar@{=}[d] \ar@{^{(}->}[r] &
\Lambda^3(A)\otimes A \ar@{=}[d] \ar@{->}[r] & \Lambda^2(A)\otimes
SP^2(A) \ar@{^{(}->}[d] \ar@{->}[r]
& A\otimes SP^3(A) \ar@{^{(}->}[d] \ar@{->>}[r] & SP^4(A) \ar@{^{(}->}[d]\\
\Lambda^4(A)\ar@{^{(}->}[r] & \Lambda^{3}(A)\otimes A \ar@{->}[r]
& \Lambda^2(A)\otimes \Gamma_2(A) \ar@{->>}[d] \ar@{->}[r] &
A\otimes \Gamma_3(A)
\ar@{->}[r] \ar@{->>}[d] & \Gamma_4(A) \ar@{->>}[d]\\
& & \Lambda^2(A)\otimes A\otimes\mathbb Z/2
\ar@{->}[r]^(.57)\lambda & A\otimes W_3(A) \ar@{->}[r] & W_4(A)}
$$
By proposition \ref{jean}, this determines a  functorial  diagram of exact sequences
\begin{equation}
\label{diag-w4}
\xymatrix@R=7pt{& SP^2(A)\otimes A\otimes \mathbb Z/2 \ar@{^{(}->}[d]\\
H_1D^4(A) \ar@{^{(}->}[r] & (A\otimes W_3(A))/\mathrm{im}(\lambda)
\ar@{->>}[d] \ar@{->}[r] & W_4(A)
\ar@{->>}[r] & \Gamma_2(A\otimes \mathbb Z/2)\\
& A\otimes A\otimes \mathbb Z/3}
\end{equation}
The map (\ref{c4iso}) defines  canonical isomorphisms
$$
H_1D^4(A)\simeq H_1C^4(A)\simeq \Lambda^2(A\otimes \mathbb
Z/2)\simeq L_1SP^2(A\otimes \mathbb Z/2).
$$
Let us define a  map
$$
\delta: \Lambda^2(A\otimes \mathbb Z/2)\to SP^2(A)\otimes A\otimes
\mathbb Z/2\ (\subset (A\otimes W_3(A))/\mathrm{im}(\lambda))
$$
  as follows, where $a, b$ are representatives in $A$ of the classes
  $\bar{a}$ and $\bar{b}$:
$$
\delta: \bar a\wedge \bar b\mapsto aa\otimes \bar b-bb\otimes \bar
a,\ a,b\in A,\ \bar a,\bar b\in A\otimes \mathbb Z/2.
$$
It follows  from this discussion that diagram  \eqref{diag-w4}
induces a short exact sequence
$$
0\to \frac{SP^2(A)\otimes A\otimes \mathbb Z/2}{\Lambda^2(A\otimes
\mathbb Z/2)}\oplus A\otimes A\otimes \mathbb Z/3\to W_4(A)\to
\Gamma_2(A\otimes \mathbb Z/2)\to 0.
$$
The filtration on $\Gamma_4(A)$  provided by this description of
$W_4(A)$ is consistent with that in \cite{Jean} proposition 3.1.2.
Together with long exact sequence (\ref{sequence}), it allows one
to  compute derived functors of the functor $SP^4$ by comparing them
to those of $\Gamma_4$ and taking into account the double d\'ecalage. For example,
one finds:
\begin{align*}
& L_9SP^4(A,3)\simeq A\otimes \mathbb Z/2,\\
& L_{10}SP^4(A,3)\simeq A\otimes A\otimes \mathbb Z/3\oplus
\Lambda^2(A\otimes \mathbb Z/2)\oplus \Tor(A,\mathbb Z/2
),\\
& L_{10}SP^4(A,4)\simeq A\otimes \mathbb Z/2,
\\
& L_{11}SP^4(A,4)\simeq \Tor(A,\mathbb Z/2
).
\end{align*}

\section{Lie and super-Lie functors}
We will now  consider the structure theory of Lie and super-Lie
functors.
\subsection{The third Lie functor}
\label{subsection-3lie} For any free abelian group $A$,  consider
the Koszul resolution:
\begin{equation}
\label{im-f} 0\to \Lambda^3(A)\to \Lambda^2(A)\otimes
A\buildrel{f}\over\rightarrow A\otimes SP^2(A)\to SP^3(A)\to 0,
\end{equation}
in which  the map $f$ is defined by
$$
f: a \wedge b\otimes c\mapsto a\otimes bc-b\otimes a c,\ a,b,c\in
A
$$
It  decomposes as
\begin{equation}\label{l3d}
\xyma{\Lambda^2(A)\otimes A \ar@{->}[r]^u & A\otimes A\otimes A
\ar@{->}[r]^v & A\otimes SP^2(A) }\end{equation} where
\begin{align}
\label{descr-uv} & u: a\wedge b\otimes c\mapsto a\otimes b\otimes
c-b\otimes a\otimes c -c\otimes a\otimes b+c\otimes b\otimes a,\\
\notag & v: a\otimes b\otimes c\mapsto a\otimes bc,\ a,b,c\in A.
\end{align}
Since  the expressions
$ u(a\wedge b \otimes  c$
 generate
$\EuScript L^3(A)$, the long exact sequence  \eqref{im-f}  decomposes as a
pair of short exact sequences
\begin{align}
\label{seq1}
&0\to \Lambda^3(A)\to \Lambda^2(A)\otimes A\to \EuScript L^3(A)\to 0\\
\label{sp3and} & 0\to \EuScript L^3(A)\buildrel{p_3}\over\to
A\otimes SP^2(A)\to
SP^3(A)\to 0\,,
\end{align}
In  particular   the map
\[ \xymatrix@R=3.5pt{
\EuScript L^3(A) \ar@{->>}[r]^{p_3} & J^3(A) \\
[a,b,c] \ar@{|->}[r] & (a,b,c) }\]
induced by  $p_3$ \eqref{defpn} is an isomorphism.

\bigskip

\begin{remark} \label{eliec}  {\rm {\it i})
The sequences \eqref{seq1} and \eqref{sp3and}  both remain exact
for an arbitrary group $A$. Indeed, \eqref{sp3and} derives for any
$A$ to  long exact sequences, and the arrow
\[\pi_1\left(A
\stackrel{L}{\otimes} LSP^2A \right) \la L_1SP^3A \] is
surjective, as follows from the presentation \eqref{syder} of
$L_1SP^3A$.

\medskip

\hspace{2cm} {\it ii}) There is a natural isomorphism
\bee
\label{natural}
\EuScript L^3(A)\simeq E^3(A):=\ker\{\Gamma_2(A)\otimes A\to
\Gamma_3(A)\}\,,
\ee
as  follows from the following prolongation of part of diagram \eqref{compkos}
for $n=3$:
\bee
\label{prolong}
\xymatrix@R=7pt{& \pi_1\left(A\buildrel{L}\over\otimes
A\buildrel{L}\over\otimes \mathbb Z/2 \right) \ar@{^{(}->}[r]
\ar@{->}[d] & \pi_1\left(A\buildrel{L}\over\otimes
A\buildrel{L}\over\otimes
\mathbb Z/2 \right)\oplus \Tor(A,\mathbb Z/3)\ar@{->}[d]\\
\EuScript L^3(A) \ar@{^{(}->}[r] \ar@{->}[d] & SP^2(A)\otimes A
\ar@{->>}[r] \ar@{->}[d]&
SP^3(A) \ar@{->}[d] \\
E^3(A) \ar@{^{(}->}[r] & \Gamma_2(A)\otimes A \ar@{->}[r]
\ar@{->>}[d] &
\Gamma_3(A) \ar@{->>}[d] \ar@{->>}[r] & A\otimes \mathbb Z/3\ar@{=}[d]\\
& A\otimes A\otimes \mathbb Z/2 \ar@{->}[r] & (A\otimes A\otimes
\mathbb Z/2)\oplus A\otimes \mathbb Z/3 \ar@{->}[r] & A\otimes
\mathbb Z/3}\ee
}
\end{remark}

\subsection{The Curtis decomposition} We will now  consider higher Lie
functors. Curtis gave in  \cite{Curtis:63} (see also
\cite{MPBook}) a decomposition of the functors $\EuScript L^n (A)$
into functors $SP^n$, $J^n$ and their iterates.
For example, when $A$ is a free abelian group  we have  the
following  decompositions in low degrees:
\begin{align}
\label{curdec}
\EuScript L^1(A)=& A\\
\EuScript L^2(A)=& J^2(A)   \notag\\
\EuScript L^3(A)=& J^3(A)\notag  \\
\EuScript L^4(A)=& J^2(J^2(A))\oplus J^4(A)  \notag \\
\EuScript L^5(A)=& (J^3(A)\otimes J^2(A))\oplus J^5(A) \notag\\
\EuScript L^6(A)=& J^3(J^2(A))\oplus J^2(J^3(A))\oplus
(J^4(A)\otimes J^2(A))\oplus J^6(A) \notag  \\
\EuScript L^7(A)=& (J^3(A)\otimes SP^2(J^2(A)))\oplus
(J^5(A)\otimes J^2(A))\oplus (J^2(J^2(A))\otimes J^3(A))\notag
\\ &
\oplus (J^4(A)\otimes J^3(A))\oplus J^7(A) \notag   \\
\EuScript L^8(A)=& J^2(J^2(J^2(A)))\oplus J^2J^4\oplus
(J^3(A)\otimes J^2(A)\otimes J^3(A))\notag    \\ & \oplus
J^5(A)\otimes J^3(A)\oplus J^4(J^2(A))\oplus (J^4(A)\otimes
SP^2(J^2(A))) \notag   \\ &\oplus (J^6(A)\otimes J^2(A))\oplus
J^8(A)\notag
\\
 & \dots  \notag
\end{align}

We will refer to these descriptions of the Lie functors  as their {\it Curtis decompositions}. It should be understood that the splittings  into direct sums displayed here  are not functorial, and
that all that exists functorially   are  filtrations  of
the $\EuScript L^n(A)$, whose associated graded  components are
the expressions displayed.  As a matter of convenience, we  will nevertheless
 refer to  these expressions
 as summands of the Lie functors. We have already come across the cases
$n=2,3$ of these decompositions (prop. \ref{slez}).
The next two  cases are the short exact sequences
\begin{align}
\label{curtis4} 0\to &\Lambda^2\Lambda^2(A)\to \EuScript  L^4(A)
\stackrel{p_4}{\to} J^4(A) \to 0\\\label{curtis5}
 0\to &\Lambda^2(A)\otimes J^3(A)\to \EuScript L^5(A)    \stackrel{p_5}{\to}  J^5(A)
\to 0,
\end{align}
where the left-hand arrows are respectively defined by
\begin{align*}
(a\wedge b) \wedge (c\wedge d) &\mapsto [[a,b],[c,d]] \\
(a\wedge b) \otimes (c,d,e) &\mapsto [[a,b],[c,d,e]]\,.
\end{align*}
 It is a general fact that   the final term in the decompostion of
$\EuScript L^n(A)$ is always $J^n(A)$, the projection  of $
\EuScript L^n(A)$ onto $J^n(A)$ being the map $p_n$ \eqref{defpn}.

\subsection{Super-Lie functors}
We will now  define  super-Lie functors
\begin{equation}\label{slie}
\EuScript L_{s}^n: {\sf Ab\to Ab},\ n\geq 1\,.
\end{equation}

\begin{defi} \cite{Leibowitz}
A {\it graded Lie ring with squares}  (GLRS for short) is a graded
abelian group $B=\bigoplus_{i=0}^\infty B_i$ with homomorphisms
\begin{align} & \{\ ,\ \}: B_i\otimes B_j\to B_{i+j},\ \label{superbracket}\\ & ^{[2]}: B_n\to
B_{2n}\ \text{for}\ n\ \text{odd}
\end{align}
such that the following conditions are satisfied (for elements
$x\in B_i,\ y\in B_j,\ z\in B_k$):
\begin{align}
& 1)\ \{x,y\}+(-1)^{ij}\{y,x\}=0 \label{santi}\\
& 2)\ \{x,x\}=0\ \ \text{for}\ i\ \text{even} \notag\\
& 3)\
(-1)^{ik}\{\{x,y\},z\}+(-1)^{ji}\{\{y,z\},x\}+(-1)^{kj}\{\{z,x\},y\}=0\label{sjacob}\\
& 4)\ \{x,x,x\}=0 \notag \\
& 5)\ (ax)^{[2]}=a^2x^{[2]}\ \text{for}\ \ i\ \text{odd},\ a\in
\mathbb Z  \notag   \\
& 6)\ (x+y)^{[2]}=x^{[2]}+y^{[2]}+\{x,y\}\ \text{for}\ \ i=j\
\text{odd}  \notag  \\
& 7)\ \{y,x^{[2]}\}=\{y,x,x\}\ \ \text{for}\ i\
\text{odd}.\label{yxx}
\end{align}
\end{defi}
For an abelian group $A$, define $\EuScript L_s(A)$ to be the
graded Lie ring with squares freely generated by $A$ in degree 1.
It may be defined as a GLRS together with a homomorphism of
abelian groups $l: A\to \EuScript L_s(A)$ such that for every map
$f: A\to B$ with $B$ a GLRS, there is a unique morphism of GLRS
$d: \EuScript L_s(A)\to B$ such that $f=d\circ l$. The abelian
group $\EuScript L_s(A)$ is naturally graded by  $\EuScript
L_s(A)=\bigoplus_{n=1}^\infty \EuScript L_s^n(A)$ and for any $x \in
\Le_s(A)$, we set $|x|= n$
whenever $x \in \Le^n_s(A)$. The $n$th
graded piece $\EuScript L_s^n(A)$ is called the $n$th super-Lie
functor applied to $A$. In particular, there is a natural
isomorphism
\begin{align*}
 \Gamma_2(A) &\simeq   \EuScript L_s^2(A)\\
\gamma_2(a)  &\mapsto a^{[2]}
\end{align*}
analogous to  \eqref{simplest}.

\bigskip

For any free abelian group $A$, the natural monomorphism
$$
z_n: \EuScript L_s^n(A)\to \otimes^n(A),
$$
is defined inductively on homogeneous elements  by
\begin{align*}
& \{x,y\}\mapsto z_{|x|}(x)\otimes
z_{|y|}(y)-(-1)^{|x||y|}z_{|y|}(y)\otimes z_{|x|}(x),\\ &
x^{[2]}\mapsto z_{|x|}(x)\otimes z_{|x|}(x).
\end{align*}

\subsection{The third super-Lie functor.} We  will now adapt the discussion of section
 \ref{subsection-3lie}  to the
context of the super-Lie functors. The relations (\ref{sjacob})
and (\ref{yxx}) imply that the group $\EuScript L_s^3(A)$ can be
identified with the subgroup of $A\otimes A\otimes A$ generated by
the elements \bee \label{slie-3} a\otimes b\otimes c+b\otimes
a\otimes c-c\otimes a\otimes b-c\otimes b\otimes a,\ a,b,c\in A,
\ee which are the super analogues of the generators \eqref{lie-3}
of $\EuScript L^3(A)$. Let us now show that there is a natural
isomorphism
\begin{align}
\label{l3sy3}
\EuScript L_s^3(A)&\simeq Y^3(A) \\
\{\{a,b\},c\} &\mapsto a\otimes b  \wedge c+b\otimes a\wedge c,\
a,b,c\in A. \notag
\end{align}
Consider, for  any  free abelian group $A$, the Koszul resolution
$\mathit{Kos}^3(A)$ \eqref{koszul2}
\begin{equation}
\label{kos2-4} 0\to \Gamma_3(A)\buildrel{i}\over\to
\Gamma_2(A)\otimes A\buildrel{\bar f}\over\rightarrow A\otimes
\Lambda^2(A)\to \Lambda^3(A)\to 0,
\end{equation}
where the maps $\bar f$  and $i$ are  defined  by
\begin{align*}
&\bar f: \gamma_2(a) \otimes b\mapsto a\otimes a\wedge b, a,b\in A \\ &\\
& i: \begin{cases}\gamma_3(a)\mapsto \gamma_2(a)\otimes a,\ \ a\in A\\
 \gamma_2(a)b\mapsto \gamma_2(a)\otimes b+ab\otimes a, \ \ \  \ a,b\in
A.
\end{cases}
\end{align*}
The map $\bar f$ factors as
\begin{equation}\label{l3d1}
\xyma{\Gamma_2(A)\otimes A
 \ar@{->}[r]^{\bar u} & A\otimes A\otimes A \ar@{->}[r]^{\bar v} & A\otimes \Lambda^2(A) }\end{equation}
where
\begin{align*}
& \bar u: \gamma_2(a)\otimes b\mapsto a\otimes a\otimes b-b\otimes a\otimes a,\\
& \bar v: a\otimes b\otimes c\mapsto a\otimes b\wedge c,\ \ \
a,b,c\in A.
\end{align*}
It follows that
$$
\bar u(ab\otimes c)=\{\{a,b\},c\},\ \ \ a,b,c\in A
$$
so that  \eqref{kos2-4}  decomposes as  a pair of short exact
sequences
\[ 0 \to \Gamma_3(A) \to \Gamma_2(A) \otimes A \to  \EuScript L_s^3(A)
\to 0 \]
\begin{equation}
\label{sp3and1} 0  \to\EuScript L_s^3(A)  \to A \otimes
\Lambda^2(A) \to \Lambda^3(A) \to 0.
\end{equation}
It follows from the presentation \eqref{derlam} of  $L_1\Lambda^3A$
that
these two  sequences remain exact  when  $A$ is  an arbitrary
abelian group.

\subsection{Higher super-Lie functors.}
 There  exists a decomposition of   super-Lie functors,
analogous to the Curtis decomposition of Lie functors, which we
will now describe in low degrees. We begin by defining inductively
an  extension of the  Lie super-bracket \eqref{superbracket} to
 left-normalized $n$-fold brackets,   by setting
$$\{a_1,\dots,
a_n\}=\{\{a_1,\dots,a_{n-1}\},a_n\}.$$ The relations
(\ref{santi}), (\ref{sjacob}) and (\ref{yxx}) imply that the group
$\EuScript L_s^n(A)$ is generated by the  elements
$\{a_1,\dots, a_n\}$ and  $b^{[2]}$ for all $\ a_i\in A,$ and  all $b\in \EuScript L^k_s(A)$ (with  $k$  odd and $n= 2k$).

For all  $n\geq 2$ and  abelian group $A$, the natural epimophism
$ \bar{p}_n: \EuScript L_s^n(A)\to Y^n(A), $ is defined by
\begin{align}
\label{barin} 
& \{a_1,a_2,\dots, a_n\}\mapsto
a_1\otimes a_2\wedge \dots \wedge a_n+a_2\otimes a_1\wedge
a_3\wedge \dots \wedge a_n
\\
& \qquad  \qquad \quad \  \   b^{[2]}\mapsto 0
.
\end{align}
\begin{prop}
\label{super} For any  free abelian group $A$, the  sequence of
abelian groups   \begin{equation}\label{equ75} 0\to
\Lambda^2\Gamma_2(A)\buildrel{j}\over\to \EuScript
L_s^4(A)\buildrel{\bar{p}_4}\over\to A\otimes \Lambda^3(A)\to
\Lambda^4(A)\to 0
\end{equation}
is exact, with  $j$  and $\bar{p}_4$ respectively defined by
\begin{align}
\label{defj} &j: \gamma_2(a_1)\wedge \gamma_2(a_2)\mapsto
\{a_1,a_2,a_1,a_2\}
\\ &
\bar{p}_4: \{a_1,a_2,a_3,a_4\}\mapsto a_1\otimes a_2\wedge
a_3\wedge a_4+a_2\otimes a_1\wedge a_3\wedge a_4,\ \notag
\end{align}
\end{prop}
 The
relations (\ref{santi}) and (\ref{sjacob}) imply that
$$
\{a_1,a_2,a_1,a_2\}=-\{a_2,a_1,a_2,a_1\},\ a_1,a_2\in A.
$$
In addition
$$
j: (\gamma_2(a+b)-\gamma_2(a)-\gamma_2(b))\wedge
\gamma_2(c)\mapsto \{a,c,b,c\}-\{b,c,a,c\},\ a,b,c\in A
$$
so that the map $j$ is well-defined.

\begin{remark} {\rm  For an arbitrary abelian group $A$, the sequence
(\ref{equ75}) must be replaced by:
\begin{equation}\label{equ76}
0\to \Lambda^2Y^2(A)\to \EuScript L_s^4(A)\to A\otimes
\Lambda^3(A)\to \Lambda^4(A)\to 0.
\end{equation}
and the following long exact sequence describes
the relation between the functors $\Gamma_2$ and $Y^2$:
$$
0\to R_2(A)\to \Tor(A,A)\to \Omega_2(A)\to \Gamma_2(A)\to Y^2(A)\to
0.
$$}
\end{remark}

 Let us  define  a functor $\widetilde{Y}^n$
 by the short exact sequence
\begin{equation}
\label{exactyn}
 \xymatrix{0 \ar[r] & \widetilde{Y}^n(A)  \ar[r]^j &  \EuScript
L_s^n(A)  \ar[r]^{\bar{p}_n} & Y^n(A)  \ar[r] &0\,.
}\end{equation} Sequence (\ref{equ76}) asserts in particular that
\begin{equation}
\label{exacty4a}\widetilde{Y}^4(A) =
\Lambda^2Y^2(A),\end{equation}
 so that we have the following
 super-analogue for $n=4$  of the short exact sequence \eqref{defjntilde}:
\begin{equation}
\label{exacty4}
 \xymatrix{0 \ar[r] &\Lambda^2Y^2(A)  \ar[r]^j &  \EuScript
L_s^4(A)  \ar[r]^{\bar{p}_4} & Y^4(A)  \ar[r] &0 \,.
}\end{equation}
\begin{flushright}
$\square$
\end{flushright}

Similarly, the     short exact sequence \eqref{exactyn} for $n=5$,
which is the  super-analogue of the decomposition \eqref{curtis5}
of $\EuScript L^5(A)$, is   described more precisely by the short exact
sequence
\begin{equation}
\label{exacty5} \xymatrix{ 0 \ar[r] &  Y^3(A) \otimes  \Ga_2(A)
\ar[r]^(.6)h & \EuScript L^5_s(A)
 \ar[r]^{\bar{p}_5}
& Y^5(A) \ar[r] &0\, ,}
\end{equation}
where the arrow $h$ is defined by
\begin{align*}
h: \{a,b,c\} \otimes \ga_2(d) & \mapsto \{a,b,c,d,d\}
\end{align*}
\begin{remark}{\rm
One can show  that there is a natural filtration on the  term
$\widetilde{Y}^6$ has a decomposition with an associated component
$\Gamma_2Y^3(A)$, so that  $\EuScript L_s^6(A)$ can have some
4-torsion whenever  there is some  2-torsion in the group $A$. In
fact, this is a general phenomenon: for all $k \geq 1$, there may
be some 4-torsion in $\EuScript L_s^{4k+2}(A)$ whenever $A$ is a
2-torsion group, whereas there will only be 2-torsion in all other
components of the super-Lie algebra  $\Le_s(A)$.}
\end{remark}
\subsection{Relations between Lie and super-Lie functors} Let $A$ be a free
abelian group, and consider, for  $n\geq 2,$ the  natural
monomorphisms
\begin{align*}
& c_n: \EuScript L^n(A)\to \otimes^n A ,\\
& z_n: \EuScript L_s^n(A)\to \otimes^n A.
\end{align*}
For $n=3$, we have by \eqref{descr-uv}, \eqref{slie-3}, for $\
a,b,c\in A$:
\begin{align*}
& c_3: [a,b,c]\mapsto a\otimes b\otimes c-b\otimes a\otimes c-c\otimes a\otimes b+c\otimes b\otimes a\\
&   z_3: \{a,b,c\}\mapsto a\otimes b\otimes c+b\otimes a\otimes
c-c\otimes a\otimes b-c\otimes b\otimes a.
\end{align*}

For any  pair of free abelian groups $A$ and $B$, and $n\geq 2,$
we  define a pair of  morphisms
\begin{align}
& \chi_n: \EuScript L_s^n(A)\otimes \Lambda^n(B)\to \EuScript
L^n(A\otimes B)\label{defchin}\\
& \bar \chi_n: \EuScript L^n(A)\otimes \Lambda^n(B)\to \EuScript
L_s^n(A\otimes B)\label{defchin8}
\end{align}
for  $a_1,\dots, a_n\in A$ and  $b_1,\dots, b_n\in B$, by
\begin{align*}
\chi_n: \{a_1,\dots, a_n\}\otimes b_1\wedge \dots \wedge b_n
&\mapsto \sum_{\sigma\in \Sigma_n}\text{sign}(\sigma)[a_1\otimes
b_{\sigma_1},\dots, a_n\otimes b_{\sigma_n}]
\\
\bar\chi_n: [a_1,\dots, a_n]\otimes b_1\wedge \dots \wedge
b_n&\mapsto \sum_{\sigma\in
\Sigma_n}\text{sign}(\sigma)\{a_1\otimes b_{\sigma_1},\dots,
a_n\otimes b_{\sigma_n}\} \,.
\end{align*}
For $n=2k$ with $k$ odd, we set
$$
\chi_n: \{a_1,\dots, a_k\}^{[2]}\otimes b_1\wedge\dots \wedge
b_n\mapsto \sum_{\sigma\in A_n}[[a_1\otimes b_{\sigma_1},\dots,
a_k\otimes b_{\sigma_k}],[a_1\otimes b_{\sigma_{k+1}},\dots,
a_k\otimes b_{\sigma_{2k}}]],
$$

.
\begin{theorem} Let $A$ and $B$ be free abelian groups. The  following  diagrams with arrows defined by
 \eqref{defchin}, \eqref{defchin8},  \eqref{rrref} are commutative:
\begin{equation}\label{diagramslie}
\xyma{\EuScript L_s^n(A)\otimes \Lambda^n(B)
\ar@{->}[d]^{z_n\otimes g_n}\ar@{->}[r]^(.55){\chi_n} & \EuScript L^n(A\otimes B) \ar@{->}[d]^{c_n}\\
(\otimes^nA)\otimes (\otimes^nB) \ar@{->}[r]^(.58){\lambda_n} &
\otimes^n(A\otimes B)} \qquad \qquad
 \xyma{\EuScript L^n(A)\otimes \Lambda^n(B)
\ar@{->}[d]^{c_n\otimes g_n}\ar@{->}[r]^(.55){\bar\chi_n} & \EuScript L_s^n(A\otimes B) \ar@{->}[d]^{z_n}\\
(\otimes^nA)\otimes (\otimes^nB) \ar@{->}[r]^(.58){\lambda_n} &
\otimes^n(A\otimes B)}
\end{equation}
\end{theorem}
\begin{proof}
Let us begin by considering the first diagram (\ref{diagramslie})
for  $n=2$. The commutativity of  diagram
$$
\xyma{\Gamma_2(A)\otimes \Lambda^2(B)
\ar@{->}[d]^{z_2\otimes g_2}\ar@{->}[r]^{\chi_2} & \Lambda^2(A\otimes B) \ar@{->}[d]^{c_2}\\
(A\otimes A)\otimes (B\otimes B) \ar@{->}[r]^{\lambda_2} &
(A\otimes B)\otimes (A\otimes B)}
$$
can be checked directly: for any $a_1,\in A,\ b_1,b_2\in B$:
\begin{align*}
& \lambda_2\circ(z_2\otimes g_2)(\gamma_2(a_1)\otimes b_1\wedge
b_2)=c_2\circ \chi_2(\gamma_2(a_1)\otimes b_1\wedge b_2)=\\ &
(a_1\otimes b_1)\otimes (a_1\otimes b_2)-(a_1\otimes b_2)\otimes
(a_1\otimes b_1).
\end{align*}
By induction on $n$, we find that
\begin{multline*}
z_n\otimes g_n(\{a_1,\dots, a_n\}\otimes b_1\wedge\dots\wedge
b_n)=\sum_{\sigma\in \Sigma_n}\text{sign}(\sigma)\,z_n\{a_1,\dots,
a_n\}\otimes b_{\sigma_1}\otimes \dots \otimes b_{\sigma_n}=\\
\sum_{\sigma\in
\Sigma_n}\text{sign}(\sigma)\,z_{n-1}\{a_1,\dots,a_{n-1}\}\otimes
a_n\otimes b_{\sigma_1}\otimes \dots\otimes b_{\sigma_n}\\+
(-1)^n\sum_{\sigma\in \Sigma_n}\text{sign}(\sigma)\,a_n\otimes
z_{n-1}\{a_1,\dots,a_{n-1}\}\otimes b_{\sigma_1}\otimes
\dots\otimes b_{\sigma_n}
\end{multline*}
Hence
\begin{multline*}
\lambda_n\circ(z_n\otimes g_n)(\{a_1,\dots, a_n\}\otimes
b_1\wedge\dots\wedge b_n)=\\
\sum_{j=1}^n\lambda_{n-1}\left(\sum_{\eta=\{\sigma_1,\dots,
\sigma_{n-1}\}\in \{1,\dots, \hat j,\dots,
n\}}\text{sign}(\eta,j)\,z_{n-1}\{a_1,\dots,a_{n-1}\}\otimes
b_{\sigma_1}\otimes \dots \otimes b_{\sigma_{n-1}}\right)\otimes
(a_n\otimes b_j)\\
+ (-1)^n\sum_{j=1}^n(a_n\otimes b_j)\otimes\lambda_{n-1}\left(
\sum_{\eta=\{\sigma_1,\dots,\sigma_{n-1}\}\in\{1,\dots,\hat
j,\dots,
n\}}\text{sign}(j,\eta)\,z_{n-1}\{a_1,\dots,a_{n-1}\}\otimes
b_{\sigma_1}\otimes \dots \otimes b_{\sigma_{n-1}}\right)\\=
\sum_{j=1}^n(\sum_{\eta=\{\sigma_1,\dots, \sigma_{n-1}\}\in
\{1,\dots, \hat j,\dots,
n\}}\text{sign}(\eta,j)\,(c_{n-1}[a_1\otimes b_{\sigma_1},\dots,
a_{n-1}\otimes b_{\sigma_{n-1}}]\otimes (a_n\otimes b_j) \ -\\
(a_n\otimes b_j)\otimes c_{n-1}[a_1\otimes
b_{\sigma_1},\dots,a_{n-1}\otimes b_{\sigma_{n-1}}]))=\\
\sum_{\sigma\in \Sigma_n}\text{sign}(\sigma)\,c_n[a_1\otimes
b_{\sigma_1},\dots, a_n\otimes b_{\sigma_n}]=c_n\circ
\chi_n(\{a_1,\dots, a_n\}\otimes b_1\wedge \dots\wedge b_n)
\end{multline*}
 One can prove in a similar manner  that
\begin{equation}\label{oddc}
c_{2k}\circ \chi_{2k}(\{a_1,\dots, a_k\}^{[2]}\otimes b_1\wedge
\dots\wedge b_{2k})= \lambda_{2k}\circ(z_{2k}\otimes
g_{2k})(\{a_1,\dots, a_k\}^{[2]}\otimes b_1\wedge \dots\wedge
b_{2k})
\end{equation} for an odd $k$,
 so that   the commutativity of the first
diagram (\ref{diagramslie}) is proved. The commutativity of the
second diagram
 is proved in a similar manner: for $n=2$, it
follows from the commutativity of the diagram (\ref{pension1}) and
one then simply repeats the previous computation for a general
$n$,
 with appropriate changes in  the signs of the various  expressions.
\end{proof}
For any pair of abelian groups $A,B$, we define as follows
 a natural arrow:
\begin{equation}
\label{defbetan} \beta_n: Y^n(A)\otimes \Lambda^n(B)\to
J^n(A\otimes B)
\end{equation}
\begin{multline*}
\beta_n: \bar p_n\{a_1,\dots, a_n\}\otimes b_1\wedge \dots\wedge
b_n\mapsto \sum_{\sigma\in
\Sigma_n}\text{sign}(\sigma)\,p_n[a_1\otimes b_{\sigma_1},\dots,
a_n\otimes b_{\sigma_n}],\\ a_1,\dots, a_n\in A,\ b_1,\dots,
b_n\in B.
\end{multline*}

\begin{prop}\label{jfunctor}
There is a natural commutative diagram with exact columns
$$
\xymatrix{Y^n(A)\otimes \Lambda^n(B)
\ar@{^{(}->}[d]\ar@{->}[r]^{\beta_n} &
J^n(A\otimes B)\ar@{^{(}->}[d] \\
A\otimes \Lambda^{n-1}(A)\ar@{->>}[d]\otimes \Lambda^n(B)
\ar@{->}[r]^(.48){\eta_n'} & (A\otimes B)\otimes SP^{n-1}(A\otimes
B)\ar@{->>}[d]\\
\Lambda^n(A)\otimes \Lambda^n(B) \ar@{->}[r]^{\eta_n} &
SP^n(A\otimes B) \,,}
$$
where the map $\eta_n $ is defined by  \eqref{etan} and $\eta_n'$
is  by
$$
\eta_n': a_1\otimes a_2\wedge \dots \wedge a_n\otimes b_1\wedge
\dots \wedge b_n\mapsto \sum_{\sigma\in \Sigma_n}
\mathrm{sign}(\sigma)\,(a_1\otimes b_{\sigma_1})\otimes
(a_2\otimes b_{\sigma_2})\dots (a_n\otimes b_{\sigma_n})\,.
$$
\end{prop}
The proof of this proposition follows directly from the definition
of the
various  maps.
\section{Derived functors of  Lie functors}\vspace{.5cm}
 It is asserted in \cite{Schlesinger:66} that  if
$p$ is an odd prime then the groups $L_{n+k} \EuScript L^p(\mathbb
Z, n)$ are $p$-torsion for all $k$, and in particular
\begin{equation}\label{primedec}
L_{n+k}\EuScript L^p(\mathbb Z,n)=\begin{cases} \mathbb Z/p,\
k=2i(p-1)-1,\ i=1,2,\dots, [n/2]\\
0,\ \text{otherwise}\end{cases}
\end{equation}
In the next three subsections, we will give a direct  proof of
this fact for $p=3$, in other words show that
\begin{equation}\label{dfr}
L_{n+k}\EuScript L^3(\mathbb Z,n)=\begin{cases} \mathbb Z/3,\
k=4i-1,\
i=1,2,\dots, [n/2]\\
0,\ \text{otherwise}
\end{cases}
\end{equation}
and we will more generally compute  in theorem \ref{derl3-1}  the
derived functors  $L_{n+k}\EuScript L^3(A,n)$  for a general
abelian group $A$. Note that  the derived functors of  the  Lie
functors $\EuScript L^q$ are complicated when $q$ a composite
number, and we refer to \cite{MPBook}, page 280 for a description
of these in low degrees. One finds for example that
$$
L_i\EuScript L^8(\mathbb Z,1)=\begin{cases} \mathbb Z/2,\
i=4,5,7\\
0,\ i\neq 4,5,7
\end{cases}
$$
The gap in the homotopy groups which occurs here for $i=6$ is the
illustration of a general phenomenon which, as we will see for example
in
\eqref{exa3}, also occurs in more elaborate contexts.

Returning to the $p=3$ case, let us
 first observe that  for any abelian group
$A$ the exact sequence (\ref{sp3and}) derives to a long exact
sequence
\begin{multline}
\dots\to L_{i+1}\EuScript L^3(A,n)\to
\pi_{i+1}\left(A[n]\buildrel{L}\over\otimes LSP^2(A,n)\right)\to
L_{i+1}SP^3(A,n)\to \\
L_{i}\EuScript L^3(A,n) \to
\pi_{i}\left(A[n]\buildrel{L}\over\otimes LSP^2(A,n)\right)\to
L_{i}SP^3(A,n)\to \dots\label{longsequ}
\end{multline}
In addition, the isomorphism \eqref{declie} implies  that
\begin{equation}\label{lie3}
    L_{i}Y^3(A,n)\simeq L_{i+3}\,\EuScript L^3(A,n+1),\ \text{for all $
i\geq 3$}.
\end{equation}

\subsection{The derived functors $L_i\EuScript L^3(A)$}
The Koszul  sequence \eqref{y-kos} associated for $n=3$ to a flat
2-term resolution $f:L \to M$ of an abelian group $A$ may be written
as follows:
$$
0\to \EuScript L_s^3(L)\to L\otimes M\otimes L
\buildrel{\delta}\over\to L\otimes M\otimes M\to \EuScript L^3(M)
\to \EuScript L^3(A) \to 0
$$
where
$$
\delta(l\otimes m\otimes l')=l\otimes m\otimes f(l')+l'\otimes
f(l)\otimes m-l\otimes f(l')\otimes m,\ \   l,l'\in L, m\in M.
$$
The three middle terms constitute a   complex which represents the object
$L\EuScript L^3(A)$ of the derived category.  In particular, if
one considers the resolution $\mathbb Z\buildrel{m}\over\to
\mathbb Z$ of $\mathbb Z/m$, one finds  that
$$
L_i\EuScript L^3(\mathbb Z/m)=\begin{cases} \mathbb Z/m,\ i=1,\\
0,\ i\neq 1\end{cases}
$$
By \eqref{syder}, the natural transformation  $\Tor(\EuScript
S_2(A),A)\to \EuScript S_3(A)$ is an epimorphism. More generally,
the exact sequence (\ref{longsequ}) provides the following
description of the derived functors of $\EuScript L^3$:
\begin{align*}
& L_2\EuScript L^3(A)=ker\{\Tor(\EuScript S_2(A),A)\to \EuScript
S_3(A)\}\\
& 0\to ker\{\EuScript S_2(A)\otimes A\to L_1SP^3(A)\}\to
L_1\EuScript L^3(A)\to \Tor(SP^2(A),A)\to 0\\
& L_i\EuScript L^3(A)=0,\ i>2.
\end{align*}
\subsection{The derived functors $L_i\EuScript L^3(A,1)$.}\label{l3a1}
Similarly, the short exact sequence \eqref{sp3and1} derives to a
long exact sequence
\begin{multline}
\dots\to L_{i+1}\EuScript L^3_s(A,n)\to
\pi_{i+1}\left(A[n]\buildrel{L}\over\otimes
L\Lambda^2(A,n)\right)\to L_{i+1}\Lambda^3(A,n)\to
\\
L_{i}\EuScript L^3_s(A,n) \to
\pi_{i}\left(A[n]\buildrel{L}\over\otimes
L\Lambda^2(A,n)\right)\to L_{i}\Lambda^3(A,n)\to
\dots\label{slongsequ}
\end{multline}
analogous to \eqref{longsequ}. For $n= 0$ this reduces to the
exact sequence
\begin{multline}\label{longex}
0\to L_2 Y^3(A)\to \Tor(\Omega_2(A),A)\to
\Omega_3(A)\to\\
L_1Y^3(A)\to \pi_1\left(L\Lambda^2(A)\buildrel{L}\over\otimes A
\right)\to L_1\Lambda^3(A)\to 0
\end{multline} This exact sequence is consistent with the results of
\cite{Breen} prop. 6.15, and with the presentation \eqref{derlam}
of the groups $L_i\Lambda^n(A)$. The latter implies that the
composite arrow
\begin{equation}
\label{surjl1lam3}
 L_1\Lambda^2(A) \ot A \to \pi_1\left(L\Lambda^2(A)
\buildrel{L}\over\otimes A\right)
 \to L_1\Lambda^3(A)\end{equation}
 is an
 epimorphism. The K\"{u}nneth formula, together with the exact
 sequence \eqref{longex},  determines a 3-step filtration of
 $L_1Y^3(A)$. Taking into account
the isomorphism (\ref{lie3}) for $n=0$, we obtain the following
description of the derived functors $ L_i\EuScript L^3(A,1)$:
\begin{align}
\label{derl3} & L_3\EuScript L^3(A,1)\simeq Y^3(A)\\\notag
 gr_1&L_4\EuScript L^3(A,1)\simeq gr_1L_1Y^3(A)\simeq
\text{coker}\{\Tor(\Omega_2(A),A)\to \Omega_3(A)\}\\\notag
 gr_2&L_4\EuScript L^3(A,1)\simeq gr_2L_1Y^3(A)\simeq
\ker\{\Omega_2(A)\otimes A\to L_1\Lambda^3(A)\}\\\notag
 gr_3&L_4\EuScript L^3(A,1)\simeq gr_3L_1Y^3(A)\simeq
\Tor(\Lambda^2(A),A)\\ & L_5\EuScript L^3(A,1)\simeq
L_2Y^3(A)\simeq \ker\{\Tor(\Omega_2(A),A)\to \Omega_3(A)\}\notag
\end{align}

\bigskip

\begin{remark}\label{ourremark}
{\rm The natural map
\begin{equation}
\label{defmap1} \mathrm{Tor}(\Omega_2(A), A) \simeq
 \pi_2\left(L\Lambda^2A \buildrel{L}\over\otimes A\right)
\to \Omega_3(A)  \end{equation}
 in the exact sequence \eqref{longex} is in general neither injective
 nor surjective.
 This can be seen by  considering the generators
$\om_i^h(x)$   \eqref{torprod} of the groups $\Om_n(A)$. We know
by \cite{Breen} (5.14) that the diagram
\begin{equation}
\label{diagom3}\xymatrix{ \Ga_2({}_hA) \ot A \ar[d] \ar[r]
\ar[d]_{\lambda^2_h \otimes 1}&
 \Gamma_3({}_hA) \ar[d]^{\lambda^3_h} \\
\mathrm{Tor}(\Omega_2(A),\, A) \ar[r] & \Om_3(A) }\end{equation}
is commutative. It follows from  the relation $\ga_2(x) x =
3\,\ga_3(x)$ in $\Ga_3(A)$ that, with the notation introduced in
\eqref{torprod}, the corresponding relation
\[\omega^h_2(x) \ast x = 3\,\omega^h_3(x)\]
is satisfied in $\Om_3(A)$ for all $x \in {}_hA$. In particular,
this implies that the arrow \eqref{defmap1} is trivial for $h= 3$
and $A= \mathbb{Z}_3$. Moreover, it is asserted in \cite{BB} that
\begin{equation}\label{buth} L_1Y^3(\mathbb Z/3)=\mathbb Z/9.\end{equation}
We refer to proposition \ref{ly3z3}  for a proof by our methods  of this
assertion.}
\end{remark}

\subsection{The derived functors $L_i\EuScript L^3(A,2)$.} The d\'ecalage
isomorphisms (\ref{bq1}) and K\"{u}nneth formula yield the
following description of the  groups
$\pi_r\left(LSP^j(A,2)\buildrel{L}\over\otimes A[2]\right)$:
\begin{align*}
& \pi_{2j+2}\left(LSP^j(A,2)\buildrel{L}\over\otimes
A[2]\right)\simeq\Gamma_j(A)\otimes A\\
& 0\to L_{i-2}\Gamma_j(A)\otimes A\to
\pi_{2j+i}\left(LSP^j(A,2)\buildrel{L}\over\otimes A[2]\right)\to
\Tor(L_{i-3}\Gamma_j(A),A)\to 0,\\ & \ \ \ \ \ \ \ \ \ \ \ \ \ \ \
\ \ \ \ \ \ \ \ \ \ \ \ \ \
\ \ \ \ \ \ \ \ \ \ \ \ \ \ \ \ \ \ \ \ \ \ \ \ \  \ \ \ \ \ \ \ \ \ \ \ \ \ \ \ \ \ \ \  i=3,\dots, j+1\\
& \pi_{3j+2}\left(LSP^j(A,2)\buildrel{L}\over\otimes
A[2])\right)\simeq \Tor(L_{j-1}\Gamma_j(A),A)
\end{align*}

From \eqref{exsp2} and \eqref{deflw3}
 we have the following commutative diagram, a prolongation of \eqref{prolong}:
\begin{equation}
\label{mapcub} { \xymatrix@C=4.5pt@R=7pt{\EuScript S_2(A)\otimes A
\ar@{^{(}->}[r] \ar@{->}[dd]  &
\pi_1\left(LSP^2(A)\buildrel{L}\over\otimes A\right)
\ar@{->}[dd]^q \ar@{->>}[dr] \ar@{->>}[rr] & & \Tor(SP^2(A),A)
\ar@{^{(}->}[dd] \\  & & L_1SP^3(A) \ar@{^{(}-}[dd] \\
R_2(A)\otimes A \ar@{^{(}->}[r] \ar@{->>}[dd]  &
\pi_1\left(L\Gamma_2(A)\buildrel{L}\over\otimes A\right)
\ar@{->}[dr]_(.6){t_1} \ar@{->>}[dd] \ar@{-}[r] & \ar@{->}[r]
 & \Tor(\Gamma_2(A),A) \ar@{->}[dd] \\
& & L_1\Gamma_3(A) \ar@{->>}[dd] \\
\Tor(A,\mathbb Z/2)\otimes A \ar@{^{(}->}[r]  & \text{coker}(q)
\ar@{->}[dr] \ar@{-}[r] & \ar@{->}[r] &
\Tor(A\otimes \mathbb Z/2,A)\\
& & ( \Tor (A,\Z/2)
 \otimes A) \oplus\ \Tor(A,\Z/3) }}
\end{equation}
\medskip
A  diagram chase yields  a  canonical isomorphism
 $\text{coker}(t_1) \simeq\ \Tor(A,\mathbb Z/3). $
We obtain the
following description of a portion of the long exact sequence \eqref{longsequ} for $n=2$:
\begin{equation}\label{g2}
\xymatrix{ R_2(A)\otimes A \ar@{->}[r]^{t_2} & L_1\Gamma_3(A) \ar[r] &
 L_6\EuScript L^3(A,2) \ar@{->}[r] & \Gamma_2(A)\otimes A\ar@{->}[r] &
\Gamma_3(A) \ar@{->>}[r] & A\otimes \mathbb Z/3
}
\end{equation}

\begin{comment}
\begin{equation}\label{g2}
\xymatrix@R=12pt { R_2(A)\otimes A \ar@{->}[d] \ar@{^{(}->}[r] &
\pi_7\left(LSP^2(A,2)\buildrel{L}\over\otimes A[2]\right)
\ar@{->}[d]_{t_2}
\ar@{->>}[r] & \Tor(\Gamma_2(A),A)\\
L_1\Gamma_3(A) \ar@{=}[r]
 & L_7SP^{3}(A,2)\ar@{->}[d] \\
& L_6\EuScript L^3(A,2) \ar@{->}[d]\\
\Gamma_2(A)\otimes A\ar@{->}[d] \ar@{=}[r] & \pi_6\left(LSP^2(A,2)\buildrel{L}\over\otimes A[2]\right) \ar@{->}[d]\\
\Gamma_3(A) \ar@{->}[d] \ar@{=}[r] & L_6SP^3(A,2) \ar@{->}[d] \\
A\otimes \mathbb Z/3 \ar@{=}[r] & L_5\EuScript L^3(A,2)}
\end{equation}
\end{comment}

\bigskip

\noindent Here $t_2 := t_1\circ q[2]$, up to a double d\'ecalage map.
There is a  functorial direct sum decomposition of  the term $L_6\EuScript
L^3(A,2)=L_3Y^3(A,1)$,
 as can be seen  from the following diagram, in which the
vertical arrows are suspension maps:
$$
\xymatrix@R=9pt{\Tor(A,\mathbb Z/3) \ar@{^{(}->}[r] \ar@{->}[d]^(.42){\wr} &
L_6\EuScript L^3(A,2) \ar@{->}[d] \ar@{->>}[r] & \EuScript L^3(A)\\
\Tor(A,\mathbb Z/3) \ar@{->}[r]^{\sim} & L_7\EuScript L^3(A,3)}
$$
(we refer to theorem \ref{derl3-1} below for this description of
$L_7\EuScript L^3(A,3)$). We may prolong   diagram (\ref{g2}) by
the following exact sequence:
$$
0\to L_8\EuScript L^3(A,2)\to \Tor(R_2(A),A)\to R_3(A)\to
L_7\EuScript L^3(A,2)\to \text{ker}(t_2)\to 0.
$$
The diagram
$$
\xymatrix@R=9pt{\Tor(\EuScript S_2(A),A) \ar@{^{(}->}[r] \ar@{->>}[d] &
\Tor(R_2(A), A)
\ar@{->}[r] \ar@{->}[d] & \Tor_2(A,A,\mathbb Z/2) \ar@{=}[d] \ar@{->>}[r] & \text{ker}\{\EuScript S_2(A)\otimes A\to R_2(A)\otimes A\}\\
\EuScript S_3(A) \ar@{^{(}->}[r] & R_3(A) \ar@{->>}[r]\ &
\Tor_2(A,A,\mathbb Z/2) }
$$
then implies that
$$
\text{ker}\{\Tor(\EuScript S_2(A),A)\to \EuScript
S_3(A)\}=\text{ker}\{\Tor(R_2(A),A)\to R_3(A)\}
$$
and
$$
\text{coker}\{\Tor(R_2(A),A)\to R_3(A)\}=\text{ker}\{\EuScript
S_2(A)\otimes A\to R_2(A)\otimes A\}.
$$
 Taking once more into account the d\'ecalage
 isomorphisms  \eqref{lie3}, this provides a complete description of
 the functors $L_i \Le^3(A,2)$:
\begin{align}\label{l3a2}
& L_5\EuScript L^3(A,2)=L_2Y^3(A,1)=A\otimes \mathbb Z/3\\ \notag
& L_6\EuScript L^3(A,2)=L_3Y^3(A,1)= \EuScript L^3(A)\oplus \Tor(A,\mathbb Z/3)\\
\notag
 gr_1&L_7\EuScript L^3(A,2)=gr_1L_4Y^3(A,1)=\text{ker}\{\EuScript S_2(A)\otimes A\to R_2(A)\otimes A\}\\ \notag
 gr_2&L_7\EuScript L^3(A,2)=gr_2L_4Y^3(A,1)= \ker\{R_2(A)\otimes A\to
L_1\Gamma_3(A)\}\\ \notag
 gr_3&L_7\EuScript L^3(A,2)=gr_3L_4Y^3(A,1)=\Tor(\Gamma_2(A),A)\\ \notag
& L_8\EuScript L^3(A,2)=L_5Y^3(A,1)=L_2\EuScript L^3(A).
\end{align}
\noi For all other values of $i, \  L_i\Le^3(A,2)= 0$.

\vspace{.5cm} As an illustration of these results, we will now
give an explicit description of the isomorphism
\begin{equation}\label{isoder1}
A\otimes \mathbb Z/3\to L_2Y^3(A,1)=L_2\EuScript L_s^3(A,1)
\end{equation} occuring in the first equation of \eqref{l3a2}, even
though this will not be used in the sequel. Consider the
simplicial model \eqref{lowdeg}   of  $L\EuScript L_s^3(A,1)$
associated to a free resolution \eqref{res11} of $A$.
 The isomorphism (\ref{isoder1}) is induced
by the map
$$
A\otimes \mathbb Z/3\to \EuScript L_s^3(L\oplus s_1(M)\oplus
s_0(L))/\partial_0(\cap_{i=1}^3\partial_i)
$$
defined, for a chosen lifing $a$ to $ M$ of $\bar a\in A\otimes \mathbb
Z/3$,   by
$$
\bar a\mapsto \{s_1(a), s_0(a), s_1(a)\}.
$$
In order for this map to be well-defined, we must verify that
$$3\{s_1(a),s_0(a),s_1(a)\}\in \partial_0(ker(\partial_1)\cap ker(\partial_2)\cap
ker(\partial_3)).$$ This is true since  the element
\begin{multline*}
\eta=3\{s_2s_0(a),s_1s_0(a),s_2s_0(a)\}-\{s_2s_1(a),s_1s_0(a),s_2s_0(a)\}+\{s_2s_1(a),s_2s_0(a),s_1s_0(a)\}\in\\
\EuScript L_s^3(s_0(A_1)\oplus s_1(A_1)\oplus s_2(A_1)\oplus
s_1s_0(A_0)\oplus s_2s_0(A_0)\oplus s_2s_1(A_0))
\end{multline*}
satisfies the equations  $\partial_i(\eta)=0,\ i=1,2,3$ and
$\partial_0(\eta)=3\{s_1(a),s_0(a),s_1(a)\}$.
\subsection{The derived functors $L_i \EuScript {L}^3(A,n) \ \text{for} \  n\geq 3$}
In each of the three following commutative diagrams, the exactness
of the upper short exact sequence follows from proposition
\ref{derSP2} and the exactness of the lower one from theorem
\ref{sp3an}. For $n \geq 3$ odd:
$$
\xymatrix@R=11pt{ \pi_1\left(\Lambda^2(A)\buildrel{L}\over\otimes
A\right)\ar@{^{(}->}[r] \ar@{->}[d] &
\pi_{3m+1}\left(LSP^2(A,n)\buildrel{L}\over\otimes A[n]\right)
\ar@{->}[d] \ar@{->>}[r] & \Tor_2(A,A,\mathbb Z/2) \ar@{=}[d]\\
L_1\Lambda^3(A) \ar@{^{(}->}[r] & L_{3n+1}SP^3(A,n) \ar@{->>}[r] &
\Tor_2(A,A,\mathbb Z/2)}
$$
$$
\xymatrix@R=11pt{ \Lambda^2(A)\otimes A \ar@{^{(}->}[r] \ar@{->}[d]
& \pi_{3n}\left(LSP^2(A,n)\buildrel{L}\over\otimes A[n]\right)
\ar@{->}[d] \ar@{->>}[r] & \Tor_1(A,A,\mathbb Z/2) \ar@{=}[d]\\
\Lambda^3(A) \ar@{^{(}->}[r] & L_{3n}SP^3(A,n) \ar@{->>}[r] &
\Tor_1(A,A,\mathbb Z/2)}
$$
and for $n >3$ even:
$$
\xymatrix@R=11pt{ \Gamma_2(A)\otimes A \ar@{^{(}->}[r] \ar@{->}[d]
& \pi_{3n}\left(LSP^2(A,n)\buildrel{L}\over\otimes A[n]\right)
\ar@{->}[d] \ar@{->>}[r] & \Tor_2(A,A,\mathbb Z/2) \ar@{=}[d]\\
\Gamma_3(A) \ar@{^{(}->}[r] & L_{3m}SP^3(A,n) \ar@{->>}[r] &
\Tor_2(A,A,\mathbb Z/2)}
$$
The sequence (\ref{longsequ}) therefore determines  the following
exact sequences:

 \medskip

\noindent{\it Case I: $n$ odd $\geq 3$}
\begin{multline*}
0\to L_{3n+2}\EuScript L^3(A,n)\to \Tor(\Omega_2(A),A)\to \Omega_3(A)\to\\
L_{3n+1}\EuScript L^3(A,n)\to
\pi_1\left(\Lambda^2(A)\buildrel{L}\over\otimes A\right)\to L_1\Lambda^3(A)\to\\
L_{3n}\EuScript L^3(A,n)\to \Lambda^2(A)\otimes A\to
\Lambda^3(A)\to L_{3n-1}\EuScript L^3(A,n)
\end{multline*}

\noindent{\it Case II: $n$ even}
\begin{multline*}
0\to L_{3n+2}\EuScript L^3(A,n)\to \Tor(R_2(A),A)\to R_3(A)\to\\
L_{3n+1}\EuScript L^3(A,n)\to
\pi_1\left(L\Gamma_2(A)\buildrel{L}\over\otimes A\right)\to L_1\Gamma_3(A)\to\\
L_{3n}\EuScript L^3(A,n)\to \Gamma_2(A)\otimes A\to \Gamma_3(A)\to
L_{3n-1}\EuScript L^3(A,n)
\end{multline*}

We may now summarize this discussion as follows:
\begin{theorem}
\label{derl3-1}
\noindent{\it Case I: $n$ odd}
$$ L_i\EuScript L^3(A,n)=\begin{cases} \ker\{\tor(\Omega_2(A),
A)\to
\Omega_3(A)\},\ i=3n+2,\\
gr_1L_{3n+1}\EuScript
L^3(A,n)=\mathrm{ker}\{\pi_1\left(L\Lambda^2(A)\buildrel{L}\over\otimes
A\right)\to L_1\Lambda^3(A)\}\\
gr_2L_{3n+1}\EuScript
L^3(A,n)=\mathrm{coker}\{\Tor(\Omega_2(A),A)\to \Omega_3(A)\}
\\Y^3(A),\ i=3n\\
 A\otimes \mathbb Z/3,\ n+3\leq i< 3n-1,\ i\equiv n+3 \mod 4\\
\tor(A,\mathbb Z/3),\ n+4\leq i\leq 3n-1,\ i\equiv n\mod
4\end{cases}
$$
\noindent and  $L_i\Le^3(A, n) = 0$ for all other values of $i$.

\bigskip
\noindent{\it Case II: $n$ even}
$$
L_i\EuScript L^3(A,n)=\begin{cases} \ker\{\Tor(R_2(A), A)\to
R_3(A)\},\ i=3n+2\\ gr_1L_{3n+1}\EuScript L^3(A,n)=
\ker\{\pi_1\left(L\Gamma_2(A)\buildrel{L}\over\otimes A\right)\to
L_1\Gamma_3(A)\}\\
gr_2L_{3n+1}\EuScript L^3(A,n)=\mathrm{coker}\{\tor(R_2(A),A)\to
R_3(A)\}.\\
\Tor(A,\mathbb Z/3)\oplus \EuScript L^3(A),\ i=3n\\
 A\otimes \mathbb Z/3,\ n+3\leq i\leq 3n-1,\ i\equiv n+ 3 \mod 4\\
\Tor(A,\mathbb Z/3),\ n+4\leq i<3n-1,\ i\equiv n\mod 4
\end{cases}
$$
\noindent  and  $L_i\Le^3(A, n) = 0$ for all other values of $i$.
\end{theorem}

\noindent  Note that in the computation of $L_{3n}\EuScript L^3(A,n)$ for
$n$ odd, we
 relied on the surjectivity of the map \eqref{surjl1lam3}.

\bigskip

\noindent{\bf Example.} The previous discussion shows in
particular that
\begin{equation}\label{exa3}
L_i\EuScript L^3(\mathbb Z/3,5)=\begin{cases} \mathbb Z/9,\
i=16,\\ \mathbb Z/3,\ i=8,9,12,13,17\\ 0,\
\text{otherwise}\end{cases}
\end{equation}
A  functorial description of some of these groups will be given in  lemma  \ref{derj3n}.
\bigskip

 For a given a free abelian group $A$ and an
integer $n\geq 2$ the composite map
$$
\EuScript L^n(A)\to \otimes^n(A)\to \EuScript L^n(A)
$$
is simply  multiplication by $n$ (\cite{Schlesinger:66} proposition 3.3). It follows that for an odd prime
$p$, the derived functors $L_i\EuScript L^p(\mathbb Z,n)$ are
$p$-groups (see (\ref{primedec})). The torsion part of  the
derived functors can be usually be  determined by  general
arguments. Recall that by  Bousfield \cite{Bou1}, corollary 9.5,
if
 $T:{\sf Ab\to Ab}$ is  a functor of finite degree which preserves
 direct limits, then  $L_qT(A,n)$ is a torsion group for every
abelian group $A$, unless $q$ is divisible by $n$.

\bigskip

The $p$-components of the derived functors of  $\EuScript
L^p$ and $J^p$ are connected by the following relation
(\cite{Schlesinger:66} proposition 4.7): for every prime $p$,
there are natural isomorphisms
$$
\pi_i\left(\mathbb Z/p\buildrel{L}\over\otimes L\EuScript
L^p(\mathbb Z,n)\right)\simeq \pi_i\left(\mathbb
Z/p\buildrel{L}\over\otimes LJ^p(\mathbb Z,n)\right),\ i\geq 0,\
n\geq 2.
$$
However the formulas for the full  derived functors $L_iJ^p(\mathbb
Z,n)$ are more complicated than those for the functors  $L_i\EuScript
L^p(\mathbb Z,n)$  (\ref{primedec}). For example, we know by
theorem \ref{derl3-1} that
\begin{align*}
& L\EuScript L^5(\mathbb Z,3)\simeq K(\mathbb Z/5, 10)
\end{align*}
so  that by \eqref{derlam2ga2}
$$
LJ^2(\mathbb Z,3)\buildrel{L}\over\otimes LJ^3(\mathbb Z,3)\simeq
K(\mathbb Z/3,12).
$$
On the other hand $ LJ^3(\mathbb Z,3)\simeq K(\mathbb Z/3,6)$ by
\eqref{primedec}, and the values of the derived functors
$L_iJ^5(\mathbb Z,3)$ now follow from those  of  $\EuScript
L^5(\mathbb Z,3)$ and the Curtis decomposition of $\EuScript L^5$.
One finds:
$$
L_iJ^5(\mathbb Z,3)\simeq \begin{cases} \mathbb Z/3,\ i=13\\
\mathbb Z/5,\ i=10\\ 0,\ i\neq 10,13\end{cases}
$$
One can  compute the groups $L_iJ^5(\mathbb Z,n)$  for a general
$n$ by similar methods . One finds that  $L_iJ^5(\mathbb Z,n)$
  is isomorphic to $
L_i\EuScript L^5(\mathbb Z,n)$ for $\ i\geq 1$ and even $n$, and
that $L_iJ^5(\mathbb Z,n)$ contains only 3-torsion and 5-torsion
elements whenever  $n$  is odd. A similar computation detects  a
non-trivial 11-torsion element in $L_{29}J^{13}(\mathbb Z,3)$,
whereas the corresponding groups $L_i\EuScript L^{13}(\mathbb Z,3)$  are
13-torsion for all $i$.

\section{Derived functors of composite functors}

\label{filder}

Consider a pair of composable functors
\[\xymatrix{ \mathcal{A} \ar[r]^G & \mathcal{B} \ar[r]^F &
  \mathcal{C} }\]
between abelian categories, and in which the categories
$\mathcal{A}$ and $\mathcal{B}$ have enough projectives. In
addition, we assume that $G(A)$ is of finite projective dimension
for each object $A \in \mathcal{A}$.
 When these functors are additive, the
composite functor spectral sequence \cite{weibel} \S 5.8
describes the derived functors of the composite functor $G\circ F$
in terms of those of $F$ and  $G$, under the condition that the
objects $G(P)$ are $F$-acyclic for any projective object $P$ of
$\mathcal{A}$ (we will refer to  this as  the
$F$-acyclicity hypothesis).  We  will
now carry out a similar discussion   when  $G$ and $F$ are no longer additive,
 in which case  chain complexes must be  replaced by
simplicial abelian groups.

\bigskip

 Let $P_\ast$ be  projective  resolution of an object $A \in
 \mathcal{A}$. Following the notations of \cite{weibel} \S 5.7, we
  construct a
 Cartan-Eilenberg  resolution  $\mathbb{P}_{\ast. \ast}$    of the
 simplicial object $G(P_\ast)$,
with $\mathbb{P}^B_{p, \ast}$ ({\it resp.}
$\mathbb{P}^Z_{p,\ast},\  \text{{\it resp.}} \ \mathbb{P}^H_{p,\ast}$) the chosen
projective resolution of $B_pG(P_\ast)$ ({\it resp.}
 $Z_pG(P_\ast),\  \text{\it resp. } L_pG(A)$). By the Dold-Kan correspondence this yields in
particular a  projective bisimplicial resolution of $LG(A))$ which we will also
denote by  $\mathbb{P}_{\ast. \ast}$, as well as a corresponding
projective  simplicial model   $\mathbb{P}^ H_{p,\ast}$   for
the Eilenberg-Mac Lane spaces $K(L_pG(A),\,p)$.
 The two filtrations on the complex associated to
the bisimplicial object
$F\mathbb{P}_{\ast. \ast}$ determine a
pair of  spectral sequences with common abutment $\pi_n(LF\circ
LG(A))$. The  initial terms of the first of these
 are given by \bee \label{ss2}
 E^2_{p,q} = L_{p+q}F\,\,(\bigoplus_{\sum_i q_i = q} \! \!K(L_{q_i}G(A),q_i))
\ee as in \cite{B-S}.
When the functor $F$ is of finite degree, we may  decompose this initial term according to
cross-effects of $F$ \cite{DoldPuppe} \S 4.18, so that the spectral sequence can be expressed
as:
\bee \label{sseq} E^2_{p,q} =  \bigoplus_{r\geq 1}
\bigoplus_{q_1
  + \ldots + q_r = q}
 L_{p+q}F_{[r]}(L_{q_1}G(A),q_1| \ldots |
L_{q_r}G(A),q_r)  \Longrightarrow \pi_{p+q}((LF\circ LG)(A))\,.
\ee The $F$-acyclicity hypothesis, which here asserts that $L_iF(G(P))=0$ for
any projective  abelian group $P$ and all $i>0$, implies that  the morphism
 $(LF\circ G)(P)
\la FG(P)$ is  a quasi-isomorphism for any projective object $P$ in $A$,
  and so is  the  induced
map \bee \label{qis} (LF \circ LG)(A) \la L(FG)(A)\,.\ee
 The   spectral sequence \eqref{sseq} can now be
written  as: \bee \label{sseq1} E^2_{p,q} =  \bigoplus_{r\geq 1}
\bigoplus_{q_1
  + \ldots + q_r = q}
 L_{p+q}F_{[r]}(L_{q_1}G(A),q_1| \ldots |
L_{q_r}G(A),q_r)  \Longrightarrow L_{p+q}(FG)(A)\,. \ee

\bigskip

Replacing the object  $A$ by the shifted derived category object
$A[n]$, in other words by the Eilenberg-Mac Lane  object $K(A,n)$,
we may now  compute under the same hypotheses the derived functors \linebreak
$L_r(FG)(A,n)$ for all $n$. The $F$-acyclicity hypothesis implies
inductively that the quasi-isomorphism \eqref{qis} determines a
quasi-isomorphism \bee \label{qis1} (LF \circ LG)(A,n) \la
L(FG)(A,n) \ee
 for all $n \geq 0$,  since we can choose  as a simplicial model
 for $K(A,n)$ the bisimplicial model
\[ \ldots \la K(A^3,n-1) \la  K(A^2,n-1)  \la  K(A,n-1) \la \{e\} \]
and work componentwise. Since  no change  is
necessary  in the discussion of   the  spectral sequence \eqref{sseq} when passing
from the case $n=0$ to the general situation, we finally obtain
for any  positive $n$, when $F$ is of finite degree and the
$F$-acyclicity hypothesis is satisfied,
 a functorial  spectral sequence:
\bee \label{sseq2} E^2_{p,q} =  \bigoplus_{r\geq 1}
\bigoplus_{q_1
  + \ldots + q_r = q}
 L_{p+q}F_{[r]}(L_{q_1}G(A,n),q_1| \ldots |
L_{q_r}G(A,n),q_r)  \Longrightarrow L_{p+q}(FG)(A,n)\,. \ee
\bigskip

We now restrict ourselves to a special  case, that  in which $F$
and $G$ are endo-functors on the category of abelian groups (or
more generally   the category of $R$-modules, with $R$ a
principal ideal  domain or even a hereditary  ring).
By construction, the  total complex of $\mathbb{P}_{\ast, \ast}$ and the
complex  $\oplus_q \,\mathbb{P}^H_{q,\ast}$  are both projective and have
as homology
$\oplus_q L_qG(A)[q]$, viewed as a complex with trivial differentials. It follows as in \cite{Dold1} \S II.4 that this identification of their
homology may be realized   by a chain homotopy equivalence between the
complexes, which in turn induces a simplicial homotopy equivalence between the
corresponding simplicial groups  $ \mathbb{P}$ and  $\oplus_q\,
\mathbb{P}^H_{q,\ast} $. The induced homotopy equivalence between
$F(\mathbb{P})$ and $F(\mathbb{P}^H_{q,\ast}) $
makes it clear that in this case the $E^2$ term of the spectral
sequence \eqref{sseq2} is (non-canonically) isomorphic to its abutment. It follows that the spectral sequence
 degenerates  at the $E^2$ level, so that this proves the following proposition:
\begin{prop}
\label{sseqprop} Let $F$ and $G$ be a pair of endofunctors on
the category of abelian groups, with $F$ of finite degree. Suppose  that for
any projective  abelian group
  $P$,  $L_qF(G(P)) =0$
whenever   $q >0$.
Then the spectral sequence \eqref{sseq2} degenerates at $E^2$, and the graded components associated to the filtration on the abutment  $L_m(FG)(A)$   of  the
spectral sequence are  described by the formula: \bee
\label{propiso}
 gr_p L_{p+q}(FG)(A,n) \simeq  \bigoplus_{r\geq 1}  \bigoplus_{q_1
  + \ldots + q_r = q}
 L_{p+q}F_{[r]}(L_{q_1}G(A,n),q_1| \ldots |
L_{q_r}G(A,n),q_r) \ee
 \end{prop}
When $F$ is one of the  functor $SP^s,\Lam^s$ or $\Ga_s$,  such an
assertion may  also be deduced, under the $F$-acyclicity
hypothesis, from the formula  \cite{illusie} V (4.2.7) of Illusie.

\subsection{The derived functors of $\Lambda^2\Lambda^2$ }

As an illustration of proposition \ref{sseqprop}, we will now
compute the derived functors of the functor
$L_i(\Lambda^2\Lambda)(A,n)$ for all
 for  $n=0,1,2$. Such results are of interest to us, since $\Lambda^2\Lambda^2$
 is the first composite functor arising in the decomposition
 \eqref{curdec} of the
 Lie functors $\Le^n A$.

\bigskip

We know by \eqref{derlam2ga2}  that
\begin{equation}
\label{lilam2}
  L_i\Lambda^2(A) =  \begin{cases} \Om_2(A),\ i=1\,\\
    \Lam^2(A),\ i=0 \\ 0, \ \ \ \ \ \ \ i\neq 0,1
\end{cases}  \qquad  \qquad L_i\Lambda^2(A,1)=\begin{cases} R_2(A),\ i=3\\
\Gamma_2(A),\ i=2,\\ 0, \ \ \ \ \ \   \      i\neq 2,3\end{cases}
\qquad
\end{equation}
\begin{equation*}
\label{lilam2a}L_i\Lambda^2(A,2)=\begin{cases}\Omega_2(A),\ \ \ \  i=5\\
\lam^2(A) ,\ i=4\\ A\otimes \mathbb Z/2,\ \ i=3\\ 0,\ \ \ \ \ \ \
\ \ \ \ \:i\neq 3,4,5\end{cases} \quad \quad
 L_i\Lambda^2(A,3) =
\begin{cases}
R_2(A),\ \ \ \ i=7\\
\Ga_2(A), \ \ \ \  i= 6\\
\tor(A,\Z/2),\ i=5\\
A \ot \Z/2,\ \ \ i=4 \\
0\ \ \ \ \ \ \ \ \ \ \ \ i \neq 4,5,6,7.
\end{cases}
\end{equation*}

\begin{equation*}
\label{lilam2b}
 L_i\Lambda^2(A,4) =
\begin{cases}
\Om_2(A), \   i=9\\
\lam^2(A), \   i= 8\\
\tor(A,\Z/2), \ i = 6\\
A \ot \Z/2, \  i = 5,7\\
0 \ \ \ \ \text{otherwise}
\end{cases} \qquad \qquad L_i\Lam^2(A,5) =
\begin{cases}
R_2(A), \  i = 11 \\
\Gamma_2(A), \ i= 10\\
\tor(A,\Z/2), \ i = 7,9 \\
A \ot \Z/2, \ i = 6,8\\
0 \ \ \ \ \text{otherwise}.
\end{cases}
\end{equation*}
with $\lam^2(A)$ defined in \eqref{newfunct}. Proposion \ref{sseqprop}
 yields the following table for the
functors $L_i\Lambda^2 \Lambda^2(A)$:

\begin{table}[H]
\def\objectstyle{\scriptstyle}
\[
\xymatrix@R=4pt@C=4pt{&&&&&&&&\\
 &\quad \quad 2&&   0  &&  R_2\Om_2(A)  &          \\
&\quad \quad 1& &\Om_2\Lam^2(A) & &   \Ga_2\Om_2(A) \oplus
\tor(\Lam^2(A)
,\Om_2(A))     &  &\\
&p= 0 && \Lam^2\Lam^2(A)&&  \Lam^2A \ot \Om^2(A)   &&\\
\ar[rrrrrr]&&&&&&\\
& & \ar[uuuu] & q= 0 && 1 &  &&& }\] \vspace{.5cm} \caption{
$gr_p(L_{p+q}(\Lam^2\Lam^2)(A))$}
\label{tab:2}
\end{table}

The functors $L_n(\Lam^2\Lam^2)(A)$ can be read off from the
$p+q=n$ line  of this table. In particular, there is a
(non-naturally split) short exact sequence
\[ 0 \la \Lam^2(A) \ot \Om_2(A) \la L_1(\Lam^2\Lam^2)(A) \la
\Om_2\Lam^2(A) \la 0\, \] which may be viewed as a symmetrized
version of a K\"{u}nneth formula for $\pi_1\left(\Lam^2(A) \lot
\Lam^2(A)\right)$.

\bigskip

We now pass to the derived functors $L_n(\Lam^2\Lam^2(A,1))$.
 The corresponding table of values  of these derived functors may be
 read off from  the values \eqref{lilam2} and \eqref{lilam2a} of the
 derived functors of $\Lam^2$:

\begin{table}[H]
\def\objectstyle{\scriptstyle}
\[
\xymatrix@R=4pt@C=4pt{
&\quad \quad &&&&\\
&\quad \quad  4&&0&R_2R_2(A)&\ \ \ \ 0\ \ \ \ &0\\
&\quad \quad  3&&\Om_2\Ga_2(A)&\Ga_2R_2(A)&0&0\\
&\quad \quad  2&&
\lam^2\Ga_2(A)
&\tor(R_2(A) ,\Z/2)&0&0\\
&\quad \quad  1&&\Ga_2(A) \ot \Z/2&R_2(A) \ot \Z/2&0&\tor(\Ga_2(A),R_2(A))\\
& p = 0   &&0&0&0&\Ga_2(A) \ot R_2(A)\\
\ar[rrrrrrr]&&&&&&&\\
& & \ar[uuuuuuu] &\hspace{-.7cm} q =
  2 &3&4&5&
}\] \vspace{.5cm} \caption{ $gr_p(L_{p+q}(\Lam^2\Lam^2)(A,1))$}
\label{tab:3}
\end{table}

When one  also takes  into account the values of $L_i\Lambda(A,n)$
 for $n= 4,5$, one finds the
following values for the derived functors of $\Lam^2\Lam^2(A,2)$:
\begin{table}[H]
\def\objectstyle{\scriptstyle}
\[
\xymatrix@R=4pt@C=4pt{
& \quad \quad 6 &&0&0&R_2\Om_2(A)&0&0&0&0&\\
& \quad \quad 5&&0&\Om_2\lam^2(A)&\Ga_2\Om_2(A)&0&0&0&0&\\
& \quad \quad 4 &&R_2(A\ot \Z/2)&\lam^2\lam^2(A)&\tor(\Om_2(A),\Z/2)&0&0&0&0&\\
& \quad \quad 3 &&\Ga_2(A\ot \Z/2)&\lam^2(A) \ot \Z/2&\Om_2(A) \ot \Z/2&0&0&0&0&\\
& \quad \quad 2&&A \ot\Z/2&\tor(\lam^2(A),\Z/2)&\tor(\Om_2(A),Z/2)&0&0&0&0&\\
& \quad \quad 1 &&A\ot \Z/2 &\lam^2A \ot \Z/2&\Om_2(A) \ot
\Z/2&0&\tor(A \ot \Z/2,\lam^2(A))     &\tor(A \ot
\Z/2,\Om_2(A))&\tor(\lam^2(A), \Om_2(A))&\\
& p = 0  &&0&0&0&0&(A\ot \Z/2) \ot \lambda^2(A) &A \ot \Z/2 \ot
\Om_2(A)
&\lam^2(A)\ot \Om_2(A)&  \\
\ar[rrrrrrrrr]&&&&&&&&&&\\
&&\ar[uuuuuuuu] & q= 3  &4&5&6&7&8&9& }\] \vspace{.5cm}
\caption{$gr_p(L_{p+q}(\Lam^2\Lam^2)(A,2))$}
\label{tab:4}
\end{table}

\subsection{The derived functors of $\Lambda^2\Gamma_2$}

We will now carry out a similar discussion for the derived functors of
$\Lambda^2\Gamma_2$.
By \eqref{derga2},
\begin{equation}
\label{derga0}
L_i\Gamma_2(A) = \begin{cases}
R_2(A) , \ i=1 \\
\Gamma_2(A) , \ i= 0 \\
0, \ \ \ \ \ \   \
i\neq 0,1\end{cases}
\quad \quad
 L_i\Gamma_2(A,1)=\begin{cases} \Om_2(A) ,\  \ i=3\\
\lambda^2(A),\ \ i=2,\\  A \ot \Z/2 , \ i=1 \\0, \ \ \ \ \ \   \ \
i\neq 1, 2,3\end{cases}
\end{equation}
 Since $\Gamma_2(P)$ is torsion-free for any
torsion-free
group $P$,  the derived functors of $\Lam^2\Ga_2$ may   be
computed by formula \eqref{propiso}.  The following tables may now be
deduced  from   \eqref{lilam2} and \eqref{derga0}:

\begin{table}[H]
\def\objectstyle{\scriptstyle}
\[
\xymatrix@R=4pt@C=4pt{&&&&&&&&\\
 &\quad \quad 2&&     && R_2R_2(A)   &          \\
&\quad \quad 1& & \Om_2 \Ga_2(A)   & & \Ga_2R_2(A) \oplus  \tor(\Ga_2(A),R_2(A))
   &  &\\
&p= 0 &&\Lam^2\Ga_2(A)  && \Ga_2(A) \ot R_2(A)  &&\\
\ar[rrrrrr]&&&&&&\\
& & \ar[uuuuu] & q= 0 &\quad \quad & 1 &  &&& }\] \vspace{.5cm} \caption{
$gr_p(L_{p+q}(\Lam^2\Gamma_2)(A))$}
\label{tab:41}
\end{table}

\begin{table}[H]
\def\objectstyle{\scriptstyle}
\[
\xymatrix@R=4pt@C=3pt{
&&&&&&&&&\\
& \quad \quad 4 & &0&0&R_2\Om_2(A)&0&0&&&\\
& \quad \quad 3 &&0&\Om_2\lam^2(A)&\Ga_2\Om_2(A)&0&0&&&&\\
& \quad \quad 2 &&R_2(A \ot \Z/2)&\lam^2\lam^2(A)&\tor(\Om_2(A), \Z/2)&0&0&&&\\
& \quad \quad 1 &&\Ga_2(A \ot \Z/2)&\lam^2(A) \ot \Z/2&\Om_2(A) \ot
\Z/2 \oplus \tor (A \ot \Z/2, \lam^2(A))&\tor (A \ot \Z/2,
\Om_2(A))&\tor(\lam^2(A), \Om_2(A)) &&&\\
&\ \  p = 0         &&0&0&A \ot \Z/2 \ot \lam^2(A)&A \ot \Z/2 \ot
\Om_2(A)&\lam^2(A) \ot \Om_2(A)&&&  \\
\ar[rrrrrrrr]  &&&&&&&&&&\\
&&\ar[uuuuuuu] & q= 1  &2&3&4&5&&& }\] \vspace{.5cm}
\caption{$gr_p(L_{p+q}(\Lam^2\Ga_2)(A,1))$}
\label{tab:42}
\end{table}

The coincidence, up to changes in degree, between certain terms in
table \ref{tab:41} and those in table \ref{tab:3}, and (more strikingly) between
certain terms in table \ref{tab:42} and those in table \ref{tab:4} is
explained by the d\'ecalage isomorphisms \eqref{bq2}  between the derived functors
of $\Gamma_2$ and those of $\Lam^2$.
\section{Derived functors of super-Lie functors}
In view of \eqref{l3sy3},  the d\'ecalage
isomorphisms(\ref{declie}) and formulas (\ref{dfr}) imply that for
$(n\geq 1)$:
\begin{equation}\label{dfr1}
L_{n+k}\EuScript L_s^3(\mathbb Z,n)=\begin{cases} \mathbb Z/3,\
k=4i+1,\
i=0, 1,\dots, [\frac{n-1}{2}]\\
0,\ \text{otherwise}
\end{cases}
\end{equation}
We will now examine the relations between the derived functors of
$\Le^n$ and $\Le^n_s$. For any free simplicial abelian group
$A_*$,   the  maps $\chi_n$ and $\bar\chi_n$ \eqref{defchin}
induce arrows :
\begin{align*}
& \chi_n^*: \pi_m\left(L\EuScript
L_s^n(A_*)\buildrel{L}\over\otimes L\Lambda^n(\mathbb
Z,1)\right)\to \pi_m(L\EuScript L^n(A)\otimes
\mathbb Z[1]),\ m\geq 0\\
& \bar\chi_n^*: \pi_m\left(L\EuScript
L^n(A_*)\buildrel{L}\over\otimes L\Lambda^n(\mathbb Z,1)\right)\to
\pi_m(L\EuScript L_s^n(A)\otimes \mathbb Z[1]),\ m\geq 0
\end{align*}
For $A_*=K(A,m)$, these determine by adjunction pension  maps
\begin{align}
\label{pensionmaps}
& \chi_n^*: L_m\EuScript L_s^n(A,k)\to L_{m+n}\EuScript L^n(A,k+1)\\
& \bar\chi_n^*: L_m\EuScript L^n(A,k)\to L_{m+n}\EuScript
L_s^n(A,k+1).\notag
\end{align}
which may be viewed as generalized d\'ecalage transformations,
even though the maps $\chi_n^\ast$ are  no longer isomorphisms. We
will for this reason refer to such maps as {\it semi-d\'ecalage}
morphisms. Similarly, the pairing $\beta_n$ \eqref{defbetan}
determines a family of pension isomophisms \eqref{declie} which we
now denote $\zeta_n$: \bee \label{def:zetan}
 \xymatrix{L_m Y^n(A,k) \ar[r]^(.4){\zeta_n}  &  L_{m+n} J^n(A,k+1)\,.
} \ee
  Proposition \ref{jfunctor} now  implies the following assertion:
\begin{theorem}\label{semidec}The following  diagram is commutative:
$$
\xymatrix@R=7pt{L_m\EuScript L_s^n(A,k) \ar@{->}[r] \ar@{->}[d]_{\chi_n^*}&
L_m
  Y^n(A,k)
 \ar@{->}[d]^{\wr}_{\zeta_n} \\
L_{m+n}\EuScript L^n(A,k+1) \ar@{->}[r] & L_{m+n} J^n(A,k+1) }
$$
\end{theorem}
We will now consider  the boundary maps:
\begin{align} \label{thetak}
& \theta_m: L_m J^n(A,k)\to L_{m-1}\tilde J^n(A,k),\\
\label{thetabark} & \bar\theta_m: L_m Y^n(A,k)\to L_{m-1}
\widetilde{Y}^n(A,k).
\end{align}
 induced   by the short
exact sequences \eqref{defjntilde} and \eqref{exactyn}. The following proposition is a corollary of theorem \ref{semidec}:
\begin{prop}\label{longlie}  For $m, n\geq 1,$  and an abelian group $A$,
the following   diagram, in which the vertical arrows are
d\'ecalage and semi-d\'ecalage morphisms,  commutes: \bee
\label{diag:semidec} \xymatrix@R=9pt@C=32pt{ L_{m+1}Y^n(A,k)
\ar@{->}[r]^{\bar\theta_{m+1}} \ar@{->}[d]^{\wr}_{\zeta_{n}} &
L_{m} \widetilde{Y}^n(A,k) \ar@{->}[r]
 \ar@{->}[d]^{\psi_n}
& L_m\EuScript L_s^n(A,k)
 \ar@{->}[r]^{\bar{p}_n}
\ar@{->}[d]_{\chi_n^\ast} & L_mY^n(A,k)
\ar@{->}[d]^{\wr}_{\zeta_{n}}
 \\
L_{m+n+1}J^n(A,k+1)\ar@{->}[r]^(.53){\theta_{m+n+1}} & L_{m+n}\tilde
J^n(A,k+1) \ar@{->}[r] & L_{m+n}\EuScript L^n(A,k+1)
\ar@{->}[r]^{p_n} & L_{m+n}J^n(A,k+1)} \ee
\end{prop}

Let $A$ be an abelian group and $n\geq 2$. The  following
diagram, in which the vertical arrows are d\'ecalage maps, is
commutative:
\bee
\label{diag:semi1}
\xymatrix@R=9pt{\pi_{2n}\left(LSP^{n-1}(A,2)\buildrel{L}\over\otimes
A[2]\right)\ar@{->}[r] \ar@{->}[d]_{\wr}& L_{2n}SP^n(A,2)
\ar@{->>}[r] \ar@{->}[d]_{\wr}&
L_{2n-1}J^n(A,2) \ar@{->}[d]_{\wr}\\
\pi_{n}\left(L\Lambda^{n-1}(A,1)\buildrel{L}\over\otimes
A[1]\right)\ar@{->}[r] \ar@{->}[d]_{\wr}& L_{n}\Lambda^n(A,1)
\ar@{->>}[r] \ar@{->}[d]_{\wr}&
L_{n-1}Y^n(A,1) \ar@{->}[d]_{\wr}\\
\Gamma_{n-1}(A)\otimes A\ar@{->}[r] & \Gamma_n(A) \ar@{->>}[r] &
H_0C^n(A)}
\ee
It follows in  particular, by proposition \ref{jean}, that  there
exists  a natural isomorphism
\begin{equation}
\label{corprop2} L_{2n-1}J^n(A,2)\simeq
\bigoplus_{p|n}\Gamma_{n/p}(A\otimes \mathbb Z/p),
\end{equation}
which describes explicitly   the right-hand terms in diagram
\eqref{diag:semi1}.
\subsection{The fourth Lie and super-Lie functors}

We will now discuss certain  derived functors of the functors
$\Le^4$ and $\Le^4_s$.   Recall that by  \eqref{curtis4} and
\eqref{exacty4a},
$$
\tilde J^4(A)\simeq \Lambda^2\Lambda^2(A),\quad \
\widetilde{Y}^4(A)\simeq \Lambda^2\Gamma_2(A).
$$
 for any  free abelian $A$.
By  \eqref{corprop2}, the
right-hand vertical arrows in diagram \eqref{diag:semi1} for $n=4$ are:
\bee
\label{defu}
\xymatrix{  L_7J^4(A,2)\ar[r]^(.48){\sim}  &
  L_3Y^4(A,1)\ar[r]^(.47){\sim}_(.47){u}
  &\Gamma_2(A\otimes \mathbb
Z/2)}.
\ee
\begin{prop}\label{sdiag} For every abelian group $A$, the arrow
$$
\xyma{L_3Y^4(A,1) \ar@{->}[r]^(.48){\bar\theta_3} &
L_2\Lambda^2\Gamma_2(A,1) }
$$
is a natural isomorphism
 between a pair of functors, both
naturally isomorphic to $\Gamma_2(A \otimes \mathbb{Z}/2)$.
\end{prop}

\begin{proof}
Let us first verify that the map $\bar{\theta}_3$ is surjective.
Consider the simplicial model \eqref{lowdeg} of
$L\Lambda^2\Gamma_2(A,1)$ determined by  a flat resolution
\eqref{res11} of $A$.
We define a map
\[ \Ga_2(A \otimes \mathbb{Z}/2) \stackrel{v}{\la}  L_2\Lambda^2\Gamma_2(A,1)\]
explicitly as follows:
$$
\begin{array}{ccc}
\Gamma_2(A\otimes \mathbb Z/2) &\to  & \Lambda^2\Gamma_2(A_1\oplus
s_0(A_0)\oplus s_1(A_0))/\partial_0(ker(\partial_1)\cap
ker(\partial_2)\cap ker(\partial_3))\\
\gamma_2(\bar a)& \mapsto& \gamma_2(s_0(a))\wedge \gamma_2(s_1(a))
\end{array}\ $$
for some lift $a$   to $M$ of $\bar{a} \in A \otimes
\mathbb{Z}/2$. Under the natural transformation $\Lambda^2\Gamma_2
\la \EuScript L_s^4\ $ \eqref{defj}, the image of
$\gamma_2(\bar{a})$ goes to the element
\[
\{s_o(a),s_1(a),s_0(a),s_1(a) \}
\]
in the term $ \EuScript L_s^4(L\oplus s_0(M)\oplus s_1(M)) $ of
the corresponding simplicial model for $L\EuScript L_s^4(A,1)$.
 The
element
\begin{multline*}
\tau: =\{s_1s_0(a),s_2s_0(a),s_1s_0(a),s_2s_0(a)\}-\{s_1s_0(a),s_2s_0(a),s_1s_0(a),s_2s_1(a)\}\in\\
\EuScript L_s^4(s_0(A_1)\oplus s_1(A_1)\oplus s_2(A_1)\oplus
s_1s_0(A_0)\oplus s_2s_0(A_0)\oplus s_2s_1(A_0))
\end{multline*}
 satifies the equations  $\partial_i(\tau)=0,\ i=1,2,3$ and
$$
\partial_0(\tau)=\{s_0(a),s_1(a),s_0(a),s_1(a)\}\in \EuScript
L_s^4(A_1\oplus s_0(A_0)\oplus s_1(A_0)).
$$
It follows that  the map $L_2\Lambda^2\Gamma_2(A,1)\to
L_2\EuScript L_s^4(A,1)$ is trivial so that, by exactness of
the upper line of diagram \eqref{diag:semidec}, the arrow
$\bar{\theta}_3$ is surjective.

\bigskip

We will  now give a more explicit description of  the target
 of $\bar{\theta}_3$.  We have a natural isomorphism
$$
\xymatrix{v:  L_2\Lambda^2\Gamma_2(A,1)\ar[r]^(.55){\sim} & \Gamma_2(A\otimes
\mathbb Z/2)}
$$ since the only non-trivial total degree  2 term  in
table   \ref{tab:42}    is the expression
 $ \Gamma_2(A\otimes \mathbb Z/2)$ in bidegree (1,1). We now have
a pair of
 arrows $u$ \eqref{defu} and $v$, which
provide natural  isomorphisms between  both the  source and target of
$\bar{\theta}_3$ and the group  $\Ga_2(A \otimes \mathbb{Z}/2)$.
 We may now assume that $A$ is
finitely generated. In that case  both  source and
target of the surjective map  $\bar{\theta}_3$ are  finite groups of
the same order, so that $\bar{\theta}_3$ is an isomorphism.
\end{proof}

We know  by the description of homotopy groups of
$L\Lambda^2\Lambda^2(A,2)$
 in table \ref{tab:4}  that there is a natural projection $ L_6\Lambda^2\Lambda^2(A,2) \to \Gamma_2(A\otimes
\mathbb Z/2)$. The following proposition is a consequence of
propositions \ref{longlie} and \ref{sdiag}:

\begin{prop}\label{propo4}
The group $G:=L_6\Lambda^2\Lambda^2(A,2)$ is endowed with a
3-step descending filtration $F^iG\  (1\leq i \leq 3)$ for which the
associated graded components are described by
\[
gr_i G =
\begin{cases}
 \Ga_2(A\ot \Z/2) \quad i= 1\\
\Tor(\lambda^2(A), \Z/2) \quad i= 2\\
\Om_2(A) \ot \Z/2 \quad i=3
\end{cases}
\]
In addition,  the projection $w$ of $G$ on the highest graded component
$gr_1G$ is the map arising from the
edge-homomorphism in the  spectral sequence \eqref{sseq2}
 described in table \ref{tab:4}. This surjection $w$ is split, up to
 isomorphism, by the boundary map $\xymatrix{ L_7J^4(A,2)
   \ar[r]^{\theta_7} \ar[r] & L_6\Lambda^2\Lambda^2(A,2)  }$ provided
 by the decomposition \eqref{curdec} of $\Le^4(A)$.
\begin{comment}
 The boundary
  arrow $\theta_7$ lives in a diagram of short exact sequences

$$
\xyma{\Omega_2(A)\otimes \mathbb Z/2\ar@{^{(}->}[d] &
L_7J^4(A,2)\ar@{->}[rd]^{\simeq} \ar@{^{(}->}[d]^{\theta_7}\\
\text{ker}(w) \ar@{^{(}->}[r] \ar@{->>}[d] &
L_6\Lambda^2\Lambda^2(A,2)
\ar@{->>}[r]^w & \Gamma_2(A\otimes \mathbb Z/2)\\
\Tor(\lambda^2(A),\mathbb Z/2)}
$$
where $w$ is the mapping to $\Ga_2(A\ot \Z/2)$ arising from the
edge-homomorphism in the  spectral sequence \eqref{sseq2}
 described in table \ref{tab:4}.
\end{comment}
\end{prop}

\begin{proof}
The associated graded  terms $\Omega_2(A)\otimes \mathbb Z/2$ and
$\Tor(\lambda^2(A), \mathbb Z/2)$
 in the  line $p+q=6$ of table \ref{tab:4} give us  the required description of
 $F^2G = \text{ker}(w)$.
The previous discussion provides us with a commutative diagram
\[ \xymatrix{\Ga_2(A\ot \Z/2) \ar[r]^{\simeq}_u & L_3Y^4(A,1) \ar[d]^{\simeq}
\ar[r]^(.47){\simeq}_(.47){\bar{\theta}_3} &
L_2(\Lambda^2\Ga_2)(A,1)
\ar@{^{(}->}[d]^{\psi_4}&\\
&L_7J^4(A,2) \ar[r]^(.45){\theta_7} & L_6(\Lambda^2\Lambda^2)(A,2)
\ar[r]^(.53)w & \Ga_2(A\ot \Z/2) }
\]
where the injectivity of the map $\psi_4$ \eqref{diag:semidec}
is obtained by  examining the behavior of the decompositions
\eqref{propiso} of its source and target under d\'ecalage.
It remains to show that the composite map
$$
w\circ\theta_7: L_7J^4(A,2)\to \Gamma_2(A\otimes \mathbb Z/2)
$$
is an isomorphism.  When $A$ is  free
abelian of finite rank,
$L_6\Lambda^2\Lambda^2(A,2)\simeq \Gamma_2(A\otimes \mathbb Z/2)$, so that
the injective map  $$ \psi_4:\Gamma_2(A\otimes \mathbb Z/2)\hookrightarrow
L_6\Lambda^2\Lambda^2(A,2)
$$ is a monomorphism between two finite groups of the same order.
It follows that  the map $w\circ \theta_7$ is an isomorphism
whenever $A$ is free
abelian,  and therefore an epimorphism for an  arbitrary abelian group
$A$.  Returning to the case of an  abelian group $A$ of finite rank,
 we conclude that the epimorphism $w\circ \theta_7$ is an
isomorphism, since  source and target are finite  groups of the same
order. This implies that the corresponding assertion is true  for an arbitrary
 abelian  group.
\end{proof}

In the sequel, we will also need the following result, which
follows since the only non-trivial terms contributed  by the
 Curtis decomposition to
$L_i\EuScript L^4(A,2)$ for $i <7$
are those provided by the derived functors of $\Lam^2\Lam^2$:
\begin{cor}
There group   $L_6\EuScript L^4(A,2)$ is canonically isomorphic to
the direct sum of the two following expressions:
\begin{align*}
& gr_1L_6\EuScript L^4(A,2)= \tor (\lambda^2(A), \Z/2) =
\Tor(\Lambda^2(A),\mathbb Z/2)\oplus
\Tor_2(A,\mathbb Z/2,\mathbb Z/2)\\
& gr_2L_6\EuScript L^4(A,2)=\Omega_2(A)\otimes \mathbb Z/2.
\end{align*}
\end{cor}

\section{Homotopical applications}
\label{main}

\subsection{Moore spaces and the Curtis spectral sequence}
\vspace{.5cm} In this section we will review the Curtis  spectral
sequence, which will be our main tool for homotopical applications
of our theory. Recall that for any abelian group $A$ and $n\geq
2$, a Moore space in degree $n$ is defined to be a simply
connected space $X$ with $H_i(X)\simeq A$ for   $i=n$, and $
\widetilde H_i(X)=0,\ i\neq n$. Such a  Moore space will be
denoted  $M(A,n)$ when in addition an isomorphism $H_n(X)\simeq A$
is fixed. The homotopy type of $M(A,n)$ is determined by the pair
$(A,n)$, since  homology equivalence implies  homotopy
equivalence for simply-connected spaces. When  $A$ is free abelian
with a chosen basis, a Moore space $M(A,n)$ can be constructed  as
a wedge of $n$-spheres, labelled by basis elements of $A$. For an
arbitrary  abelian group $A$ and $n\geq 2$,
 an $n$-dimensional Moore space  is constructed as follows:
 choose a  2-step free  resolution
\eqref{res11}  of $A$   with chosen bases.
 $M(A,n)$
can then be defined as the mapping cone \cite{G-Z}  VI 2 of the induced
  map between the wedges
of spheres $M(L,n)\to M(M,n)$. For any homomorphism of abelian
groups $f: A\to B$, it is possible to construct a map $\phi:
M(A,n)\to M(B,n)$ such that  $H_n(\phi)=f$. However, the
construction of the map $\phi$ is not canonical and the
construction of Moore spaces is non-functorial. The canonical
class in $H^n(M(A,n), A)$ induces a map \bee \label{natmap} M(A,n)
\la K(A,n) \ee which is well-defined up to homotopy.

\bigskip

We will now recall the construction of the Curtis spectral sequence.
 Let $G$ be a simplicial group. The lower central series filtration
on $G$ gives rise to the long exact sequence
\begin{equation*}
\dots\to \pi_{i+1}(G/\gamma_r(G))\to
\pi_i(\gamma_r(G)/\gamma_{r+1}(G))\to
\pi_i(G/\gamma_{r+1}(G))\to \pi_i(G/\gamma_r(G))\to \dots
\end{equation*}
This exact sequence defines a graded exact couple, which gives
rise to a natural spectral sequence $E(G)$ with the initial terms
\begin{align*}
& E_{r,q}^1(G)=\pi_q(\gamma_{r}(G)/\gamma_{r+1}(G))
\end{align*}
and  differentials
\begin{align}
\label{dinm} & d^i_{r,q}: E_{r,q}^i(G)\to E_{r+i,\, q-1}^i(G).
\end{align}

According to
 \cite{Curtis:65}, for $K$  a
connected  simplicial set  and  $G=GK$ the associated  Kan
construction  \cite{May} \S 26,  this spectral sequence $E^i(G)$
converges to $E^\infty(G)$ and $\oplus_rE_{r,q}^\infty$ is the
graded group associated to the filtration on $\pi_q(GK)$ induced
by the lower central series filtration on $K$. Since $GK$ is a
loop group of $K$, this spectral sequence may be written as \bee
\label{curtisgen}
 E_{r,q}^1(K):=\pi_q(\gamma_{r}(GK)/\gamma_{r+1}(GK))
\Longrightarrow \pi_{q+1}(|K|). \ee
 The groups $E^1(G)$ are
homology invariants of $K$.
By the Magnus-Witt isomorphism \eqref{magnuswitt}, the spectral
sequence  can be rewritten as \bee \label{curtisgen1}
 E_{r,q}^1(K) = \pi_q(\Le^r(\widetilde{\Z}K,-1)) \Longrightarrow
 \pi_{q+1}(|K|).
\ee
 since  the abelianization $ GK_{\mathrm{ab}}:= GK/\gamma_2GK $ of
 $GK$ corresponds to
    the reduced chains
 $\widetilde{\Z}K$  on $K$, with degree shifted by 1. When $K = M(A,n)$,   $\widetilde{\Z}K$ corresponds to an
 Eilenberg-Mac Lane space $K(A,n)$ so that the spectral sequence is
 simply  of
 the form
\bee \label{curtisgen2} E^1_{r,q} = L_q \Le^r(A, n-1)
\Longrightarrow \pi_{q+1}(M(A,n))\,. \ee In particular,
\[E^1_{1,q} = \pi_q(K(A,n-1)) = \begin{cases} A, \ q= n-1\\0, \ q \neq
  n-1
\end{cases}
\]
 For a additional information regarding
this spectral sequence, see \cite{Curtis:65}, \cite{MPBook} ch. 5.
\subsection{The 3-torsion of $\pi_n(S^2)$}
As a first illustration of our techniques, we will now discuss the
3-torsion components of the homotopy groups of the sphere $S^2$.
For this,  consider the 3-torsion parts of the various terms in
the  spectral sequence \eqref{curtisgen2}, with   $A = \mathbb{Z}$
and $n=2$: \bee \label{csse1} E_{r,q}^1=L_q\EuScript L^r(\mathbb
Z,1)\Rightarrow \pi_{q+1}(S^2). \ee From now on, we will
  denote by $_pA$ the  $p$-torsion
subgroup  of   an abelian group $A$  and by $_{(p)}A$ the quotient
of $A$ by the $q$-torsion elements, for all primes $q \neq p$ . We
will refer to  this quotient  group as the
 $(p)$-torsion group of $A$.

\bigskip

It is shown in \cite{Curtis:65} (see also \cite{MPBook} props.
5.33 and 5.35) that
\begin{equation}\label{curtisres}
L_iJ^n(\mathbb Z,1)=\begin{cases} \mathbb Z,\ i=2,\ n=2\\
0,\ \text{otherwise}\end{cases}
\end{equation}
This, together with the Curtis decomposition \eqref{curdec} of the  Lie
functors and
the computation of the groups  $L_i \Lambda^2\Lambda^2(\Z,1)$ in
table \ref{tab:3}, implies that there is no 3-torsion in any of
the expressions   $L_q\EuScript L^p(\mathbb Z,1)$ for $p<6$. Let
us show that  the first non-trivial 3-torsion term  in the
spectral sequence \eqref{csse1} occurs in the group
 $L_5\EuScript L^6(\mathbb Z,1)$. It follows from
 \eqref{curtisres} and the K\"{u}nneth formula that no 3-torsion is
 produced  by either of the factors $J^6(\Z,1)$ and $J^4(\Z,1) \ot J^2(\Z,1)$
of $\EuScript L^6(\Z,1)$, nor is any contribution made by
$J^2J^3(\Z,1)$
  since
$J^3(\Z,1)$ is contractible.  It thus  follows from
\eqref{curtisres} and \eqref{dfr} (or \eqref{l3a2}) that \bee \label{lil6}
L_i\EuScript L^6(\Z,1)    \simeq    L_iJ^3J^2(\Z,1) \simeq
L_i\EuScript L^3(\Z,2)  \simeq  \begin{cases}
  \Z/3, &i =5\\0, & i \neq 5.
\end{cases}
\ee We  restate this result as: \bee \label{lil6a} LJ^3J^2(\Z,1)
\simeq K(\Z/3, 5). \ee

More generally, the  Curtis decomposition \eqref{curdec}, together
with (\ref{curtisres}) and (\ref{corprop2}), implies  that
3-torsion in the groups  $L_q\EuScript L^r(\mathbb Z,1)$ can
 only arise from   components of the
decomposition of the form  $FJ^{3^k}J^2$ and their tensor products
(for functors  $F=SP^k,\ F=J^k$), so that there is no
3-torsion in the initial terms of \eqref{csse1} unless $6|r$. The
analysis of the   $r=18$ case  is similar to that of $r=6$. The
only contribution to the 3-torsion in
 $L_q\EuScript L^r(\mathbb Z,1)$, for $q\leq 14$, comes from the
derived functors of  $J^3J^3J^2(\mathbb Z,1) $,  and by
\eqref{lil6a}:
  \[L_iJ^3J^3J^2(\mathbb Z,1) \simeq L_i\EuScript L^3(\mathbb
Z/3,5)\,.\]  These groups were computed in (\ref{exa3}), so  it
now follows  from the  connectivity result (\ref{dp}) that
\begin{equation}\label{some3}
_3L_q\EuScript L^r(\mathbb Z,1)=\mathbb Z/3,\ r=18,\ q=8,9
\end{equation}
and
$$
_3L_q\EuScript L^r(\mathbb Z,1)=0,\ 5<q<10,\ r\neq 18.
$$

 We refer to \cite{MPBook} ch. 5 for a
similar analysis of the  2-torsion components in the spectral
sequence  \eqref{csse1}.

\bigskip

For   $r\neq 12$, the 3-torsion components of
 $L_q\EuScript L^r(\mathbb Z,1)$  may all be computed  by the previous
 method so long as  $q\leq 14$,  and indeed all of
these components are trivial except for  those provided by
\eqref{lil6} and (\ref{some3}). We will now consider in detail the
case of the 12th Lie functor. We will need to introduce additional
techniques   in order to achieve  a complete understanding of the
derived functors of $\Le^{12}$ and of the differentials in  the
spectral sequence \eqref{csse1}  within the range $q \leq 14$.

\bigskip

 First observe that only the functors $J^6J^2$,    $J^3J^2J^2$, $J^2J^3J^2,$ $J^4J^2\otimes J^2J^2$   may
give any  contribution to  the 3-torsion in $L_q\Le^{12}(A, 1)$ in
degrees $q \leq 14$. By \eqref{derlam2ga2} and (\ref{l3a2}),  the
derived functors of  $J^3J^2J^2$ and $J^4J^2 \otimes J^2J^2$  are
all 2-torsion groups for $A=\Z$. It follows that that  3-torsion
in $L_q\EuScript L^{12}(\mathbb Z,1)$ within our range  can only
occur in degrees $ q= 10, 11$. In fact we  will now show that
while $J^6J^2(\Z,1) = K(\Z_6,11)$ by \eqref{corprop2}, and so
could in principle contribute to the 3-torsion of $L_{11}\EuScript
L^{12}(\mathbb Z,1)$, this is not  in fact the case:
\begin{prop}
\label{L12} The groups $ _{(3)}L_q\EuScript L^{12}(\mathbb Z,1)=0
$ are trivial for all $\ q\geq 2$.
\end{prop}
{\bf Proof:}    By \eqref{curtisres}, we may think of
the
 Curtis decomposition of $\Le^{12}(\Z,1)$ as reducing    to a
 short exact sequence
\[  \xymatrix{0 \ar[r] &J^2J^3J^2(\Z,1)\ar@{->}[r] &
  \Le^{12}(\Z,1)\ar@{->}[r] &J^6J^2(\Z,1) \ar[r] & 0 } \]
when only 3-torsion is considered.
This induces following commutative diagram of finite groups with
exact horizontal lines, and boundary maps ${}_3\eta_{11}$:
\begin{equation}\label{klaq} \xymatrix@R=20pt{_3L_{11}\EuScript L^{12}(\mathbb
Z,1)\ar@{=}[d]\ar@{^{(}->}[r] & _3L_{11}J^6J^2(\mathbb Z,1)
\ar@{->}[r]^{_3\eta_{11}}\ar@{=}[d] & _3L_{10}J^2J^3J^2(\mathbb
Z,1)\ar@{=}[d]
\ar@{->>}[r] & _3L_{10}\EuScript L^{12}(\mathbb Z,1)\ar@{=}[d]\\
_3L_{11}\EuScript L^{6}(\mathbb Z,2)\ar@{^{(}->}[r] &
_3L_{11}J^6(\mathbb Z,2) \ar@{->}[r]^{_3\eta_{11}} &
_3L_{10}J^2J^3(\mathbb Z,2)
\ar@{->>}[r] & _3L_{10}\EuScript L^{6}(\mathbb Z,2)\\
& \Gamma_2(\mathbb Z_3) \ar@{=}[u] \ar@{->}[r] & \Gamma_2(\mathbb
Z_3) \ar@{=}[u] & }
\end{equation}
In this diagram, the value of $_3L_{11}J^6(\mathbb Z,2) $ was
determined by \eqref{corprop2} and that of    $
_3L_{10}J^2J^3J^2(\mathbb Z,1)  $  follows from  \eqref{lil6a}  and
\eqref{derlam2ga2}.

\bigskip

Let us  now  consider the  Curtis spectral sequence \eqref{curtisgen1}
for the space $K := K(A,n)$ for some abelian group $A$:
\begin{equation}\label{nz0}
  E^1_{r,q}=L_{q+1}\Le^r(\widetilde{\Z} K(A,n),-1) \Longrightarrow  \pi_q(K(A,n-1))
\end{equation}
We will now look at this in more detailed , for $A = \Z$:
\begin{equation}\label{nz1}
  E^1_{r,q}=L_{q+1}\Le^r(\widetilde{\Z} K(\Z,n),-1) \Longrightarrow  \pi_q(K(\Z,n-1))
\end{equation}
  By  Dold's theorem
\cite{Dold} th. 5.1, we may replace the expression  $\Z K(\Z,n)$
in the initial term of \eqref{nz1}   by  $ \oplus_i
K(\widetilde{H}_{i+1}(\Z,n),i)$ so that the spectral sequence
becomes \bee \label{nz11}
 E^1_{r,q} = \pi_{q}\Le^r(\oplus_i K(\widetilde{H}_{i+1}(\Z,n),i))
 \Longrightarrow \Z[n-1]
\ee
 In particular,
\bee \label{nz12}
 E^1_{1,q} = \widetilde{H}_{q+1}(K(\Z,n))\,.
\ee
We now consider the case $n=3$.
 The low-degree  (3)-torsion integral homology groups  of
$K(\Z,3)$ are well-known \cite{Cartan}, \cite{Decker}, in fact the
only nontrivial generators for such  groups  are the fundamental
class $i_3$ in degree 3, the  degree 7 suspension of element
$\gamma_3(i_2) \in H_6(K(\Z,2))$, and their product in degree 10
(under the multiplication induced by the $H$-space structure of
$K(\Z,3)$):
\begin{center}
\label{hz3}
\begin{tabular}{cccccccccccccc}
 $n$ & \vline & 3 & 4 & 5 & 6 & 7 & 8 & 9 & 10 & 11 & 12\\ \hline &\vline\\  $_{(3)}H_nK(\mathbb Z,3)$ & \vline &
 $\mathbb Z$ & 0 & 0 & 0 & $\mathbb
 Z/3$ & 0 & 0 & $\mathbb Z/3$ & 0 & 0
\end{tabular}
\end{center}
 When $n=3$, Dold's theorem  also allows us to  (non-functorially) compute the
other initial terms in \eqref{nz11}, since within the  range of
values of $q \leq 11$ we may replace the expression
$\widetilde{\Z} K(\Z,3)$   by the product of Eilenberg-Mac Lane
spaces $K(\Z,2) \oplus K(\Z/3,6) \oplus K(\Z/3,9)$. For $r=2$ we
must therefore compute the homotopy of the induced  $LJ^2(K(\Z,2)
\oplus K(\Z/3,6) \oplus K(\Z/3,9))$.
 No  3-torsion in the homotopy is provided by the functor $J^2$
  applied to any of the three summands, so the only non-trivial terms
  are those coming from  the cross-effect terms $\Z[2] \otimes \Z/3[6]$ and $\Z[2] \otimes
  \Z/3[9]$, in other words copies of $\Z/3$ in degrees 8 and 11 respectively.

\bigskip

 Similarly, in looking for   the  3-torsion of the  $r=3$ initial terms of
\eqref{nz1} within our range of values  $q \leq 11$,  we need only
consider the homotopy of $LJ^3(K(\Z,2) \oplus K(\Z/3,6))$. Let us
record here the functorial form of \eqref{exa3}, for all $n$ and
a more restricted  range of values of   $k$:
 \begin{lemma}
\label{derj3n} For any abelian group $A$, and integer $n >4$
\[{}_3 L_{n+k}J^3(A,n)=
 \begin{cases}  A\ot \Z_3 \quad k=3,7\\
\tor(A,\Z/3) \quad k=4,8\\
 0, \quad  k= 2,5,6
    \end{cases} \]
 \end{lemma}
It follows that  the summand  $LJ^3(\Z,2)$
 contributes a
 term $\Z/3$ in degree 5  to the 3-torsion of $E^1_{3,q}$\,, while  the summand  $LJ^3(\Z/3,6)$
 contributes  a pair of terms
 $\Z/3$ in degrees 9 and 10.   In addition, since the  second
 third
 cross-effect of the functor $J^3$ is  the functor
\[ J^3_{[2]}(A|B) \simeq (A \ot B \ot A) \oplus (A \ot B \ot B), \]
it  contributes  an additional term $\Z[2] \ot
\Z/3[6] \ot  \Z[2]$ to the homotopy of
 $LJ^3(K(\Z,2) \oplus
K(\Z/3,6))$ 9 in degree 10, in other words a  second  factor  $\Z/3$ to the initial
term $E^1_{3,10}$  of \eqref{nz11}.

\bigskip

There is no contribution to the 3-torsion component of the initial
terms of the spectral sequence  \eqref{nz1} for $ r= 4,5,7,8$
since none of these numbers  is a multiple of 3. If we leave aside  the
case $p=6$  for the time being, the only initial terms  which
we still need to consider are those for which  $r=9$. In our
range $q \leq 11$, the  only the summand of $\Le^9$ which comes
into play   is  $J^3J^3$  and  by   \eqref{exa3}  the homotopy
groups of $LJ^3K(\Z/3,5)$ contribute    a pair of groups $\Z/3$ to
the 3-torsion of $L\Le^9(\Z,2)$  in degrees 8 and 9.

\bigskip

We now collect in the following table the outcome of this
discussion  of the $(3)$-torsion components of the initial terms
of the spectral sequence (\ref{nz1}) for $n=3$ :

\begin{table}[H]
\begin{tabular}{ccccccccccccccccccc}
 $r$ & \vline & 1 &  2 & 3 & 4 & 5 & 6 & 7 & 8 & 9\\ \hline
 $_{(3)}E_{r,11}^1$ & \vline & 0 & $\mathbb Z/3$ & 0& 0 & 0 & $*$ & 0 & 0 & $*$\\
$_{(3)}E_{r,10}^1$ & \vline & 0 & 0 & $(\mathbb Z/3)^2$ & 0 & 0 & $_3L_{10}\EuScript L^6(\mathbb Z,2)$ & 0 & 0 & 0\\
$_{(3)}E_{r,9}^1$ & \vline & $\mathbb Z/3$ & 0 & $\mathbb Z/3$ & 0 & 0 & 0 & 0& 0 & $\mathbb Z/3$\\
$_{(3)}E_{r,8}^1$ & \vline & 0 & $\mathbb Z/3$ & 0& 0 & 0 & 0 & 0 & 0& $\mathbb Z/3$\\
$_{(3)}E_{r,7}^1$ & \vline & 0 & 0 & 0& 0 & 0 & 0 & 0 & 0 & 0\\
$_{(3)}E_{r,6}^1$ & \vline & $\mathbb Z/3$ & 0 & 0& 0 & 0 & 0 & 0 & 0 & 0\\
$_{(3)}E_{r,5}^1$ & \vline & 0 & 0 & $\mathbb Z/3$ & 0 & 0 & 0 & 0
& 0
& 0\\
$_{(3)}E_{r,4}^1$ & \vline & 0 & 0 & 0& 0 & 0 & 0 & 0 & 0
& 0\\
$_{(3)}E_{r,3}^1$ & \vline & 0 & 0 & 0& 0 & 0 & 0 & 0 & 0
& 0\\
$_{(3)}E_{r,2}^1$ & \vline & $\mathbb Z$ & 0 & 0& 0 & 0 & 0 & 0 &
0
& 0\\
\end{tabular}
\vspace{.5cm} \caption{The 3-torsion in the initial terms for the
spectral sequence
   \eqref{nz1} when $n=3$}
\end{table}

Since all the terms in the abutment of this spectral sequence
vanish
 (except for a  copy of $\Z$ in degree 2), it follows by examining the possible differentials in
the spectral sequence that the term $ _{(3)}L_{10}\EuScript
L^6(\mathbb Z,2) $ survives all the way to $E^\infty_{6,10}$ and must
therefore be  trivial. Diagram \eqref{klaq} now makes it clear
that ${}_{(3)}L_{11}\EuScript L^{12}(\mathbb Z,1) =
{}_{(3)}L_{11}\EuScript L^6(\mathbb Z,2) $  also vanishes. These
were the only possibly non-vanishing terms within our range of
degrees, so that finally: \bee \label{Le12} _{(3)}L_r\EuScript
L^{12}(\mathbb Z,1)=0,\  r \leq 14  . \ee \qed

\begin{remark}{\rm
 A direct computation shows that the triviality of
 $L_{10}\Le^6(\Z,2)$ is equivalent to the assertion the class  in $\EuScript L_s^6(\mathbb Z,1)_4,$  of the element
\begin{align*}
\zeta=&
\{\{s_2s_1s_0(a),s_2s_1s_0(a),s_3s_1s_0(a)\},\{s_3s_2s_0(a),s_3s_2s_0(a),s_3s_2s_1(a)\}\}-\\
&
\{\{s_2s_1s_0(a),s_2s_1s_0(a),s_3s_2s_0(a)\},\{s_3s_1s_0(a),s_3s_1s_0(a),s_3s_2s_1(a)\}\}+\\
&
\{\{s_3s_1s_0(a),s_3s_1s_0(a),s_3s_2s_0(a)\},\{s_2s_1s_0(a),s_2s_1s_0(a),s_3s_2s_1(a)\}\}
\end{align*}
is trivial, where  $a$ is a generator of
$\Z = \pi_1K(\mathbb Z,1)$. It would be of some interest to find a specific  element in
$\EuScript L_s^6(\mathbb Z,1)_5$ with boundary  $\zeta$.}
\end{remark}

\bigskip

We now return to the spectral sequence \eqref{csse1}, where we now
know  that $ E^1_{12,q} = 0$ for all $q \leq 14$.
 We will  now
display the entire  table of initial terms in the range $q \leq
14$:

\begin{table}[H]
\begin{tabular}{ccccccccccccccc}
 $r$ & \vline & 6 & 12 & 18 & 24 & 30 & 36 & 42 & 48 & 54 & 162\\ \hline
$_3L_{14}\EuScript L^r(\mathbb Z,1)$ & \vline & 0 & 0 & 0& 0 & 0 & 0 & 0 & 0 & 0 & $\mathbb Z/3$\\
$_3L_{13}\EuScript L^r(\mathbb Z,1)$ & \vline & 0 & 0 & $\mathbb Z/3$ & 0 & 0 & 0 & 0 & 0 & $\mathbb Z/3$ & 0\\
$_3L_{12}\EuScript L^r(\mathbb Z,1)$ & \vline & 0 & 0 & $\mathbb Z/3$ & 0 & 0 & 0 & 0 & 0 & $\mathbb Z/3\oplus \mathbb Z/3$ & 0\\
$_3L_{11}\EuScript L^r(\mathbb Z,1)$ & \vline & 0 & 0 & 0 & 0 & 0 & 0 & 0 & 0 & $\mathbb Z/3$ & 0\\
$_3L_{10}\EuScript L^r(\mathbb Z,1)$ & \vline & 0 & 0 & 0 & 0 & 0 & 0 & 0 & 0 & 0 & 0\\
$_3L_9\EuScript L^r(\mathbb Z,1)$ & \vline & 0 & 0 & $\mathbb Z/3$ & 0 & 0 & 0 & 0 & 0 & 0 & 0\\
$_3L_8\EuScript L^r(\mathbb Z,1)$ & \vline & 0 & 0 & $\mathbb Z/3$ & 0 & 0 & 0 & 0 & 0 & 0 & 0\\
$_3L_7\EuScript L^r(\mathbb Z,1)$ & \vline & 0 & 0 & 0 & 0 & 0 & 0 & 0 & 0 & 0 & 0\\
$_3L_6\EuScript L^r(\mathbb Z,1)$ & \vline & 0 & 0 & 0 & 0 & 0 & 0 & 0 & 0 & 0 & 0\\
$_3L_5\EuScript L^r(\mathbb Z,1)$ & \vline & $\mathbb Z/3$ & 0 & 0
& 0 & 0 & 0 & 0 & 0 & 0 & 0\\
\end{tabular}
\vspace{.5cm} \caption{The 3-torsion in the initial terms of the
    spectral sequence \eqref{csse1}}
\end{table}

The values of the various terms in this table are justified as
follows. Observe first of all that the
 vanishing of all terms $  E^1_{12,q}$ terms
 implies that there are  no non-zero terms  $  E^1_{r,q}$    whenever  $r$ is a
multiple of 12. Non-trivial terms with  $r= 18$ arise by applying the
functor $J^3$  according to the rule of lemma \ref{derj3n} to  the
cyclic group
  ${}_3E^1_{6,5}= \Z/3$, so that they
are contributed by  derived functors of the summands $J^3J^3J^2$
of $\Le^{18}$. Applying one more functor $J^3$  to each of the two
cyclic groups $E^1_{18,8}$ and  $E^1_{18,9}$  provides us, according to
the same rule, with two additional copies of $\Z/3$  in the columns $r= 54$.
Finally, a last composition with a $J^3$ yields the only
non-trivial term in column $r=162$ within our range  $r \leq 14$. Our
discussion makes it clear that this cyclic group has been contributed by the
appropriate derived functor of the summands  $J^3J^3J^3J^3J^2$  of
$\Le^{162}$.

\bigskip

It now follows from this discussion, by taking  into account the
possible differentials in the spectral sequence, that we have obtained
 the following description of the 3-torsion in
$\pi_i(S^2)$ in the range $i \leq 11$: \begin{equation}
\label{pi2s2}
_3\pi_i(S^2) = \begin{cases} \Z/3  \quad &i= 6,9,10\\
\, 0 & \text{otherwise}\,.
\end{cases}
\end{equation}
In addition,
$$
\pi_i(S^2)\supseteq \mathbb Z/3,\ i=13,14.
$$
We  recover in this way  by purely algebraic methods certain of
Toda's results \cite{To}.  In fact, it can be shown by comparing once more
once more the  differentials in a spectral sequence for the Moore space
\eqref{cssec} with those in the    corresponding spectral sequence  for an  Eilenberg-Mac Lane
space, and by   suspension arguments,
  that the additional differentials
$d_{18,12}^{36}: \mathbb Z/3\to \mathbb Z/3$ and $d_{18,13}^{36}:
\mathbb Z/3\to \mathbb Z/3\oplus \mathbb Z/3$ in
\eqref{csse1}  are both  monomorphisms.
 In this way, we recover algebraically  the entire  description of the 3-torsion in  $\pi_n(S^2)$ up
to degree 14.

\subsection{Some homotopy groups of $M(A,2)$}
We now consider  the spectral sequence \eqref{curtisgen2}
 for $n=2$:
\begin{equation}\label{cssec}
E_{r,q}^1=L_q\EuScript L^r(A,1)\Rightarrow \pi_{q+1}M(A,2).
\end{equation}
For $r=3$, some intial terms in this spectral  sequence  were
computed in  \S\ref{l3a1}.
 We will  now study  the terms  $E ^1_{4,q} =L_q\EuScript
L^4(A,1)$. The short exact sequences \eqref{defjntilde} and
\eqref{curtis4} derive to the horizontal lines of the two
following diagrams, while the vertical ones arise  from
semi-d\'ecalage and the computations of the groups
$L_i\Lambda^2\Lambda^2(A,1)$ in  table
\ref{tab:3}
:

\begin{equation}
\label{semi-4} \xymatrix@R=14pt{ L_1Y^4(A) \ar@{->}[r]^{\bar\theta_1}
\ar@{->}[d]^\simeq & \Lambda^2\Gamma_2(A)\ar@{->}[r]
\ar@{^{(}->}[d]
& \EuScript L_s^4(A) \ar@{->>}[r] \ar@{->}[d] & Y^4(A)\ar@{->}[d]^\simeq\\
L_5J^4(A,1) \ar@{->}[r]^{\theta_4} &
L_4\Lambda^2\Lambda^2(A,1)\ar@{->}[r]
\ar@{->>}[d] & L_4\EuScript L^4(A,1) \ar@{->>}[r] & L_4J^4(A,1)\\
& \pi_1\left(\Gamma_2(A)\buildrel{L}\over\otimes \mathbb
Z/2\right) \\
}
\end{equation}
$$
\xymatrix@R=14pt{ L_2Y^4(A) \ar@{->}[r]^{\bar\theta_2} \ar@{->}[d]^\simeq &
L_1\Lambda^2\Gamma_2(A)\ar@{->}[r]
\ar@{^{(}->}[d] & L_1\EuScript L_s^4(A) \ar@{->}[r] \ar@{->}[d] & L_1Y^4(A)\ar@{->}[d]^\simeq \\
L_6J^4(A,1) \ar@{->}[r]^{\theta_6} &
L_5\Lambda^2\Lambda^2(A,1)\ar@{->}[r] \ar@{->>}[d] & L_5\EuScript
L^4(A,1)
\ar@{->}[r] & L_5J^4(A,1)\\
& \Tor(R_2(A),\mathbb Z/2) }
$$
The computation of the $L_i\Lambda^2\Lambda^2(A)$  also implies
that there are  genuine d\'ecalage isomorphisms
\begin{equation}
\label{genuine}
  L_{i}\EuScript L_s^n(A)
\simeq  L_{i+n} \EuScript L^n(A,1)
\end{equation}
for $n=4$ whenever $i>2$. The same is true for $n=5$ and all $i$
by comparison of the derived long exact sequences associated to
the sequences \eqref{curtis5} and  \eqref{exacty5}.

\bigskip

This discussion
 provides the justification for the
description  of the columns $q= 1,2,3,5$
  of the  following table of initial terms of the spectral
sequence (\ref{cssec}):

{\scriptsize
\begin{table}[H]
\begin{align*}
& \begin{tabular}{cccccccccccccccc}
 $q$ & \vline & $E_{1,q}^1$ & \vline & $E_{2,q}^1$ & \vline & $E_{3,q}^1$ & \vline & $E_{4,q}^1$ & \vline & $E_{5,q}^1$\\
 \hline
  $7$ & \vline & 0 & \vline & 0 & \vline & 0 & \vline & $L_3\EuScript L_s^4(A)$ & \vline & $L_2\EuScript L_s^5(A)$\\
 $6$ & \vline & 0 & \vline & 0 & \vline & 0 & \vline & $L_2\EuScript L_s^4(A)$ & \vline & $L_1\EuScript L_s^5(A)$\\
$5$ & \vline & 0 & \vline & 0 & \vline & $L_2Y^3(A)$ & \vline & $\Tor(R_2(A),\mathbb Z/2) \oplus  L_1\EuScript L_s^4(A)$ & \vline & $\EuScript L_s^5(A)$\\
$4$ & \vline & 0 & \vline & 0 & \vline & $L_1Y^3(A)$ & \vline & $
\pi_1\left(L\Gamma_2(A)\buildrel{L}\over\otimes \mathbb Z/2\right)  \oplus  \EuScript L_s^4(A)$ & \vline & 0\\
$3$ & \vline & 0 & \vline & $R_2(A)$ & \vline & $Y^3(A)$ & \vline & $\Gamma_2(A)\otimes\mathbb Z/2$ & \vline & 0\\
$2$ & \vline & 0 & \vline & $\Gamma_2(A)$ & \vline & 0 & \vline & 0 & \vline & 0\\
$1$ & \vline & $A$ & \vline & 0 & \vline & 0 & \vline & 0 & \vline
& 0
\end{tabular}
\end{align*}
\begin{align*}
& \begin{tabular}{cccccccccc}
 $q$ & \vline & $E_{6,q}^1$ & \vline & $E_{7,q}^1$\\
 \hline
  $7$  & \vline &$ \Tor(R_2(A),Z/3) \oplus  \Tor(L_2Y^3(A),\mathbb Z/2)\oplus   L_1\EuScript L_s^6(A)$ & \vline
  & $ \EuScript L_s^7(A)$\\
 $6$  & \vline & $\pi_1\left(\Gamma_2(A)\buildrel{L}\over\otimes \mathbb Z/3\right)\oplus
 \pi_2\left(Y^3(A)\buildrel{L}\over\otimes \mathbb Z/2\right)\oplus
 \EuScript L_s^6(A)$ & \vline & $
0$\\
$5$  & \vline & $\Gamma_2(A)\otimes \mathbb Z/3\oplus
\pi_1\left(Y^3(A)\buildrel{L}\over\otimes\mathbb Z/2\right)$ &
\vline & 0
\\
$4$ & \vline & $Y^3(A)\otimes \mathbb Z/2$ & \vline & 0\\
$3$ & \vline & 0 & \vline & 0\\
$2$ & \vline & 0 & \vline & 0\\
$1$ & \vline & 0 & \vline & 0
\end{tabular}
\end{align*}
\begin{align*}
& \begin{tabular}{cccccccccc}
 $q$ & \vline & $E_{8,q}^1$ & \vline & $E_{9,q}^1$ & \vline & $E_{10,q}^1$\\
 \hline
 $6$  & \vline & $\pi_2\left(\Gamma_2(A)\buildrel{L}\over\otimes
\mathbb Z/2\lotimes \mathbb Z/2 \right)\oplus \pi_1\left(\EuScript
L_s^4(A)\lotimes \mathbb Z/2\right)$ & \vline &
$Y^3(A)\otimes \mathbb Z/3$ & \vline & $\EuScript L_s^5(A)\otimes \mathbb Z/2$\\
$5$  & \vline & $\pi_1\left(\Gamma_2(A)\buildrel{L}\over\otimes
\mathbb Z/2\lotimes \mathbb Z/2 \right)\oplus \EuScript
L_s^4(A)\otimes \mathbb Z/2$ & \vline & 0 & \vline & 0
\\
$4$ & \vline & $\Gamma_2(A)\otimes \mathbb Z/2$ & \vline & 0 & \vline & 0\\
$3$ & \vline & 0 & \vline & 0 & \vline & 0\\
$2$ & \vline & 0 & \vline & 0 & \vline & 0\\
$1$ & \vline & 0 & \vline & 0 & \vline & 0
\end{tabular}
\end{align*}
\caption{The $E^1$-terms of the spectral sequence
\eqref{cssec}}
\label{tab:9}
\end{table}
}

\noindent The second column  in table \ref{tab:9} follows from
\eqref{simplest}  and from the computation \eqref{derlam2ga2} for
$n=1$, and the third one from \eqref{lie3}.
The  first summand in the  term $E^1_{4,3}$ in the fourth column is
provided by the subgroup $\Lambda^2\Lambda^2(A)$ of $\Le^4(A)$
exhibited in \eqref{curtis4},  when the expression in position $(2,1)$
in table \ref{tab:3} is taken into account, and its derived versions then occur
above it. The second summand in  $E^1_{4,3}$
arises  from  the d\'ecalage isomorphism \eqref{genuine}  and  a
diagram chase in
 diagram \eqref{semi-4}. Once more,  its derived versions are then to
 be found above it.

\bigskip

  We  will now show how to
 find the terms  of
 interest to us in
 columns 6 and 8, by the methods of \S \ref{filder}.
 Those in the sixth column  in degrees
$q=4,5$  only depend on the first two summands  $J^3J^2(A)$ and
$J^2J^3(A)$ of $\EuScript L^6(A)$.
 The term $J^2J^3(A)$ in $\EuScript L^6(A)$ contributes an expression
\[ L_4(J^2J^3(A,1)) \simeq  L_4\Lambda^2(Y^3(A),3) \simeq Y^3A \otimes
\Z/2\,,\]  to $E_{6,4}^1$,  since $L_2\Lambda^2(A,1) \simeq
\Ga_2A$ and  $L_5 J^3(A,2) \simeq A \ot \Z/3$  \eqref{l3a2}. The
same  computation provides the corresponding factor in
$E^1_{6,5}$. Similarly, the term $J^3J^2(A)$  provides the
expression $\Ga_2(A) \ot \Z/3$ in  $E_{6,5}^1$, since
$L_2\Lambda^2(A,1) \simeq \Ga_2A$ and  $L_5 J^3(A,2) \simeq A \ot
\Z/3$  \eqref{l3a2}. Finally, the term
 $E^1_{8,4}$  comes from the term  $L_4J^2J^2J^2(A,1)$ in
 $L_4\EuScript L^8(A,1)$ by the same sort of  reasoning: we already know
 that \[
L_3\Lambda^2\Lambda^2(A,1) \simeq \Gamma_2(A) \ot \Z/2 \,.\] This
implies that
\[L_4 \Lambda^2(\Lambda^2\Lambda^2 (A,1)) \simeq L_4\Lambda^2(\Gamma_2A
\ot \Z/2, 3)\] and the result follows, since
\[L_4\Lambda^2(\Gamma_2A
\ot \Z/2, 3)  \simeq L_6SP^2(\Ga_2A \ot \Z/2,4) \simeq
H_6(K(\Ga_2A \ot \Z/2,4)) \simeq \Ga_2A \ot \Z/2\,,
\]
with the last isomorphism following by a direct calculation, or by
reference to the well-know Eilenberg-MacLane  functorial stable
isomorphism
\[H_6(K(B,4)) \simeq B \otimes \Z/2 \,.\]
We refer to \cite{MPBook} \S 5.5 for a more complete discussion by
one of us of the derived functors of iterates of $\Lambda^2$ when
$A = \Z$.

\bigskip

It is immediate from the line $q=2$ of table \ref{tab:9} that
\bee
\label{certain} \pi_3(M(A,2)) \simeq \Ga_2(A) \,,
\ee
a result which essentially goes back to J.H.C. Whitehead's ``certain exact
sequence'' \cite{Whitehead}, and that in particular
a generator of $\Ga_2(\Z)$ corresponds to the class of the Hopf map
$\eta: S^3 \la S^2$. By comparing the spectral sequence \eqref{cssec}
with the corresponding spectral sequence \eqref{nz0} for $n=3$, one
 verifies that the differential $d^1_{3,4}: E^1_{3,4} \la E^1_{4,3}$
 in \eqref{cssec} is trivial.
The line $q=3$  of table \ref{tab:9} then   implies that there is a  short exact
sequence
\begin{equation}
\label{moore4} 0\to \EuScript L_s^3(A)\oplus (\Gamma_2(A)\otimes
\mathbb Z/2)\to \pi_4M(A,2)\to R_2(A)\to 0\,,
\end{equation}
a result already proved in \cite{Bau}, \cite{BB}, where the expression
$\EuScript L_s^3(A)\oplus (\Gamma_2(A)\otimes \mathbb Z/2)$ is
denoted $\Ga_2^2(A)$.

\bigskip

 Similarly, the    last  two terms in
the line $q=4$ of our table, together with the factor   $\EuScript
L_s^4(A)$ from $E^1_{4,4}$, regroup to the expression denoted
$\Ga_2^3(A)$ in \cite{BB}, while  the direct sum of the two
remaining terms on the line $q=4$   correspond to   the derived
functor $L_1\Ga_2^2(A)$ of the functor $\Ga_2^2(A)$ mentioned
above. By considering
   the  restriction of the differential $d_{4,5}^4:
E_{4,5}^4\to E_{8,4}^4$ in our table to the factor $\Tor (R_2(A),
\Z /2)$ of   $ E_{4,5}^4 $ we therefore recover the description of $\pi_5M(A,2)$ in
\cite{BB} as a middle term in an  exact sequence:
\begin{equation}\label{moore5}
L_2\Gamma_2^2(A)\buildrel{d_2}\over\to \Gamma_2^3(A)\to
\pi_5M(A,2)\to L_1\Gamma_2^2(A)\to 0.
\end{equation}
where $d_2$ is a differential in the spectral sequence from
\cite{Dre} (for a  generalized version
of this  sequence, see  \cite{BG} theorem  5.1). We will verify later
on in this section (see diagram \eqref{suspdiag}) that this
restriction of   $d_{4,5}^4$ is not zero. This implies that the
corresponding differential $d_2$ in \eqref{moore5}  is also non-trivial.
This discussion is consistent with the   low-dimensional homotopy groups of
the Moore space $M(\mathbb Z/2,2)=\Sigma \mathbb RP^2$ as known
from \cite{Wu2}:
\begin{table}[H]
\begin{tabular}{cccccccccc}
 $i$ & \vline & 2 & 3 & 4 & 5 & 6 & 7\\ \hline &\vline\\  $\pi_iM(\mathbb Z/2,2)$ & \vline &
 $\mathbb Z/2$ & $\mathbb Z/4$ & $\mathbb Z/4$ & $(\mathbb Z/2)^{\oplus 3}$ & $(\mathbb
 Z/2)^{\oplus 5}$ & $(\mathbb Z/2)^{\oplus 2}\oplus (\mathbb Z/4)^{\oplus 2}\oplus \mathbb Z/8 $
\end{tabular}
\medskip
\caption{}
\label{pnmz22}
\end{table}
\bigskip

Finally, returning to the case $A = \Z$, we also observe in table
\ref{tab:9}, in positions $E^1_{2p,\, 2p-1}$  with $p$ prime, the early
occurences in ${}_p\pi_{2p}(S^{2})$  of Serre's first non-trivial
$p$-torsion in the homotopy of $S^2$ (see also for this \cite{MPBook}
corollary 5.40 and the discussion pp. 280--281).
\begin{remark}{\rm
In the spectral sequence from \cite{Dre} there are  terms
$E_{p,q}^2=L_p\Gamma_2^q(A),$ where $\Gamma_2^q(A)$ is the $q$-th
term arising  from the homotopy operation algebra. In particular,
there is  a natural homomorphism $\EuScript L_s^{q+1}(A)\to
\Gamma_2^q(A),$ where the occurence of the $(q+1)$-st super-Lie
functor is due to
 Whitehead products, viewed as homotopy operations. It is natural
to conjecture that the semi-d\'ecalage described in theorem
\ref{semidec} connects the homotopy operation spectral sequence
from \cite{Dre} with the Curtis spectral sequence, with for
example the existence of a commutative diagram

$$
\xyma{L_i\EuScript L_s^{q}(A) \ar@{->}[d]^{\hat d^2}
\ar@{->}[r]^{\chi_1^*} & L_{i+q}\EuScript L^q(A,1) \ar@{->}[d]^{d^1}\\
L_{i-2}\EuScript L_s^{q+1}(A) \ar@{->}[r]^{\chi_1^*} &
L_{i+q-1}\EuScript L^{q+1}(A,1)}
$$
where $\hat d^2$ is a natural map induced by the second
differential in the homotopy operation spectral sequence. The
low-dimensional
 computation   given above support  this
conjectural connection between the two spectral sequences.}
\end{remark}

\subsection{Some  homotopy groups of   $M(\mathbb Z/p,2)$,\ $p\neq 2$}
The next proposition provides us with some information regarding
the derived functors of $\Le^4_s(A)$. We begin with the following
lemma:
\begin{lemma}\label{y4} Let  $A=\mathbb Z/p$ for some prime  $p \neq 2$.  The natural map $\Tor(\Omega_3(A),A)\to \Omega_4(A)$ is an isomorphism.
\end{lemma}
\begin{proof}
By \cite{Breen} (5.14), there exists, for any abelian group $A$
and integer $h$, a commutative diagram of abelian groups
\begin{equation}\label{gammadi}
\xyma{\Gamma_3(\ _hA)\otimes\ _hA\ar@{->}[r] \ar@{->}[d]^{\lambda_h^3\otimes 1}& \Gamma_4(\ _hA)\ar@{->}[d]^{\lambda_4^h}\\
\Tor(\Omega_3(A),A) \ar@{->}[r] & \Omega_4(A)}
\end{equation}
where the upper horizontal arrow is induced by the multiplication
in the divided power algebra. The arrows $\lambda_h^i$ provide,
where $h$ varies, and the slide relations are taken into account,
presentations for the groups $\Tor(\Omega_3(A),A)$ and
$\Omega_4(A)$ respectively.
 Let us  now suppose that $A$ is cyclic of order $p$, with a chosen
generator $a\in A$. In that case the only relevant integer is
$h=p$. We know that $\Omega_3(\mathbb Z/p)=\Omega_4(\mathbb
Z/p)=\mathbb Z/p$ so that the lower horizontal map in
(\ref{gammadi}) is a homomorphism $\mathbb Z/p\to \mathbb Z/p$.
Let us show that this morphism is non-trivial:  the image of
$\gamma_3(a)\otimes a$ in $\Omega_4(\mathbb Z/p):$
$$
\gamma_3(a)\otimes a\mapsto 4\gamma_4(a)\mapsto 4\omega_4^p(a)
$$
and $4\,\omega_4^p(a) \neq 0$ since $p \neq 2$.
\end{proof}
\begin{prop}
 For any integer   $i\geq 0$ and any odd prime  $p$, one then has
 $$L_i\EuScript
L_s^4(\mathbb Z/p)=\begin{cases} \mathbb Z/p,\ i=1,2\\
0,\ i\neq 1,2\end{cases}$$
\end{prop}
\begin{proof}
It follows from definition that $\EuScript L_s^4(A)=0$ for every
cyclic group $A$. By   \eqref{exacty4},  the  sequence
\begin{multline}\label{jkq}
0 \to L_3\Lambda^2\Gamma_2(A)\to L_3\EuScript L_s^4(A)\to
L_3Y^4(A)\to L_2\Lambda^2\Gamma_2(A)\to\\ L_2\EuScript L_s^4(A)\to
L_2Y^4(A)\stackrel{\partial}{\to} L_1\Lambda^2\Gamma_2(A)\to
L_1\EuScript L_s^4(A)\to L_1Y^4(A)
\end{multline}
is exact.  By \eqref{derga2},  $L\Gamma_2(\Z/p) = K(\Z/p,0)$
for $p$ odd,  so that
\bee
\label{ll2g2}
L_i\Lambda^2\Gamma_2(\mathbb Z/p)=\begin{cases} \mathbb Z/p,\
i=1,\\ 0,\ i\neq 1\,.\end{cases}
\ee In particular, the right-hand arrow in \eqref{jkq} is surjective.
The definition of $Y^4$  implies that there is a  long
exact sequence
\begin{equation*}
L_2Y^4(A)\to \pi_2\left(\Lambda^3(A)\buildrel{L}\over \otimes
A\right)\to L_2\Lambda^4(A)\to L_1Y^4(A)\to
\pi_1\left(\Lambda^3(A)\buildrel{L}\over \otimes A\right)\to
L_1\Lambda^4(A)
\end{equation*}
and  an isomorphism
$$
L_3Y^4(A) \simeq \ker\{\Tor(\Omega_3(A),A)\to \Omega_4(A)\}.
$$
Lemma \ref{y4} asserts that the  group $L_3Y^4(\Z/p)$ is trivial,
and we know  by \eqref{omeg-cycl} and \eqref{LLambda-cycl} that
$$
L_i\Lambda^4(\mathbb Z/p)=\begin{cases} \mathbb Z/p,\ i=3\\ 0,\
i\neq 3.\end{cases}
$$
The exactness of the sequence \eqref{jkq} then implies that
\bee
\label{liy4}
L_iY^4(\mathbb Z/p)=\begin{cases} \mathbb Z/p,\ i=2,\\
0,\ i\neq 2.\end{cases}
\ee

 We will now  show that the boundary map (\ref{jkq}):
\begin{equation}\label{mn1}
\xymatrix@R=12pt{L_2Y^4(\mathbb Z/p)\ar@{->}[r]^{\partial} \ar@{=}[d]&
L_1\Lambda^2\Gamma_2(\mathbb Z/p) \ar@{=}[d]\\
\mathbb Z/p \ar@{->}[r] & \mathbb Z/p}
\end{equation}
 is trivial.
Consider first the case $p\neq 2,3$. One then has the following
values for the homology groups $H_*K(\mathbb Z/p,2):$
\begin{table}[H]
\begin{tabular}{cccccccccccccc}
 $n$ & \vline & 2 & 3 & 4 & 5 & 6 & 7 & 8 & 9\\ \hline &\vline\\  $H_nK(\mathbb Z/p,2)$ & \vline &
 $\mathbb Z/p$ & 0 & $\mathbb Z/p$ & 0 & $\mathbb
 Z/p$ & 0 & $\mathbb Z/p$ & 0
\end{tabular}
\caption{}
\label{hzp2}
\end{table}

\bigskip

\noindent  The  analogue of the spectral sequence \eqref{nz1} for
$K= K(\Z/p,2)$ is
\begin{equation}\label{nzp2}
  E^1_{r,q}=L_{q+1}\Le^r(\widetilde{\Z} K(\Z/p,2),-1) \Longrightarrow  \pi_q(K(\Z/p,1))
\end{equation}
Reasoning as in the proof of proposition \eqref{L12}, we find
that the initial terms in this spectral sequence  are the
following:
\begin{table}[H]
\begin{tabular}{ccccccccccccccc}
 $q$ & \vline & $E_{1,q}^1$ & $E_{2,q}^1$ & $E_{3,q}^1$ & $E_{4,q}^1$ & $E_{5,q}^1$ & $E_{6,q}^1$ & $E_{7,q}^1$ & $E_{8,q}^1$ & $E_{9,q}^1$ & $E_{10,q}^1$\\
 \hline
 8 & \vline & 0 & $*$ & $*$ & $*$ & $*$ & $*$ & $*$ & $*$ & $*$ & $*$\\
 7 & \vline & $\mathbb Z/p$ & $\mathbb Z/p$ & $*$ & $*$ & $*$ & $*$ & $*$ & $*$ & $*$ & $*$\\
6 & \vline & 0 & $\mathbb Z/p^2$ & $\mathbb Z/p^2$ & $*$ & $*$ & $*$ & $*$ & $*$ & $*$ & $*$\\
5 & \vline & $\mathbb Z/p$ & $\mathbb Z/p$ & $\mathbb Z/p$ & $L_1\EuScript L_s^4(\mathbb Z/p)$ & 0 & 0 & 0 & 0 & 0 & 0\\
4 & \vline & 0 & $\mathbb Z/p$ & $\mathbb Z/p$ & 0 & 0 & 0 & 0 & 0 & 0 & 0\\
3 & \vline & $\mathbb Z/p$ & 0 & 0 & 0 & 0 & 0 & 0 & 0 & 0 & 0\\
2 & \vline & 0 & $\mathbb Z/p$ & 0 & 0 & 0 & 0 & 0 & 0 & 0 & 0\\
1 & \vline & $\mathbb Z/p$ & 0 & 0 & 0 & 0 & 0 & 0 & 0 & 0 & 0 \\
\end{tabular}
\vspace{.5cm} \caption{The $E^1$-term of the spectral sequence
  \eqref{nzp2} for  $ p
  \neq 2,3$}
\end{table}

\bigskip

This spectral sequence converges to the graded group $\mathbb
Z/p[1]$. Suppose that $L_1\EuScript L_s^4(\mathbb Z/p)=0$. In that
case $E_{2,6}^\infty\oplus E_{3,6}^\infty\neq 0$ and  this
contradicts the fact that homotopy groups $\pi_iK(\mathbb Z/p,1)$
 are
trivial for $i \geq 2$. It follows by \eqref{ll2g2} and \eqref{liy4}     that  $L_1\EuScript
L_s^4(\mathbb Z/p)=\mathbb Z/p$ and the map (\ref{mn1}) is the
zero map. The description of all  the derived functors  of
$\EuScript L_s^4(\mathbb Z/p)$ for $p \neq 2,3$ now follows from the  exact
sequence (\ref{jkq}).

\bigskip

  We now consider the  $p=3$ case. We have the
following description of the low degree  homology  of $K(\Z/3,2)$:
\begin{table}[H]
\begin{tabular}{cccccccccccccc}
 $n$ & \vline & 2 & 3 & 4 & 5 & 6 & 7 & 8 & 9\\ \hline &\vline\\  $H_nK(\mathbb Z/3,2)$ & \vline &
 $\mathbb Z/3$ & 0 & $\mathbb Z/3$ & 0 & $\mathbb
 Z/9$ & $\mathbb Z/3$ & $\mathbb Z/3$ & $\mathbb Z/3$
\end{tabular}
\caption{}
\end{table}
\nopagebreak
The initial terms for the spectral sequence    \eqref{nzp2} for $p=3$
 are the following:
\begin{table}[H]
\begin{tabular}{ccccccccccccccc}
 $q$ & \vline & $E_{1,q}^1$ & $E_{2,q}^1$ & $E_{3,q}^1$ & $E_{4,q}^1$ & $E_{5,q}^1$ & $E_{6,q}^1$ & $E_{7,q}^1$ & $E_{8,q}^1$ & $E_{9,q}^1$ & $E_{10,q}^1$\\
 \hline
 8 & \vline & $\mathbb Z/3$ & $*$ & $*$ & $*$ & $*$ & $*$ & $*$ & $*$ & $*$ & $*$\\
 7 & \vline & $\mathbb Z/3$ & $\mathbb Z/3^2$ & $*$ & $*$ & $*$ & $*$ & $*$ & $*$ & $*$ & $*$\\
6 & \vline & $\mathbb Z/3$ & $\mathbb Z/3^2$ & $\mathbb Z/3^3$ & $*$ & $*$ & $*$ & $*$ & $*$ & $*$ & $*$\\
5 & \vline & $\mathbb Z/9$ & $\mathbb Z/3$ & $\mathbb Z/3^2$ & $L_1\EuScript L_s^4(\mathbb Z/3)$ & 0 & $\mathbb Z/3$ & 0 & 0 & 0 & 0\\
4 & \vline & 0 & $\mathbb Z/3$ & $\mathbb Z/9$ & 0 & 0 & 0 & 0 & 0 & 0 & 0\\
3 & \vline & $\mathbb Z/3$ & 0 & 0 & 0 & 0 & 0 & 0 & 0 & 0 & 0\\
2 & \vline & 0 & $\mathbb Z/3$ & 0 & 0 & 0 & 0 & 0 & 0 & 0 & 0\\
1 & \vline & $\mathbb Z/3$ & 0 & 0 & 0 & 0 & 0 & 0 & 0 & 0 & 0 \\
\end{tabular}
\vspace{.5cm} \caption{The $E^1$-term of the  spectral sequence
  \eqref{nzp2} for
$p=3$} \label{table32}
\end{table}

We will now prove  that $L_1\EuScript L_s^4(\mathbb Z/3)=\mathbb
Z/3$.
For  any  abelian group  $A$, the
differentials $d_{1,7}^1$ and $d_{1,8}^1$ in the corresponding
spectral sequence \eqref{nz0} for $n=2$
have the property that the
following natural diagrams are commutative:
$$
\xymatrix@R=20pt{\Gamma_4(A)\ar@{->}[r]^(.45){\kappa^4}\ar@{^{(}->}[d] &
\Gamma_3(A)\otimes
A\ar@{^{(}->}[d]\\
E_{1,7}^1 \ar@{->}[r]^{d_{1,7}^1} & E_{2,6}^1} \qquad  \qquad
\xymatrix@R=12pt{L_1\Gamma_4(A)\ar@{->}[r]^(.4){\kappa_1^4}\ar@{=}[d] &
\pi_1\left(\Gamma_3(A)\lotimes
A\right)\ar@{^{(}->}[d]\\
E_{1,8}^1 \ar@{->}[r]^{d_{1,8}^1} & E_{2,7}^1}\,,
$$
where $\kappa^4$ is the homomorphism in the Koszul complex
$Kos^4(A\buildrel{1_A}\over\to A)$ \eqref{koszul2} and $\kappa_1^4$ its first
derived analog. This implies, in the case $A=\mathbb Z/3$, that
the differentials $d_{1,7}^1$ and $d_{1,8}^1$ are monomorphisms.
The assumption  $L_1\EuScript L_s^4(\mathbb Z/3)=0,$ implies that
$E_{3,6}^\infty\neq 0$ and this contradicts the triviality of the
sixth homotopy group of $K(\mathbb Z/3,1)$.
\end{proof}

\begin{remark} {\rm
For $p=2$ the description of $L_i\EuScript L_s^4(\mathbb Z/p)$ is also  more
 complicated. For example, the group $L_2\EuScript L_s^4(\mathbb Z/2)$
 contains non-trivial  4-torsion elements. In the simplicial language, a generator of the
 4-torsion subgroup  is provided by the following element:
\begin{multline*}
\{s_0a_1,s_1a_1,s_0a_1,s_1a_1\}\,-\,\{s_1a_1,s_1s_0a_0,s_0a_1,s_0a_1\}\,+\\
\{s_0a_1,s_1s_0a_0,s_1a_0,s_1a_0\}\,+\,
\{\{s_1a_1,s_1s_0a_0\},\{s_0a_1,s_1s_0a_0\}\}
\end{multline*}
This 4-torsion element corresponds to  twice the 8-torsion
element of Cohen-Wu \cite{Cohen-Wu},\cite{Wu2} App. A, which lives in
the last  summand of $\pi_7\Sigma \mathbb RP^2=\pi_7M(\mathbb Z/2,2)$
 (see table \ref{pnmz22}).
We will not discuss this computation, since it
involves more elaborate techniques than those described here.}
\end{remark}

 After these preliminaries regarding the derived functors of
 $\Le^4_s$,  let us  begin our computations of the
 homotopy of the spaces  $M(\Z/p,2)$ with the case $p=3$. By remark
 \ref{ourremark}, we know
that
$$
L_i\EuScript L^3(\mathbb Z/3,1)=\begin{cases} \mathbb Z/9,\ i=4\\
\mathbb Z/3,\ i=5\\ 0,\ i\neq 4,5\,. \end{cases}
$$
The computation of the  derived functors $L_q\EuScript L^r(\mathbb
Z/3,1)$ for $q<7$ follows easily  from the Curtis decomposition of
$\Le^r$  and the known values of the derived functors of
 $\EuScript L^3(\Z/3,1)$. We display  the
result  in the following table:

\begin{table}[H]
\begin{tabular}{ccccccccccccccc}
 $q$ & \vline & $E_{1,q}^1$ & $E_{2,q}^1$ & $E_{3,q}^1$ & $E_{4,q}^1$ & $E_{5,q}^1$ & $E_{6,q}^1$ & $E_{7,q}^1$ & $E_{8,q}^1$ & $E_{9,q}^1$ & $E_{10,q}^1$\\
 \hline
6 & \vline & 0 & 0 & 0 & $\mathbb Z/3$ & $\mathbb Z/3$ & $\mathbb Z/3$ & 0 & 0 & 0 & 0\\
5 & \vline & 0 & 0 & $\mathbb Z/3$ & $\mathbb Z/3$ & 0 & $\mathbb Z/3$ & 0 & 0 & 0 & 0\\
4 & \vline & 0 & 0 & $\mathbb Z/9$ & 0 & 0 & 0 & 0 & 0 & 0 & 0\\
3 & \vline & 0 & 0 & 0 & 0 & 0 & 0 & 0 & 0 & 0 & 0\\
2 & \vline & 0 & $\mathbb Z/3$ & 0 & 0 & 0 & 0 & 0 & 0 & 0 & 0\\
1 & \vline & $\mathbb Z/3$ & 0 & 0 & 0 & 0 & 0 & 0 & 0 & 0 & 0 \\
\end{tabular}
\vspace{.5cm} \caption{The $E^1$-term of the spectral sequence
  \eqref{curtisgen2} for $A=\Z/3$ and $n=2$
}
\label{table:11}
\end{table}
The differentials $d_{5,6}^1: \mathbb Z/3\to \mathbb Z/3$ and
$d_{4,6}^2:\mathbb Z/3\to \mathbb Z/3$ are trivial, as follows
from  the comparison between  the Curtis spectral sequences for $K=
M(\mathbb Z/3,2) $ and $K=  K(\mathbb Z/3,2)$
 and from the structure of table \ref{table:11}.
The assumption  that either $d_{5,6}^1$ or $d_{4,6}^2$ is an isomorphism
would  produce a non-trivial term  $E_{3,6}^\infty$ in the
spectral sequence whose initial terms were given in table
\ref{table32}. Looking at the horizontal lines in
  table  \ref{table:11}, we now
see that
\begin{align*}
& \pi_2M(\mathbb Z/3,2)=\mathbb Z/3\\
& \pi_3M(\mathbb Z/3,2)=\mathbb Z/3\\
& \pi_4M(\mathbb Z/3,2)=0\\
& \pi_5M(\mathbb Z/3,2)=\mathbb Z/9\\
& |\pi_6M(\mathbb Z/3,2)|=27.
\end{align*}
Observe that we have in particular exhibited here the cyclic group of
order 9 of \cite{Neisendorfer}, \cite{Leibowitz} mentioned in the introduction.

\bigskip

The homotopy groups of the  spaces $M(\mathbb Z/p,2)$ for a prime
$p\neq 2,3,$ are simpler to describe. In that case, the initial
terms of the spectral sequence  \eqref{curtisgen2} are:

\begin{table}[H]
\begin{tabular}{ccccccccccccccc}
 $q$ & \vline & $E_{1,q}^1$ & $E_{2,q}^1$ & $E_{3,q}^1$ & $E_{4,q}^1$ & $E_{5,q}^1$ & $E_{6,q}^1$ & $E_{7,q}^1$ & $E_{8,q}^1$ & $E_{9,q}^1$ & $E_{10,q}^1$\\
 \hline
6 & \vline & 0 & 0 & 0 & $\mathbb Z/p$ & $\mathbb Z/p$ & 0 & 0 & 0 & 0 & 0\\
5 & \vline & 0 & 0 & 0 & $\mathbb Z/p$ & 0 & 0 & 0 & 0 & 0 & 0\\
4 & \vline & 0 & 0 & $\mathbb Z/p$ & 0 & 0 & 0 & 0 & 0 & 0 & 0\\
3 & \vline & 0 & 0 & 0 & 0 & 0 & 0 & 0 & 0 & 0 & 0\\
2 & \vline & 0 & $\mathbb Z/p$ & 0 & 0 & 0 & 0 & 0 & 0 & 0 & 0\\
1 & \vline & $\mathbb Z/p$ & 0 & 0 & 0 & 0 & 0 & 0 & 0 & 0 & 0 \\
\end{tabular}
\vspace{.5cm} \caption{The $E^1$-term of the   spectral sequence
  \eqref{curtisgen2} for
$A = \Z/p \quad  \text{when} \
  p
  \neq 2,3$ and $n=2$}
\label{tab:14}
\end{table}
\noindent In particular, the derived functors $L_i\EuScript
L^3(\mathbb Z/p,1)$ are simpler for these values of $p$, which
explains the difference between the third columns in tables
\ref{table:11} and \ref{tab:14}. Note also that  the $p$-torsion in $L_4\EuScript
L^3(\mathbb Z/p,1)$ comes from the term $\ker\{\Omega_2(\mathbb
Z/p)\otimes \mathbb Z/p\to L_1\Lambda^3(\mathbb Z/p)\}$ in
(\ref{derl3}). We obtain in particular
\begin{align*}
& \pi_6M(\mathbb Z/p,2)=\mathbb Z/p\\
& |\pi_7M(\mathbb Z/p,2)|=p^2.
\end{align*}

\bigskip

\subsection{Some homotopy groups of  $M(A,3)$} Consider the
spectral sequence \eqref{curtisgen2} for $n= 3$:
\begin{equation}
\label{cssec2} E_{r,q}^1=L_q\EuScript L^r(A,2)\Rightarrow
\pi_{q+1}M(A,3).
\end{equation}
The initial terms of this spectral sequence are
obtained as in the spectral sequence \eqref{cssec}  from the Curtis
decomposition of Lie functors  and the computation of the derived
functors of its graded  components. In addition,  the occurence of a summand $\Le^3(A)$ in
$E^1_{3,6}$ follows from the composition of the maps \eqref{pensionmaps}.
 These initial terms are given in the
following table (recall that $\lambda^2$ is the functor
\eqref{newfunct}):
\begin{table}[H]
\label{lnlma2}
\begin{tabular}{ccccccccccc}
 $q$ & \vline & $E_{1,q}^1$ & \vline & $E_{2,q}^1$ & \vline & $E_{3,q}^1$ & \vline \\
 \hline
  $8$ & \vline & 0 & \vline & 0 & \vline & $L_2\EuScript L^3(A)$ & \vline\\
 $7$ & \vline & 0 & \vline & 0 & \vline & $L_1\EuScript L^3(A)$ & \vline \\
$6$ & \vline & 0 & \vline & 0 & \vline & $\EuScript L^3(A)\oplus
\Tor(A,\mathbb
Z/3)$ & \vline \\
$5$ & \vline & 0 & \vline & $\Omega_2(A)$ & \vline & $A\otimes \mathbb Z/3$ & \vline\\
$4$ & \vline & 0 & \vline & $\lambda^2(A)$ & \vline & 0 & \vline \\
$3$ & \vline & 0 & \vline & $A\otimes \mathbb Z/2$ & \vline & 0 & \vline\\
$2$ & \vline & $A$ & \vline & 0 & \vline & 0 & \vline
\end{tabular}

\begin{tabular}{ccccccccccc}
 $q$ & \vline & $E_{4,q}^1$ & \vline & $E_{5,q}^1$ & \vline & $E_{6,q}^1$ & \vline & $E_{7,q}^1$ & \vline & $E_{8,q}^1$\\
 \hline
  $8$ & \vline & * & \vline & 0 & \vline & * & \vline & * & \vline & *\\
 $7$ & \vline & * & \vline & 0 & \vline & * & \vline & * & \vline & *\\
$6$ & \vline & $\Tor(\lambda^2(A),\mathbb Z/2)\oplus \Omega_2(A)\otimes \mathbb Z/2$ & \vline & 0 & \vline & 0 & \vline & 0 & \vline & *\\
$5$ & \vline & $ \Tor_1(A,\mathbb Z/2,\mathbb Z/2) \oplus  (\Lambda^2(A)\otimes \mathbb Z/2) $ & \vline & 0 & \vline & 0 & \vline & 0 & \vline & $A\otimes\mathbb Z/2$\\
$4$ & \vline & $A\otimes \mathbb Z/2$ & \vline & 0 & \vline & 0 & \vline & 0 & \vline & 0\\
$3$ & \vline & 0 & \vline & 0 & \vline & 0 & \vline & 0 & \vline & 0\\
$2$ & \vline & 0 & \vline & 0 & \vline & 0 & \vline & 0 & \vline &
0
\end{tabular}
\vspace{.5cm} \caption{The initial terms of the spectral sequence
\eqref{cssec2}}
\label{tab:15}
\end{table}

\bigskip

\noindent As a result we have a natural isomorphism
$$
\pi_4M(A,3)\simeq A\otimes \mathbb Z/2,
$$
 which is simply the suspended version of the isomorphism
 \eqref{certain}, as well as the following natural short exact sequence:
$$
0\to A\otimes \mathbb Z/2 \to \pi_5M(A,3)\to \lambda^2(A)\to 0.
$$
However, the latter is not split,   since it is known for example that
$\pi_5M(\mathbb Z/2,3)=\mathbb Z/4$.

\bigskip

The differential $ d_{3,6}^1:E_{3,6}^1\to E_{4,5}^1 $ is trivial,
as can be seen by reduction to the case of  $A$  free abelian of
finite rank, and a
comparison  of the rank of $E_{3,6}^1$  with that of  the  homotopy
group of the corresponding wedge of spheres $S^3$, as computed by the
Hilton-Milnor theorem (see \cite{Curtis:71}
theorem 4.21). On the other hand, the differential $
d_{4,6}^4: E_{4,6}^4\to E_{8,5}^4$ can be  non-trivial. It is an
isomorphism for $A= \Z/2$, as  follows
  from the known description of the groups $\pi_i(M(\mathbb
Z/2,3))=\pi_i(\Sigma^2\mathbb RP^2)$ for small values of  $i$: \hspace{1cm}

\begin{table}[H]
\begin{tabular}{cccccccccc}
 $i$ & \vline & 3 & 4 & 5 & 6 & 7 & 8 &9\\ \hline &\vline\\  $\pi_iM(\mathbb Z/2,3)$ & \vline &
 $\mathbb Z/2$ & $\mathbb Z/2$ & $\mathbb Z/4$ & $\mathbb Z/4 \oplus
 \mathbb Z/2$ & $\mathbb Z/2\oplus \mathbb Z/2$
 & $\mathbb Z/2\oplus \mathbb Z/2$& $\mathbb Z/4 \oplus \mathbb Z/2$
\end{tabular}
\medskip
\caption{}
\label{pnmz23}
\end{table}

In addition,  one  can express the   differential $d^4_{4,6}$
 in \eqref{cssec2}
 as a suspension  by comparing the spectral sequences \eqref{cssec} and
\eqref{cssec2}. For this, consider the following
 commutative diagram, in which  the vertical arrows are suspension
 morphisms:
\bee
\label{suspdiag}
\xymatrix@C=60pt{\Tor(R_2(A),\mathbb Z/2) \ar@{->}[r]^(.55){d^4_{4,5}(A,2)}
\ar@{->}[d] & \Gamma_2(A)\otimes \mathbb Z/2\ar@{->}[d]\\
\Tor(\lambda^2(A),\mathbb Z/2) \ar@{->}[r]_(.57){d^4_{4,6}(A,3)}
& A\otimes \mathbb Z/2}
\ee
 The upper arrow in this diagram  is the restriction to the
first summand of the differential $d^4_{4,5}$ from \eqref{cssec},
 whereas the lower one is the differential $d^4_{4,6}$ from
\eqref{cssec2}.
 The suspension maps are isomorphisms for $A = \Z/2$. Since we
 know that   $d^4_{4,6}$ is an isomorphism  in that case,
so is
     the   differential  $d^4_{4,5}$
 in \eqref{cssec}.

\bigskip

The spectral sequence \eqref{cssec2} determines in particular a
 filtration on the  group $\pi_6M(A,3)$, with  the following  non-trivial
associated graded components:
\begin{align*}
& gr_2\pi_6M(A,3)=\Omega_2(A)\\
& gr_3\pi_6M(A,3)=A\otimes \mathbb Z/3\\
& gr_4\pi_6M(A,3)=(\Lambda^2(A)\otimes \mathbb Z/2) \oplus
\Tor_1(A,\mathbb Z/2,\mathbb Z/2)\\
& gr_8\pi_6M(A,3)=A\otimes \mathbb Z/2/\mathrm{im}(d^4_{4,6})\\
\end{align*}
For $A=\mathbb Z$,  this determines precisely  12 elements in
$\pi_6(S^3)$, which  are the non-trivial elements in the
associated graded components $gr_2,gr_4,gr_5$ listed above. Table \ref{tab:15}  also
implies that there is a natural epimorphism
$$
\pi_7M(A,3)\to \EuScript L^3(A)\oplus \Tor(A,\mathbb Z_3).
$$
As an example of this computation, consider the case $A=\mathbb
Z/3$. A simple analysis, with the help of (\ref{l3a2}), gives the
following description of the initial terms of the corresponding spectral
sequence  \eqref{cssec2}:

\begin{table}[H]
\begin{tabular}{ccccccccccccccc}
 $q$ & \vline & $E_{1,q}^1$ & $E_{2,q}^1$ & $E_{3,q}^1$ & $E_{4,q}^1$ & $E_{5,q}^1$ & $E_{6,q}^1$ & $E_{7,q}^1$ & $E_{8,q}^1$ & $E_{9,q}^1$ & $E_{10,q}^1$\\
 \hline
8 & \vline & 0 & 0 & 0 & 0 & 0 & $\mathbb Z/3$ & 0 & 0 & $\mathbb Z/3$ & 0\\
7 & \vline & 0 & 0 & $\mathbb Z/3$ & 0 & 0 & 0 & 0 & 0 & 0 & 0\\
6 & \vline & 0 & 0 & $\mathbb Z/3$ & 0 & 0 & 0 & 0 & 0 & 0 & 0\\
5 & \vline & 0 & $\mathbb Z/3$ & $\mathbb Z/3$ & 0 & 0 & 0 & 0 & 0 & 0 & 0\\
4 & \vline & 0 & 0 & 0 & 0 & 0 & 0 & 0 & 0 & 0 & 0\\
3 & \vline & 0 & 0 & 0 & 0 & 0 & 0 & 0 & 0 & 0 & 0\\
2 & \vline & $\mathbb Z/3$ & 0 & 0 & 0 & 0 & 0 & 0 & 0 & 0 & 0 \\
\end{tabular}
\vspace{.5cm} \caption{The initial terms of the spectral sequence
\eqref{cssec2} for $A= \Z/3$}
\end{table}
\noindent We conclude that
$$
\pi_7M(\mathbb Z/3,3)=\pi_8M(\mathbb Z/3,3)=\mathbb Z/3.
$$
\subsection{Some homotopy groups of  $M(A,4)$}
The  spectral sequence \eqref{curtisgen2} for $n=4$:
\bee
\label{csse4}
E_{p,q}^1=L_q\EuScript L^p(A,3)\Rightarrow \pi_{q+1}M(A,4).
\ee
has the following initial terms in low dimensions:
{\scriptsize
\begin{table}[H]
\begin{tabular}{ccccccccccc}
 $q$ & \vline & $E_{1,q}^1$ & \vline & $E_{2,q}^1$ & \vline & $E_{3,q}^1$ & \vline \\
 \hline
  $9$ & \vline & 0 & \vline & 0 & \vline & $Y^3(A)$ & \vline\\
 $8$ & \vline & 0 & \vline & 0 & \vline & 0 & \vline \\
$7$ & \vline & 0 & \vline & $R_2(A)$ & \vline & $\Tor(A, \mathbb
Z/3)$ & \vline \\
$6$ & \vline & 0 & \vline & $\Gamma_2(A)$ & \vline & $A\otimes \mathbb Z/3$ & \vline\\
$5$ & \vline & 0 & \vline & $\Tor(A,\mathbb Z/2)$ & \vline & 0 & \vline \\
$4$ & \vline & 0 & \vline & $A\otimes \mathbb Z/2$ & \vline & 0 & \vline\\
$3$ & \vline & $A$ & \vline & 0 & \vline & 0 & \vline
\end{tabular}
\begin{tabular}{ccccccccccc}
 $q$ & \vline & $E_{4,q}^1$ & \vline & $E_{5,q}^1$ & \vline & $E_{6,q}^1$ & \vline & $E_{7,q}^1$ & \vline & $E_{8,q}^1$\\
 \hline
  $9$ & \vline & * & \vline & 0 & \vline & * & \vline & * & \vline & *\\
 $8$ & \vline & * & \vline & 0 & \vline & * & \vline & * & \vline & *\\
$7$ & \vline & $A \ot \Z/2 \oplus \Tor_2(A,\mathbb Z/2,\mathbb Z/2)\oplus \Gamma_2(A)\otimes \mathbb Z/2$ & \vline & 0 & \vline & 0 & \vline & 0 & \vline & $\Tor_1(A,\mathbb Z/2,\mathbb Z/2,\mathbb Z/2)$\\
$6$ & \vline & $\Tor_1(A,\mathbb Z/2,\mathbb Z/2)$ & \vline & 0 & \vline & 0 & \vline & 0 & \vline & $A\otimes\mathbb Z/2$\\
$5$ & \vline & $A\otimes \mathbb Z/2$ & \vline & 0 & \vline & 0 & \vline & 0 & \vline & 0\\
$4$ & \vline & 0 & \vline & 0 & \vline & 0 & \vline & 0 & \vline & 0\\
$3$ & \vline & 0 & \vline & 0 & \vline & 0 & \vline & 0 & \vline &
0
\end{tabular}
\vspace{.5cm} \caption{The initial terms in the spectral sequence
  \eqref{curtisgen2} for $n=4$}
\label{tab:20}
\end{table}
}
\bigskip

 Observe in particular that  the torsion-free expression
 $ \Gamma_2(\Z)$ which
appears in column 2 of
 table \ref{tab:20}  for $A= \Z$ survives to $E^{\infty}_{2,6}$
since  a non-trivial morphism  $d^2_{2,6}:
\Gamma_2(\Z) \la \Z \ot \Z/2$ would  contradict the
known non-trivial value of $\pi_6(S^4)$. We can therefore recognize in
a generator
of  this group  $ \Gamma_2(\Z)$
 the class  of the generalized Hopf
 fibration $\nu:S^7 \la S^4$ .

\bigskip

The suspension homomorphisms $\pi_4M(A,2)\to \pi_5M(A,3)\to
\pi_6M(A,4)$ can be described in terms of the
 suspension homomorphisms between
the corresponding derived functors of Lie functors. The result  is
expressed by the following  commutative diagram (see also
\cite{Bau} VIII \S 3, IX \S 2, XI \S 1):
$$
\xymatrix@R=11pt{\EuScript L_s^3(A)\oplus \Gamma_2(A)\otimes \mathbb
Z/2\ar@{^{(}->}[r] \ar@{->}[d] & \pi_4M(A,2) \ar@{->>}[r] \ar@{->}[d]& R_2(A)\ar@{->}[d]\\
A\otimes \mathbb Z/2 \ar@{^{(}->}[r] \ar@{->}[d]^(.43){\wr}&
\pi_5M(A,3) \ar@{->}[d] \ar@{->>}[r] & \lambda^2(A)
 \ar@{->>}[d]\\ A\otimes \mathbb Z/2
\ar@{^{(}->}[r] & \pi_6M(A,4) \ar@{->}[r] & \Tor(A,\mathbb Z/2)}.
$$
Note that  these horizontal short exact sequences are in general not
split, since  $\pi_4(M(\Z/2,2)) = \pi_5(M(\Z/2,3)) =  \Z/4$.

\subsection{Solving the extension problem}
In  simple cases one can solve the extension problems with the
help of functoriality. For example we have just seen that there is a
 natural exact sequence
$$
0\to A\otimes \mathbb Z/2\to \pi_6M(A,4)\to \Tor(A,\mathbb Z/2)\to
0.
$$
 In particular, for $A= \Z/4$, this   reduces to the
sequence
$$
0\to \mathbb Z/2\to \pi_6M(\mathbb Z/4,4)\to \mathbb Z/2\to 0.
$$
In order to compute the group $\pi_6M(\mathbb Z/4,4)$, we must still determine whether this sequence
is split. The following simple argument will show that this indeed is the
case.

\bigskip

 Let  $F$ be any endofunctor on the category of abelian groups which
 is
 endowed with a natural extension of functors
\begin{equation}\label{bv12}
0\to A\otimes \mathbb Z/2\to F(A)\to \Tor(A, \mathbb Z/2)\to 0.
\end{equation}
We will now prove that for any such functor $F$ the extension
\eqref{bv12} is split. Suppose  on the contrary that $F(\mathbb Z/4)=\mathbb
Z/4$. In that case the group  $F(\mathbb Z/2)$ is isomorphic  either  to
$\mathbb Z/4$ or to $\mathbb
Z/2\oplus \mathbb Z/2$. Let us first suppose that  $F(\mathbb
Z/2)=\mathbb Z/4$ and  consider the
following diagram,
 induced by the
injection  of $\mathbb Z/2$ into $\mathbb Z/4$:
$$
\xyma{\mathbb Z/2\ar@{^{(}->}[r]\ar@{->}[d]^0 & \mathbb
Z/4\ar@{->}[d] \ar@{->>}[r] & \mathbb Z/2\ar@{=}[d]\\ \mathbb
Z/2\ar@{^{(}->}[r] & \mathbb Z/4\ar@{->>}[r] & \mathbb Z/2}
$$
Such a commutative  diagram cannot exist since the pushout of any
extension by the trivial homomorphism is a trivial extension. If on
the other hand  we suppose that $F(\mathbb Z/2)=\mathbb Z/2\oplus
\mathbb Z/2$, then  the natural
projection $\mathbb Z/4\to \mathbb Z/2$ induces a
commutative diagram
$$
\xyma{\mathbb Z/2\ar@{^{(}->}[r]\ar@{=}[d] & \mathbb Z/4\ar@{->}[d] \ar@{->>}[r] & \mathbb Z/2\ar@{->}[d]^0\\
\mathbb Z/2\ar@{^{(}->}[r] & \mathbb Z/2\oplus \mathbb
Z/2\ar@{->>}[r] & \mathbb Z/2}
$$
which also cannot exist, since the pullback of any extension by the
trivial homomorphism is  a trivial extension. This proves  that
the extension \eqref{bv12} is split,
and in particular that  $\pi_6M(\mathbb Z/4,4)=\mathbb Z/2\oplus \mathbb
Z/2$.
\subsection{Some homotopy groups of $M(\Z/3,5)$.} In this simple example, we will illustrate
some  lines of reasoning by which we computed certain differentials
 in  Curtis spectral sequences. Consider such a
 spectral sequence \eqref{curtisgen} for $n=5$, with abutment
$M(\mathbb Z/3,5)$ and  initial terms
$E_{p,q}^1=L_q\EuScript L^p(\mathbb Z/3,4)$
 One finds  in low degree:
\begin{table}[H]
\begin{tabular}{ccccccccccccccc}
 $q$ & \vline & $E_{1,q}^1$ & $E_{2,q}^1$ & $E_{3,q}^1$ & $E_{4,q}^1$ & $E_{5,q}^1$ & $E_{6,q}^1$ & $E_{7,q}^1$ & $E_{8,q}^1$ & $E_{9,q}^1$\\
 \hline
10 & \vline & 0 & 0 & 0 & 0 & 0 & 0 & 0 & 0 & $\mathbb Z/3$\\
9 & \vline & 0 & $\mathbb Z/3$ & 0 & 0 & 0 & 0 & 0 & 0 & 0\\
8 & \vline & 0 & 0 & $\mathbb Z/3$ & 0 & 0 & 0 & 0 & 0 & 0\\
7 & \vline & 0 & 0 & $\mathbb Z/3$ & 0 & 0 & 0 & 0 & 0 & 0\\
6 & \vline & 0 & 0 & 0 & 0 & 0 & 0 & 0 & 0 & 0\\
5 & \vline & 0 & 0 & 0 & 0 & 0 & 0 & 0 & 0 & 0\\
4 & \vline & $\mathbb Z/3$ & 0 & 0 & 0 & 0 & 0 & 0 & 0 & 0\\
\end{tabular}
\vspace{.5cm} \caption{The initial terms in the spectral sequence
  \eqref{curtisgen2} for  $n=5$ and $A = \Z/3$}
\label{tab:19}
\end{table}

We will now   provide two separate  justifications  for  the
triviality of the  differential
$d_{2,9}^1: \mathbb Z/3\to \mathbb Z/3$, both
of which were used  in more complex situations
 in the previous paragraphs. The first
argument goes as follows. The differential $d^1_{2,9}$ for $n=5$ and $A= \Z/3$  lives in
the following commutative diagram, in which  the notation is the same as in diagram   \ref{suspdiag} (and the vertical arrows
are suspension maps):
$$
\xymatrix@C=60pt{L_9\EuScript L^2(\mathbb Z/3,4) \ar@{->}[r]^{d_{2,9}^1(\Z/3,5)}
\ar@{->}[d]^\Sigma &
L_8\EuScript L^3(\mathbb Z/3,4) \ar@{->}[d]^{\Sigma}\\
L_{10}\EuScript L^2(\mathbb Z/3,5) \ar@{->}[r]_{d_{2,10}^1(\Z/3,6)} &
L_9\EuScript L^3(\mathbb Z/3,5)}
$$
This commutative square is actually  of the form:
$$
\xymatrix@C=60pt{\mathbb Z/3 \ar@{->}[r]^{d_{2,9}^1(\Z/3,5)} \ar@{->}[d] &
\mathbb Z/3 \ar@{->}[d]^{\wr}\\
0 \ar@{->}[r]_{d_{2,10}^1(\Z/3,6)} & \mathbb Z/3 }
$$
 so  that the map $d_{2,9}^1(\Z/3,5)$ is trivial. As a consequence,
 $\pi_9M(\mathbb
Z/3,5)=\pi_{10}M(\mathbb Z/3,5)=\mathbb Z/3$.

\bigskip

Here is the second proof of this assertion. Consider
the natural map $M(\mathbb Z/3,5)\to K(\mathbb Z/3,5)$
\eqref{natmap} and the corresponding map between the  spectral
sequences    \eqref{curtisgen2}  and \eqref{nz11} for  $n=5$ and
$A = \Z/3$.  The homology groups of $K(\mathbb Z/3,5)$ are  given
by:
\begin{table}[H]
\begin{tabular}{cccccccccc}
 $n$ & \vline & 5 & 6 & 7 & 8 & 9 & 10 & 11\\ \hline &\vline\\  $H_nK(\mathbb Z/3,5)$ & \vline &
 $\mathbb Z/3$ & 0 & 0 & 0 & $\mathbb
 Z/3$ & $\mathbb Z/3$ & $\mathbb Z/3$
\end{tabular}
\caption{}
\end{table}

The initial terms of the spectral sequence  \eqref{nz11} for
$n=5$ and $A = \Z/3$     are the following:
\begin{table}[H]
\begin{tabular}{ccccccccccccccc}
 $q$ & \vline & $E_{1,q}^1$ & $E_{2,q}^1$ & $E_{3,q}^1$ & $E_{4,q}^1$ & $E_{5,q}^1$ & $E_{6,q}^1$ & $E_{7,q}^1$ & $E_{8,q}^1$ & $E_{9,q}^1$\\
 \hline
10 & \vline & $\mathbb Z/3$ & 0 & 0 & 0 & 0 & 0 & 0 & 0 & $(\mathbb Z/3)$\\
9 & \vline & $\mathbb Z/3$ & $(\mathbb Z/3)$ & 0 & 0 & 0 & 0 & 0 & 0 & 0\\
8 & \vline & $\mathbb Z/3$ & 0 & $(\mathbb Z/3)$ & 0 & 0 & 0 & 0 & 0 & 0\\
7 & \vline & 0 & 0 & $(\mathbb Z/3)$ & 0 & 0 & 0 & 0 & 0 & 0\\
6 & \vline & 0 & 0 & 0 & 0 & 0 & 0 & 0 & 0 & 0\\
5 & \vline & 0 & 0 & 0 & 0 & 0 & 0 & 0 & 0 & 0\\
4 & \vline & $(\mathbb Z/3)$ & 0 & 0 & 0 & 0 & 0 & 0 & 0 & 0\\
\end{tabular}
\vspace{.5cm} \caption{The initial terms in the  spectral
sequence\eqref{nz0}  for $A =  \Z/3\:$ and $n=5$}
\end{table}

We displayed within brackets those terms  which are in the image
of elements from the corresponding  spectral sequence
\eqref{curtisgen2}, as given in table  \ref{tab:19}. Since the spectral sequence \eqref{nz0}
converges here  to the graded group $\mathbb Z/3[4]$, it follows
that the map $d_{2,9}^1$ is necessarily zero: otherwise, the
element $E_{1,10}^1=\mathbb Z/3$ would contribute non-trivially
to $\pi_{10}(K(\Z/3,4))$. It follows that the corresponding map
 $d^1_{2,9}$ in the spectral sequence whose initial terms are displayed in table \ref{tab:19}  is also
 trivial.  We deduce from this   that
$\pi_9M(\mathbb Z/3,5)=\pi_{10}M(\mathbb Z/3,5)=\mathbb Z/3$.

{\appendix
\section{Derived Koszul complex}
\label{app:derKcx}

In this appendix, we illustrate our derived functor methods, by
giving an explicit description of  certain objects and morphisms
obtained by deriving the  Koszul  sequence \eqref{koszul2}.

\bigskip

Let $$ 0\to L\buildrel{\delta}\over\to M\to A\to 0
$$
be a flat resolution of the abelian group $A$. A convenient  model
for the derived category object $L\Lambda^n(A)$ is provided by the
dual Koszul complex of the morphism $L\buildrel{\delta}\over\to
M$. Recall that for $n=2$ this is the complex
\begin{equation}\label{l2}
\Gamma_2(L)\buildrel{\delta_2}\over\to L\otimes
M\buildrel{\delta_1}\over\to \Lambda^2(M)
\end{equation}
with the differentials
\begin{align*}
& \delta_2(\gamma_2(l))=l\otimes \delta(l)\\
& \delta_1(l\otimes m)=\delta(l)\wedge m
\end{align*}
and for $n=3$ the complex $$
\Gamma_3(L)\buildrel{\delta_3}\over\to \Gamma_2(L)\otimes
M\buildrel{\delta_2}\over\to L\otimes
\Lambda^2(M)\buildrel{\delta_1}\over\to \Lambda^3(M)
$$
with the differentials
\begin{align*}
& \delta_3(\gamma_3(l))=\gamma_2(l)\otimes \delta(l)\\
& \delta_3(\gamma_2(l)l')=ll'\otimes \delta(l)+\gamma_2(l)\otimes
\delta(l')\\
& \delta_2(\gamma_2(l)\otimes m)=l\otimes m\wedge \delta(l)\\
& \delta_1(l\otimes m\wedge m')=\delta(l)\wedge m\wedge m'
\end{align*}
The derived category object $L\Lambda^2(A)\buildrel{L}\over\otimes
A$ may be represented by the tensor product of the complex
(\ref{l2}) with the complex $L\buildrel{\delta}\over\to M$, in
other words by the total complex associated to the bicomplex
$$
\xyma{ \Gamma_2(L)\otimes L \ar@{->}[r] \ar@{->}[d] & L\otimes
M\otimes L \ar@{->}[r]
\ar@{->}[d] & \Lambda^2(M)\otimes L \ar@{->}[d]\\
\Gamma_2(L)\otimes M \ar@{->}[r] & L\otimes M\otimes M\ar@{->}[r]
& \Lambda^2(M)\otimes M}
$$
in other words  the complex
\begin{multline}\label{longco}
\Gamma_2(L)\otimes L\buildrel{\delta_3'}\over\to\Gamma_2(L)\otimes
M\oplus (L\otimes M\otimes L)\buildrel{\delta_2'}\over\to
(L\otimes M\otimes M)\oplus \Lambda^2(M)\otimes
L\buildrel{\delta_1'}\over\to \Lambda^2(M)\otimes M
\end{multline}
with differentials
\begin{align*}
& \delta_3'(\gamma_2(l)\otimes l')=(\gamma_2(l)\otimes \delta(l'),
-l\otimes \delta(l)\otimes l')\\
& \delta_2'(\gamma_2(l)\otimes m)=(l\otimes \delta(l)\otimes m,0)\\
& \delta_2'(l\otimes m\otimes l')=(l\otimes m\otimes
\delta(l'),m\wedge \delta(l)\otimes l')\\
& \delta_1'(l\otimes m\otimes m')=\delta(l)\wedge m\otimes m'\\
& \delta_1'(m\wedge m'\otimes l)=m\wedge m'\otimes \delta(l)
\end{align*}

Recall that \begin{align} & L_1\Lambda^2(A)=\Omega_2(A)\label{n1}\\
& \pi_2\left(\Lambda^2(A)\buildrel{L}\over\otimes A
\right)=\Tor(\Omega_2(A),A).\label{n2}
\end{align}
Given elements $a,a'\in\ _nA$, let us choose its representatives
$m,m'\in M$ and cross-cap elements $l,l'\in L$ such that
$\delta(l)=nm,\ \delta(l')=nm'.$ The maps
\begin{align*}
& \Omega_2(A)\to (L\otimes M)/\text{im}(\delta_2)\\
& \Tor(\Omega_2(A),A)\to (\Gamma_2(L)\otimes M\oplus (L\otimes
M\otimes L))/\text{im}(\delta_3')
 \end{align*}
which define the isomorphisms (\ref{n1}) and (\ref{n2}) are given
by
\begin{align*}
& w_2(a)\mapsto l\otimes m+\text{im}(\delta_2)\\ & w_2(a)*b\mapsto
(-\gamma_2(l)\otimes m',l\otimes m\otimes l')+\text{im}(\delta_3)
\end{align*}

Next we consider the following diagram with exact rows and
columns:
\begin{equation}\label{1dia}
\xymatrix@R=13pt{\EuScript L^3(L) \ar@{^{(}->}[r]^{\delta_3''}
\ar@{^{(}->}[d] & L\otimes M\otimes L \ar@{->}[r]^{\delta_2''}
\ar@{^{(}->}[d] &
L\otimes M\otimes M\ar@{->}[r]^{\delta_1''} \ar@{^{(}->}[d] & Y^3(M)\ar@{^{(}->}[d]\\
\Gamma_2(L)\otimes L
\ar@{->}[d]_{\phi_3}\ar@{^{(}->}[r]^{\delta_3'\ \ \ \ \ \ \ \ \ \
\ \ } & \Gamma_2(L)\otimes M\oplus (L\otimes M\otimes L)
\ar@{->>}[d]_{\phi_2}\ar@{->}[r]^{\delta_2'} & (L\otimes M\otimes
M)\oplus \Lambda^2(M)\otimes L \ar@{->>}[d]_{\phi_1}
\ar@{->}[r]^{\ \ \ \ \ \  \ \ \ \delta_1'} & \Lambda^2(M)\otimes M \ar@{->>}[d]_{\phi_0}\\
\Gamma_3(L) \ar@{^{(}->}[r]^{\delta_3} \ar@{->>}[d] &
\Gamma_2(L)\otimes M \ar@{->}[r]^{\delta_2} & L\otimes
\Lambda^2(M) \ar@{->}[r]^{\delta_1}  & \Lambda^3(M)\\ L\otimes
\mathbb Z/3 }
\end{equation}
with
\begin{align*}
& \phi_0(m\wedge m'\otimes m'')=m\wedge m'\wedge m''\\
& \phi_1(l\otimes m\otimes m')=l\otimes m\wedge m'\\
& \phi_1(m\wedge m'\otimes l)=l\otimes m\wedge m'\\
& \phi_2(\gamma_2(l)\otimes m)=-\gamma_2(l)\otimes m\\
& \phi_2(l\otimes m\otimes l')=ll'\otimes m\\
& \phi_3(\gamma_2(l)\otimes l')=-\gamma_2(l)l'.
\end{align*}
and
\begin{align*}
& \delta_3''(\overline{l\otimes l'\wedge l''})=l\otimes
\delta(l'')\otimes l'+l''\otimes \delta(l)\otimes l'-l\otimes
\delta(l')\otimes l''-l'\otimes \delta(l)\otimes l''\\ &
\delta_2''(l\otimes m\otimes l')=l\otimes \delta(l')\otimes
m+l'\otimes \delta(l)\otimes m+l\otimes m\otimes \delta(l')\\
& \delta_1''(l\otimes m\otimes m')=\{\delta(l),m',m\}
\end{align*}
Here $\overline{l\otimes l'\wedge l''}$ denotes the image of the
element $l\otimes l'\wedge l''$ under the natural epimorphism
$L\otimes \Lambda^2(L)\to \EuScript L^3(L)$. We will now make use
of  the fact that the dual de Rham complex
$$
0\to \Lambda^3(L)\to L\otimes \Lambda^2(L)\to \Gamma_2(L)\otimes
L\to \Gamma_3(L)
$$
has trivial homology in positive dimensions and hence
$$
\EuScript L^3(L)=\ker\{\Gamma_2(L)\otimes L\to
\Gamma_3(L)\}=\text{coker}\{\Lambda^3(L)\to L\otimes
\Lambda^2(L)\}
$$

Consider the functor $\bar E^3(A):=\text{im}\{\Gamma_2(A)\otimes
A\to \Gamma_3(A)\}.$ We have a natural short exact sequence
$$
0\to \bar E^3(A)\to \Gamma_3(A)\to A\otimes \mathbb Z/3\to 0
$$
which induces the following commutative diagram
\begin{equation}\label{2dia}
\xymatrix@R=13pt{ \bar E^3(L) \ar@{^{(}->}[d]
\ar@{->}[r]^{\bar\delta_3} & \Gamma_2(L)\otimes M \ar@{=}[d]
\ar@{->}[r]^{\delta_2} & L\otimes
\Lambda^2(M) \ar@{->}[r]^{\delta_1} \ar@{=}[d] & \Lambda^3(M)\ar@{=}[d]\\
\Gamma_3(L) \ar@{->>}[d] \ar@{->}[r]^{\delta_3} &
\Gamma_2(L)\otimes M \ar@{->}[r]^{\delta_2} & L\otimes
\Lambda^2(M)
\ar@{->}[r]^{\delta_1}  & \Lambda^3(M)\\
L\otimes \mathbb Z/3}
\end{equation}
From diagrams (\ref{1dia}) and (\ref{2dia}) we deduce the
following diagram with exact  arrows and columns:
\begin{equation}\label{conne1}
\xymatrix@R=13pt{ & H_2W \ar@{^{(}->}[r] \ar@{^{(}->}[d] &
L_2Y^3(A)
\ar@{^{(}->}[d]\\
& \pi_2\left(L\Lambda^2(A)\buildrel{L}\over\otimes A\right)
\ar@{=}[r] \ar@{->}[d] &
\pi_2\left(L\Lambda^2(A)\buildrel{L}\over\otimes A\right)
\ar@{->}[d]\\
L\otimes \mathbb Z/3 \ar@{->}[r] & H_2Q \ar@{->>}[r] \ar@{->}[d] & L_2\Lambda^3(A)\ar@{->}[d]\\
& H_1W \ar@{->}[r] \ar@{->}[d] & L_1Y^3(A)\ar@{->}[d]\\
& \pi_1\left(L\Lambda^2(A)\buildrel{L}\over\otimes A\right)
\ar@{=}[r] \ar@{->>}[d] &
\pi_1\left(L\Lambda^2(A)\buildrel{L}\over\otimes
A\right)\ar@{->>}[d]\\
& L_1\Lambda^3(A) \ar@{=}[r] & L_1\Lambda^3(A)}
\end{equation}
where $W$ and $Q$ are the upper rows in diagrams (\ref{1dia}) and
(\ref{2dia}) respectively. We give the following  simple example,
which illustrates the inner life of the previous
diagrams.\\

\begin{prop}
\label{ly3z3}
$$
L_iY^3(\mathbb Z/3)=\begin{cases} \mathbb Z/3,\ i=2\\
\mathbb Z/9,\ i=1,\\ 0,\ i\neq 1,2\end{cases}
$$
\end{prop}
\begin{proof}
It follows from the description $Y^3(A)=\ker\{\Lambda^2(A)\otimes
A \to \Lambda^3(A)\}$ that
$$
L_2Y^3(A)=\ker\{\Tor(\Omega_2(A),A)\to L_2\Lambda^3(A)\}
$$
Let $A=\mathbb Z/3$ and $L\buildrel{\partial}\over\to M$ is
$\mathbb Z\buildrel{3}\over\to \mathbb Z.$ Let $l,m$ be generators
of $L$ and $M$ respectively with $\delta(l)=3m$. The group
$\Tor(\Omega_2(A),A)$ is generated by the homology class of the
element $(\gamma_2(l)\otimes m, l\otimes m\otimes l)$ in the
complex (\ref{longco}). We have
$$
\phi_2(\gamma_2(l)\otimes m, l\otimes m\otimes
l)=\gamma_2(l)\otimes m+ll\otimes m=3\gamma_2(l)\otimes
m=\gamma_2(l)\otimes \delta(l)=\delta_3(\gamma_3(l))
$$
Hence the map
$$
\Tor(\Omega_2(A),A)\to L_2\Lambda^3(A)
$$
induced by the map $\phi_2$ is the zero map. This proves that
$$
L_2Y^3(\mathbb Z/3)=\mathbb Z/3.
$$
 It is easy to see that for our
choice of $L$ and $M$, the complex $W$   has the form $\mathbb
Z\buildrel{9}\over \to \mathbb Z \to 0$ so that  $H_1W=\mathbb
Z/9,\ H_2W=0.$ In this case, the diagram (\ref{conne1}) has the
form
$$
\xymatrix@R=13pt{ & 0 \ar@{->}[r] \ar@{->}[d] & \mathbb Z/3
\ar@{->}[d]^\simeq\\
& \mathbb Z/3\ar@{=}[r] \ar@{->}[d] & \mathbb Z/3 \ar@{->}[d]^0 \\
\mathbb Z/3 \ar@{^{(}->}[r] & H_2Q \ar@{->>}[r] \ar@{->}[d] & \mathbb Z/3\ar@{->}[d]\\
& \mathbb Z/9 \ar@{->}[r] \ar@{->}[d] & L_1Y^3(A)\ar@{->}[d]\\
& \mathbb Z/3 \ar@{=}[r] \ar@{->}[d] & \mathbb Z/3
\ar@{->}[d]\\
& 0 & 0 }
$$
and we see that the map $H_1W\to L_1Y^3(A)$ is an isomorphism and
hence
$$
L_1 Y^3(\mathbb Z/3)=\mathbb Z/9.
$$
\end{proof}

The object $L\Gamma_2(A)$ of the derived category may be
represented as the following complex:
$$
L\otimes L\to \Gamma_2(L)\otimes (M\otimes L)\to \Gamma_2(M)
$$
Consider the following diagram:
\begin{equation}\label{cncn1}
\xymatrix@R=13pt{\Lambda^2(L) \ar@{^{(}->}[d] \ar@{->}[r] &
M\otimes L \ar@{->}[r] \ar@{^{(}->}[d]^{(0,id)} &
\Gamma_2(M) \ar@{=}[d]\\
L\otimes L\ar@{->}[r] \ar@{->>}[d] & \Gamma_2(L)\oplus (M\otimes
L)\ar@{->>}[d] \ar@{->}[r]
& \Gamma_2(M)\\
SP^2(L) \ar@{->}[r] & \Gamma_2(L)} \end{equation} Denote by
$C=C(L\buildrel{\delta}\over\to M)$ the upper complex in
(\ref{cncn1}). Diagram (\ref{cncn1}) implies the following exact
sequence of homology groups:
$$
0\to H_1C\to R_2(A)\to L\otimes \mathbb Z/2\to H_0C\to
\Gamma_2(A)\to 0
$$
In particular, the $p$-torsion components of $H_iC$ and
$L_i\Gamma_2$ are naturally isomorphic for $p\neq 2$.

The objects $L\Gamma_3(A)$ and $A\buildrel{L}\over\otimes
L\Gamma_2(A)$ of the derived category may be represented by the
following complexes:
\begin{multline} L\otimes L\otimes L \to(\Gamma_2(L)\otimes L)\oplus (L\otimes
\Gamma_2(L))\oplus (L\otimes L\otimes M)\to\\ \Gamma_3(L)\oplus
(\Gamma_2(L)\otimes M)\oplus (L\oplus \Gamma_2(M))\to\Gamma_3(M)
\end{multline}

\begin{multline}
L\otimes L\otimes L\to (M\otimes L\otimes L)\oplus (L\otimes
\Gamma_2(L))\oplus (L\otimes L\otimes M)\to \\ (L\otimes
\Gamma_2(M))\oplus (M\otimes \Gamma_2(L))\oplus (M\otimes L\otimes
M)\to M\otimes \Gamma_2(M)
\end{multline}

Consider the following diagram:
\begin{equation}\label{mz1}
\xymatrix@R=13pt{\Lambda^3(L) \ar@{->}[r] \ar@{=}[d] & M\otimes
\Lambda^2(L) \ar@{->}[r] \ar@{=}[d] & SP^2(M)\otimes L\ar@{->}[r]
\ar@{^{(}->}[d]
 & SP^3(M)
\ar@{^{(}->}[d]
\\
\Lambda^3(L) \ar@{->}[r] & M\otimes \Lambda^2(L) \ar@{->}[r] &
\Gamma_2(M)\otimes L\ar@{->}[r] \ar@{->>}[d]& \Gamma_3(M)\ar@{->>}[d]\\
& & M\otimes L\otimes \mathbb Z/2 \ar@{->}[r] & M\otimes \mathbb
Z/3\oplus (M\otimes M\otimes \mathbb Z/2)}
\end{equation}
The upper complex in (\ref{mz1}) is a model for the element
$LSP^3(A)$ in the derived category. Denote the middle horizontal
complex by $D=D(L\buildrel{\delta}\over\to M)$. We have the
natural isomorphism
$$
H_2D\simeq L_2SP^3(A)
$$
and the following exact sequence:
\begin{multline}
0\to L_1SP^3(A)\to H_1D\to \Tor(M\otimes A,\mathbb Z/2)\to\\
SP^3(A)\to H_0D\to M\otimes \mathbb Z/3\oplus (M\otimes A\otimes
\mathbb Z/2)\to 0
\end{multline}

Now consider the following diagram which extends the diagram
(\ref{1dia}):
\begin{equation}\label{bdi}
 \xymatrix@R=13pt{\Lambda^3(L) \ar@{^{(}->}[r]
\ar@{^{(}->}[d] & M\otimes \Lambda^2(L)
 \ar@{->}[r] \ar@{^{(}->}[d] & \Gamma_2(M)\otimes L \ar@{->}[r] \ar@{^{(}->}[d] & \Gamma_3(M)\ar@{^{(}->}[d]\\
L\otimes \Lambda^2(L) \ar@{^{(}->}[r]^{\bar\delta_3\ \ \ \ \ \ \ \
\ } \ar@{->}[d]^{\phi_3'} & M\otimes \Lambda^2(L)\oplus (L\otimes
M\otimes L)\ar@{->}[r]^{\bar\delta_2} \ar@{->}[d]^{\phi_2'} &
(M\otimes M\otimes L)\oplus L\otimes \Gamma_2(M)\ar@{->}[r]^{\ \ \ \ \ \ \ \ \ \bar\delta_1} \ar@{->}[d]^{\phi_1'} & M\otimes \Gamma_2(M)\ar@{->}[d]^{\phi_0'}\\
\Gamma_2(L)\otimes L
\ar@{->}[d]^{\phi_3}\ar@{^{(}->}[r]^{\delta_3'\ \ \ \ \ \ \ \ \ }
& \Gamma_2(L)\otimes M\oplus (L\otimes M\otimes L)
\ar@{->>}[d]^{\phi_2}\ar@{->}[r]^{\delta_2'} & (L\otimes M\otimes M)\oplus \Lambda^2(M)\otimes L \ar@{->>}[d]^{\phi_1}\ar@{->}[r]^{\ \ \ \ \ \ \ \ \ \delta_1'} & \Lambda^2(M)\otimes M \ar@{->>}[d]^{\phi_0}\\
\Gamma_3(L) \ar@{^{(}->}[r]^{\delta_3} \ar@{->>}[d] &
\Gamma_2(L)\otimes M \ar@{->}[r]^{\delta_2} & L\otimes
\Lambda^2(M) \ar@{->}[r]^{\delta_1}  & \Lambda^3(M)\\ L\otimes
\mathbb Z/3 }
\end{equation} Here
\begin{align*}
& \bar \delta_1(m\otimes m'\otimes l)=m\otimes m'\delta(l)\\
& \bar \delta_1(l\otimes \gamma_2(m))=\delta(l)\otimes
\gamma_2(m)\\
& \bar \delta_2(m\otimes l\wedge l')=(m\otimes \delta(l)\otimes
l'-m\otimes \delta(l')\otimes l,0)\\
& \bar \delta_2(l\otimes m\otimes l')=(\delta(l)\otimes m\otimes
l',-l\otimes m\delta(l))\\
& \bar \delta_3(l\otimes l'\wedge l'')=(\delta(l)\otimes l'\wedge
l'',-l\otimes \delta(l')\otimes l''+l\otimes \delta(l'')\otimes
l')
\end{align*}
and
\begin{align*}
& \phi_0'(m\otimes \gamma_2(m'))=m\wedge m'\otimes m'\\
& \phi_1'(m\otimes m'\otimes l)=(-l\otimes m\otimes m', m\wedge
m'\otimes l)\\
& \phi_1'(l\otimes \gamma_2(m))=(l\otimes m\otimes m,0)\\
& \phi_2'(l\otimes m\otimes l')=(-ll'\otimes m,-l\otimes m\otimes
l')\\
& \phi_2'(m\otimes l\wedge l')=(0,l\otimes m\otimes l'-l'\otimes
m\otimes l)\\
& \phi_3'(l\otimes l'\wedge l'')=-ll'\otimes l''+ll''\otimes l'
\end{align*}

We obtain the natural isomorphism of complexes:
$$
\xymatrix@R=13pt{H_2D \ar@{->}[r] \ar@{=}[d] &
H_2\left(A\buildrel{L}\over\otimes C \right) \ar@{=}[d]
\ar@{->}[r] \ar@{=}[d] & \Tor(\Omega_2(A),A)\ar@{->}[r] \ar@{=}[d]
&
\Omega_3(A) \ar@{=}[d]\\
L_2SP^3(A) \ar@{->}[r] & \Tor(A,L_1SP^2(A))\ar@{->}[r] &
\tor(\Omega_2(A),A) \ar@{->}[r] & \Omega_3(A) }
$$
The map
$$
\tor(A,L_1SP^2(A))\to (M\otimes \Lambda^2(L)\oplus (L\otimes
M\otimes L))/\text{im}(\bar\delta_3)
$$
is given as follows: let $a,a',a''\in A$ with $na=na'=na''=0$ are
represented by elements $m,m',m''$ is $M$ and
$\delta(l)=nm,\delta(l')=nm',\delta(l'')=nm'',$ then
$$
a*a'\hat *a'' \mapsto (-m\otimes l'\wedge l'', l\otimes m'\otimes
l''-l\otimes m''\otimes l')+\text{im}(\bar\delta_3)
$$
The map $\phi_2'$ induces the Koszul-type map
$$
\tor(A,L_1SP^2(A))\to \tor(\Omega_2(A),A)
$$
 defined by
$$
a*a'\hat * a''\mapsto
(w(a+a'')-w(a)-w(a''))*a'-(w(a+a')-w(a)-w(a'))*a''.
$$

\noindent{\bf Example.} In the case $(L\buildrel{\delta}\over\to
M)=(\mathbb Z\buildrel{n}\over\to \mathbb Z),$ the diagram
(\ref{bdi}) has the following form:
$$
\xymatrix@R=18pt{& & \mathbb Z \ar@{->}[r]^{3n} \ar@{->}[d]^(.45){(1,1)} & \mathbb Z\ar@{->}[d]^{1}\\
& \mathbb Z \ar@{->}[d]^{(-2,-1)} \ar@{->}[r]^{(n,-2n)} &
\mathbb Z\oplus \mathbb Z \ar@{->}[r]_(.55){(2n,n)} \ar@{->}[d]^(.58){(-1,1)} & \mathbb Z\\
\mathbb Z \ar@{->}[d]_{-3} \ar@{->}[r]^{(n,-n)} & \mathbb Z\oplus
\mathbb Z
\ar@{->}[d]^{(-1,2)}\ar@{->}[r]^{(n,n)} & \mathbb Z\\
\mathbb Z \ar@{->}[r]^n & \mathbb Z }
$$
This diagram implies that the map
$$
H_1\left(\mathbb Z/n\buildrel{L}\over\otimes C(\mathbb
Z\buildrel{n}\over\to \mathbb Z)\right)\to
\pi_1\left(L\Lambda^2(\mathbb Z/n)\buildrel{L}\over \otimes
\mathbb Z/n\right)
$$
is multiplication by 3 in the group $\mathbb Z/n$.

\bigskip

When $A=\mathbb Z/2$ and $C=C(\mathbb Z\buildrel{2}\over\to
\mathbb Z)$, we have the following commutative diagram
$$
\xymatrix@R=13pt{A\otimes H_1C\ar@{->}[r]
\ar@{^{(}->}[d]
 & A\otimes R_2(A)
\ar@{^{(}->}[d]
\\
H_1\left(A\buildrel{L}\over\otimes C\right) \ar@{->}[r]
\ar@{->>}[d] & \pi_1\left(A\buildrel{L}\over\otimes
L\Gamma_2(A)\right)\ar@{->>}[d] \ar@{->}[r] &
\pi_1\left(L\Lambda^2(A)\buildrel{L}\over\otimes A\right)\\
\Tor(A,H_0C) \ar@/_55pt/[rru]_{\simeq} \ar@{->}[r]^{\simeq} &
\Tor(A,\Gamma_2(A)) }
$$
\vspace{.5cm}

As a corollary, we find  that:
\begin{prop}
The derived Koszul map
$$
\pi_1\left(A\buildrel{L}\over\otimes L\Gamma_2(A)\right)\to
\pi_1\left(L\Lambda^2(A)\buildrel{L}\over\otimes A\right)
$$
is the zero map for $A=\mathbb Z/3$ and an epimorphism for
$A=\mathbb Z/p$, where $p$ is a prime $\neq 3$.
\end{prop}

}

\vspace{1cm}


\begin{thebibliography}{XXX}
\bibitem{a-b-w} K. Akin, D.A. Buchsbaum and J. Weyman, Schur functors
  and Schur complexes, Adv.Math. {\bf 44} (1982), 207--278.

\bibitem{Bau} H.-J. Baues,  Homotopy type and homology,  {\it Oxford Science
Publications}, Oxford, (1996).


\bibitem{BB} H.-J. Baues and J. Buth: On the group of homotopy equivalences of simply connected five
manifolds, {\it Math. Z.} {\bf 222} (1996), 573--614.


\bibitem{BG} H.-J. Baues and P. Goerss:  A homotopy operation spectral sequence for the computation of homotopy
groups, {\it Topology} {\bf 39} (2000), 161--192.

\bibitem{BauesPirashvili} H.-J. Baues and T. Pirashvili:  A
    universal coefficient theorem for quadratic functors, {\it J. Pure
      Appl. Alg. } {\bf 148} (2000).


\bibitem{B-S} D. Blanc  and C. Stover, A generalized  Grothendieck spectral
    sequence. In Adams Memorial Symposium on Algebraic Topology,
  {\it London Math. Soc. Lecture Notes Ser.}  {\bf 175}, CUP (1992), 145--161.

\bibitem{Bou1} A. K. Bousfield: Homogeneous functors and their
derived functors, preprint.

\bibitem{Bou} A. K. Bousfield: Operations on derived functors of
non-additive functors, preprint.

\bibitem{6A} A. K. Bousfield, E. B. Curtis, D. M. Kan, D. G. Quillen,
  D. L. Rector and  J. W. Schlesinger: The mod-$p$ lower central series
  and the Adams spectral sequence, {\it Topology} {\bf 5} (1966), 331--342.

\bibitem{Breen} L. Breen: On the functorial homology of abelian
groups, {\it J. Pure Appl. Alg.} {\bf 142} (1999), 199--237.


\bibitem{Buth} J. Buth: Einfach zusammenh\"angende Poincar\'e-Komplexe
der Dimension 6, {\it Bonner Math. Schriften} {\bf 272} (1994),
100 pages.

\bibitem{Cartan} H. Cartan: Alg\`ebres d'Eilenberg-Mac Lane et
  homotopie {\it S\'eminaire Cartan} {\bf 7} (1954/55), Secr\'etariat
  Mat\'ematique.


\bibitem{CohenNeisendorfer} F.R.  Cohen, J. C.  Moore and J. A.
Neisendorfer: Torsion in homotopy groups, {\it
  Annals of Mathematics} {\bf 109} (1979), 121--168.

 \bibitem{Cohen-Wu} F. R.  Cohen and Jie  Wu, A remark on the homotopy groups of $\Sigma^nRP^2$, Contemporary Math. {\bf 181} (1995), 65-81.

\bibitem{Curtis:63} E. B.  Curtis: Lower central series of semi-simplicial
complexes, {\it Topology} {\bf 2} (1963), 159--171.

\bibitem{Curtis:65} E. B.  Curtis: Some relations between homotopy and homology.
{\it Ann. of Math.} {\bf 82} (1965), 386--413.

\bibitem{Curtis:71} E. B. Curtis: Simplicial homotopy theory, {\it
    Advances in Math.} {\bf 6} (1971), 107--209.

\bibitem{Decker} G. J.  Decker, University of Chicago Ph.D.  thesis (1974),
  available at:

 http://www.maths.abdn.ac.uk/~bensondj/html/archive/decker.html

\bibitem{Dennett} J. Dennett: A functorial description of
$\pi_qL^2K(\Pi,n)$, {\it Quart. J. Math. Oxford} {\bf 20}, (1969),
59--64.

\bibitem{Dold} A. Dold: Homology of Symmetric Products and Other
  Functors of Complexes, {\it  Annals of Math.}. {\bf 68} (1958), 54--80.

\bibitem{DoldPuppe} A. Dold and D. Puppe: Homologie nicht-additiver
Funtoren; Anwendugen. {\it Ann. Inst. Fourier} {\bf 11} (1961)
201--312.

\bibitem{Dold1} A. Dold: Lectures on Algebraic Topology, Grundlehren
  der Math. Wissenchaften {\bf 200}, Springer-Verlag (1980).

\bibitem{Dre} W. Dreckmann: Distributivgesetze in der
Homotopietheorie, Dissertation Bonn, (1992).

\bibitem{EM} S. Eilenberg and S. Mac Lane: On the groups
$H(\pi,n),$ II: Methods of computation,\ {\it Annals of  Math.} {\bf
60}, (1954), 49--139.

\bibitem{ful} W. Fulton, Young Tableaux, {\it London Mathematical Society
  Student Texts}  {\bf  35}, Cambridge University Press (1997).

\bibitem{ful-har} W. Fulton and J. Harris: Representation theory, A
  first course, {\it Graduate Texts in Mathematics} {\bf  129}, Springer-Verlag (1991).

\bibitem{FLS} V. Franjou, J. Lannes, L. Schwartz: Autour de la
    cohomologie de Mac Lane des corps finis, {\it Invent. Math} {\bf 89}
  (1987), 247--270.


\bibitem{Franjou} V. Franjou: Cohomologie de de Rham enti\`ere,
  preprint arXiv:math/0404123.

\bibitem{G-Z} P. Gabriel and M. Zisman, Calculus of Fractions and
  Homotopy Theory, {\it Ergebnisse der Math. und ihrer Grenzgebiete},
  N.S. {\bf 35}, Springer-Verlag New York Inc. (1967).



\bibitem{Hamsher} R. M.  Hamsher, University of Chicago Ph. D  thesis
  (1973), available at:

http://www.maths.abdn.ac.uk/~bensondj/html/archive/hamsher.html



\bibitem{illusie} L. Illusie, Complexe Cotangent et D\'eformations
    I, {\it Lecture Notes in Mathematics} {\bf 239}, Springer, Berlin,
    1971.

%\bibitem{Hiltonmilnor} P. J. Hilton, On the homotpy groups of the
 % union of spheres, {\it  J. London. Math. Soc.} {\bf 30}  (1955), 154--172.


\bibitem{Jean} F. Jean: Foncteurs d\'eriv\'es de l'alg\'ebre
sym\'etrique: Application au calcul de certains groupes
d'homologie fonctorielle des espaces $K(B,n)$, Doctoral  thesis,
University of Paris 13, 2002, available at:

http://www.maths.abdn.ac.uk/~bensondj/html/archive/jean.html


\bibitem{Kock} B. K\"ock: Computing the homology of Koszul
complexes, {\it Trans. Amer. Math. Soc.} {\bf 353} (2001),
3115-3147.

\bibitem{lascoux} A. Lascoux,  Syzygies des vari\'et\'es
  d\'eterminantales,  Adv. in Math.
30 (1978), 202--237.

\bibitem{Leibowitz} D. Leibowitz: The $E^1$ term of the lower central
  series spectral sequence for the homotopy of spaces, Brandeis
  University  Ph.D. thesis (1972).

%\bibitem{Loveless} A. Loveless: A congruence for products of
%binomial coefficients modulo a composite, {\it Elec. J. Comb.
%Number Theory}, {\bf 7} (2007)

\bibitem{Mac} S. Mac Lane: Triple torsion products and multiple
K\"unneth formulas, {\it Math. Ann.} {\bf 140} (1960), 51--64.

\bibitem{Mac1} S. Mac Lane: Decker's sharper K\"unneth formula,
{\it Lecture Notes in Mathematics}, {\bf 1348}, (1988), 242--256.

\bibitem{Mag} W. Magnus: \"Uber Beziehungen zwischen h\"oheren
Kommutatoren, {\it J. reine angew. Math.} {\bf 177} (1937),
105-115.

\bibitem{May}  J. P. May, Simplicial objects in Algebraic Topology,
  {\it Van Nostrand Mathematical Studies} {\bf 11} (1967).

\bibitem{MPBook} R. Mikhailov and I.B.S. Passi:  Lower central and
dimension series of groups, {\it  Lecture Notes in Mathematics}, {\bf
1952}, Springer-Verlag, (2008)

\bibitem{roman} R. Mikhailov: On the homology of the dual de Rham
  complex, in preparation.

\bibitem{Neisendorfer} J. Neisendorfer,
 3-primary exponents.
{\it  Math. Proc. Cambridge Philos. Soc.} {\bf  90} (1981), 63--83.


\bibitem{Quillen} D. Quillen: On the (co-)homology of commutative
rings, {\it Proc. Symp. Pure Math.} {\bf 17} (1970), 65--87.


\bibitem{roby}
N. Roby: Lois de polyn\^omes et lois formelles  en th\'eorie
des
  modules, {\it  Annales  Sci. de l'\'Ec. Norm. Sup}, 3\`eme s\'erie, {\bf 80}
(1953), 213--348.

\bibitem{Schlesinger:66} J. W. Schlesinger: The semi-simplicial free Lie
ring, {\it Trans. Amer. Math. Soc.} {\bf 122} (1966), 436-442.


\bibitem{To} H. Toda: Composition methods in homotopy groups of
spheres, {\it Annals of Mathematics Studies} {\bf  49},  Princeton
University Press, Princeton, N.J. (1962)

\bibitem{weibel} C.A. Weibel: An introduction to homological algebra,
  {\it Cambridge
studies in  advanced mathematics} {\bf 38}, Cambridge University
Press  (1994).

\bibitem{Whitehead} J.H.C. Whitehead: A certain exact sequence, {\it
    Annals of Math.} {\bf 52}, (1950), 51--110.

\bibitem{Witt} E. Witt: Treu Darstellung Liescher Ringe, {\it  J. reine angew.
Math.} {\bf 177}, (1937), 152--160.

\bibitem{Wu2}  J. Wu : Homotopy theory of the suspensions of the
  projective plane, {\it Memoirs of the A. M. S.}  {\bf 162}, No. 769, (2003).
\end{thebibliography}
\end{document}